\pdfoutput=1

\documentclass[11pt,a4paper]{memoir}
\usepackage[T1]{fontenc}
\usepackage[utf8]{inputenc}
\usepackage{amsmath}
\usepackage{amsthm}
\usepackage{amsfonts}
\usepackage{amssymb}
\usepackage{mathtools}
\usepackage{hyperref}
\usepackage{enumitem}
\usepackage{tikz}
\usetikzlibrary{cd}
\usepackage{graphicx}
\usepackage{mathrsfs}
\usepackage{soul}
\usepackage{glossaries}
\usepackage{pdfpages}
\usepackage[style=alphabetic,giveninits=true]{biblatex}
\usepackage[total={210mm,297mm},left=20mm,right=20mm,bindingoffset=10mm,top=25mm,bottom=25mm]{geometry}

\newcommand*\circled[1]{\tikz[baseline=(char.base)]{
            \node[shape=circle,draw,inner sep=2pt] (char) {#1};}}

\newcommand{\R}{\mathbb{R}}
\newcommand{\Q}{\mathbb{Q}}
\newcommand{\Z}{\mathbb{Z}}
\newcommand{\N}{\mathbb{N}}
\newcommand{\C}{\mathbb{C}}
\newcommand{\im}{\mathrm{Im}}
\newcommand{\nuc}{\mathrm{Nuc}}
\newcommand{\sgn}{\mathrm{sgn}}
\newcommand{\vol}{\mathrm{vol}}
\newcommand{\grad}{\nabla}
\newcommand{\bsm}{\left[ \begin{smallmatrix}}
\newcommand{\esm}{\end{smallmatrix} \right]}
\newcommand{\del}{\partial}
\newcommand{\eps}{\varepsilon}
\newcommand{\dif}{\Delta}
\newcommand{\lam}{\lambda}
\newcommand{\orbit}{\mathcal{O}}
\newcommand{\vfi}{\varphi}
\newcommand{\til}{\sim}
\newcommand{\lcm}{\mathrm{lcm}}
\newcommand{\B}{\mathbb{B}}
\newcommand{\cat}{\mathscr{C}}
\newcommand{\perm}{\mathcal{S}}
\newcommand{\cyc}{\mathcal{C}}
\newcommand{\rooted}{\mathcal{A}}
\newcommand{\tree}{\mathfrak{a}}
\newcommand{\hit}{\mathfrak{h}}
\newcommand{\hirt}{\mathcal{H}}
\newcommand{\orit}{\mathfrak{o}}
\newcommand{\E}{\mathbb{E}}

\theoremstyle{plain}
\newtheorem{theorem}{Theorem}[section]
\newtheorem{corollary}[theorem]{Corollary}
\newtheorem{lemma}[theorem]{Lemma}
\newtheorem{prop}[theorem]{Proposition}

\theoremstyle{definition}
\newtheorem{defi}[theorem]{Definition}
\newtheorem{notn}[theorem]{Notation}
\newtheorem{exa}[theorem]{Example}

\setlist[enumerate]{label=\roman*), leftmargin=*, widest = iii}

\OnehalfSpacing

\chapterstyle{madsen}   
\setsecheadstyle{\Large\bfseries\sffamily\raggedright}
\setsubsecheadstyle{\large\bfseries\sffamily\raggedright}
\setsubsubsecheadstyle{\bfseries\sffamily\raggedright}

\makepagestyle{plain}
\makeevenfoot{plain}{\thepage}{}{}
\makeoddfoot{plain}{}{}{\thepage}
\makeevenhead{plain}{}{}{}
\makeoddhead{plain}{}{}{}

\begin{document}
	
%
% Capa
% 	
	
\thispagestyle{empty}
	
{
\sffamily
\centering
\Large

\includegraphics[scale=0.25]{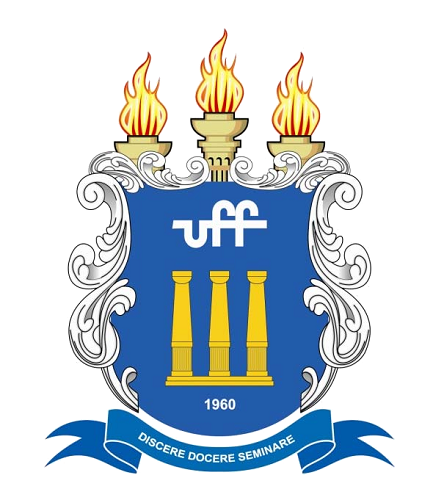}

~\vspace{2cm}

Universidade Federal Fluminense

\vspace{\fill}

{\huge
ON COMBINATORIAL DIFFERENTIAL OPERATORS ON SPECIES OF STRUCTURES

}

\vspace{3.5cm}

{\LARGE

Arthur Gonçalves Fidalgo

}

\vspace{\fill}

{\large Niterói }

{\large  03 / 2023}

}	

\clearpage

\thispagestyle{empty}

{

\sffamily
\centering
\Large

~\vspace{\fill}

{\huge

ON COMBINATORIAL DIFFERENTIAL OPERATORS ON SPECIES OF STRUCTURES

}

\vspace{3.5cm}

{\LARGE

Arthur Gonçalves Fidalgo

}

\vspace{3.5cm}

{\normalsize
\raggedleft	
\begin{minipage}{210pt}
Dissertação submetida ao Programa de Pós-Graduação
em Matemática da Universidade Federal Fluminense
como requisito parcial para a obtenção do grau de 
Mestre em Matemática.
\end{minipage}

}

\vspace{3.5cm}

Orientador: Prof. Slobodan Tanushevski

\vspace{\fill} 

{\large Niterói }

{\large 03 / 2023}

}	

\clearpage

\thispagestyle{empty}

{
\centering

\includepdf{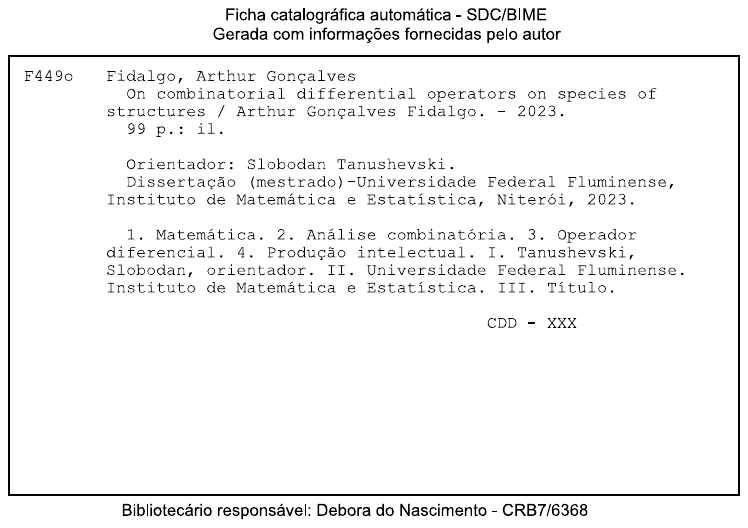}

}
\clearpage

\thispagestyle{empty}

{

\sffamily\centering

\textbf{Dissertação de Mestrado da Universidade Federal Fluminense}

\vspace{1cm}

 por 

\vspace{1cm}

\textbf{Arthur Gonçalves Fidalgo}

\vspace{1.5cm}

apresentada ao Programa de Pós-Graduação em Matemática como requesito parcial para a obtenção do grau de

\vspace{1.5cm}

\textbf{Mestre em Matemática}

\vspace{1.5cm}

Título da tese:

\hrulefill

\begin{minipage}{0.8\textwidth}
\centering
\textbf{ON COMBINATORIAL DIFFERENTIAL OPERATORS ON SPECIES OF STRUCTURES}
\end{minipage}
~\vspace{5pt}
\hrule

\vspace{1cm}

\textit{Defendida publicamente em  03 de Março de 2023.}
	
\vspace{1cm}

{
\raggedright
Diante da banca examinadora composta por:
}

~\vspace{5pt}
%
% Liste os nomes por ordem alfabética

\begin{tabular}{lll}
	Slobodan Tanushevski& Universidade Federal Fluminense & Orientador\\
	Taísa Lopes Martins& Universidade Federal Fluminense & Examinador\\
	Hugo de Holanda Cunha Nobrega& Universidade Federal do Rio de Janeiro & Examinador\\
	%Nome 4& Instituição & Orientador/Coorientador/Examinador\\
	%Nome 5& Instituição & Orientador/Coorientador/Examinador\\
\end{tabular}

~\vspace{\fill}

~\vspace{\fill}

}

%\clearpage

%\thispagestyle{empty}

%{
%{\centering

%\textbf{DECLARAÇÃO DE CIÊNCIA E CONCORDÂNCIA DO(A) ORIENTADOR(A)}
%}

%~\vspace{2cm}

%Autor(a) da Dissertação: Arthur Gonçalves Fidalgo

%Data da defesa: 03/03/2023

%Orientador(a): Slobodan Tanushevski

%\vspace{2cm}

%Para os devidos fins, declaro \textbf{estar ciente} do conteúdo desta \textbf{versão corrigida} elaborada em atenção às sugestões  dos membros da banca examinadora na sessão de defesa do trabalho, manifestando-me \textbf{favoravelmente} ao seu encaminhamento e publicação no \textbf{Repositório Institucional da UFF}.

%~\vspace{1cm}

%{
%\raggedright
%Niterói, 03/03/23.
%}
%~\vspace{1cm}

%\begin{center}
%\begin{minipage}{200pt}
%\centering 	
%	\hrule
	
%	Slobodan Tanushevski
%\end{minipage}
%\end{center}

%}
\includepdf{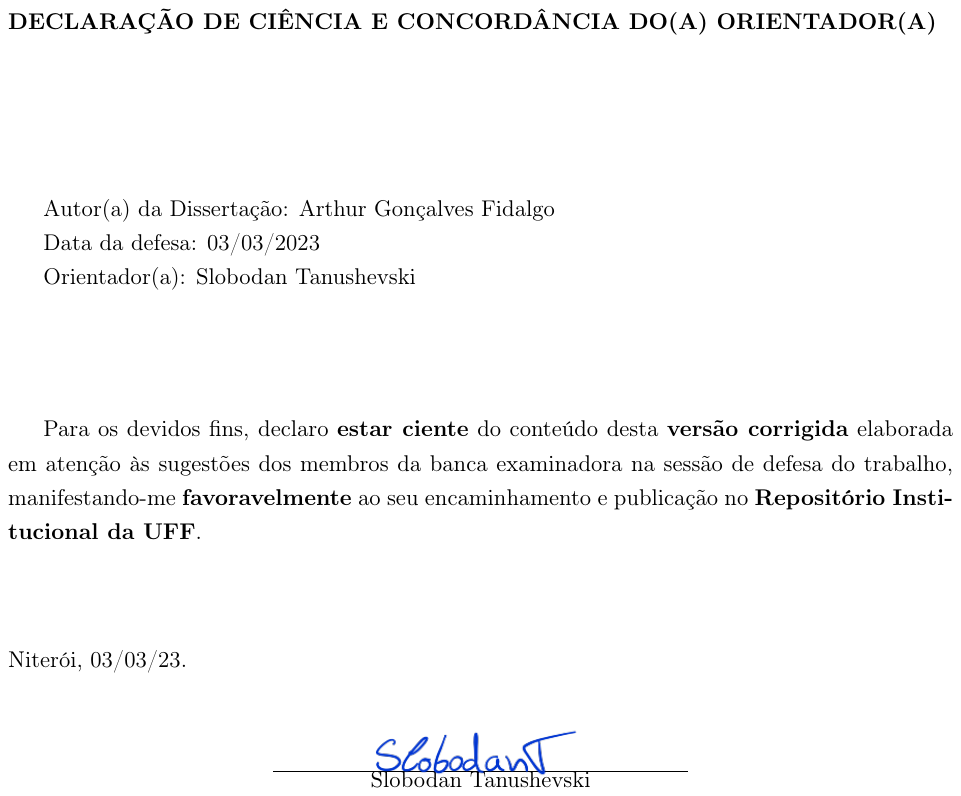}
\clearpage

\thispagestyle{empty}

{

~\vspace{\fill}

\raggedleft
Este trabalho é dedicado àqueles que empregam seu tempo a aumentar seu conhecimento, aos que se veem sempre aprendendo, e àqueles que se mantiveram ao meu lado ao longo desta jornada, em particular Patrícia, Guilherme e Ariane.
}

\clearpage

\thispagestyle{empty}

{

\sffamily

{\Large \centering	
	
  AGRADECIMENTOS
  
}

~\vspace{1cm}

Primeiramente, à minha família, em Patrícia, Guilherme e Ariane, por todo o apoio dado durante essa jornada.

A meus professores da graduação, em particular Laiz e Vinícius, que me indicaram o caminho correto a seguir.

A todos os meus professores da pós-graduação, por todo o conhecimento e capacitação que me passaram, e aos outros profissionais que lá fazem parte do nosso dia-a-dia.

E finalmente, a meus bons amigos, tanto os que fiz na UFF quanto os que trago da minha vida, pelo carinho que me deram durante este tempo, em particular Caio e Alicia, que vêm comigo desde a entrada na pós, e meu irmão do coração Flávio Isaac, que me acompanha há tantos anos.

O presente trabalho foi realizado com apoio da Coordenação de Aperfeiçoamento de Pessoal de Nível Superior - Brasil (CAPES) - Código de Financiamento 001. Esse trabalho foi apoiado com uma bolsa de mestrado da CAPES, e agradeço por tal.

}

\clearpage

\thispagestyle{empty}
\vspace*{\fill}
\begin{flushright}
\textit{"Mathematics was something infinitely more interesting(...)  Its practitioners dwelt in a veritable
conceptual heaven, a majestic poetic realm totally inaccessible to the unmathematical hoi polloi."\\ (Apostolos Doxiadis)}
\end{flushright}
\clearpage

\thispagestyle{empty}

{
	
	\sffamily
	
	{\Large	\centering
		
		RESUMO
		
	}

~\vspace{1cm}

Em 1981, André Joyal \cite{joyal} forneceu uma interpretação combinatória da álgebra de séries de potências formais, um aparato central dentre as ferramentas da combinatória enumerativa. Na teoria de espécies de estruturas de Joyal, espécies combinatórias (como permutações, grafos, partições etc.) são encarnadas em endofuntores da categoria de conjuntos finitos e bijeções. Espécies podem ser somadas, multiplicadas, compostas e derivadas; novas espécies surgem como soluções de equações diferenciais e funcionais.
Além disso, tudo o que se alcança a nível de espécies pode então ser diretamente traduzido para a linguagem de séries geradoras para enumerar estruturas rotuladas e não-rotuladas.

Mais recentemente, Labelle e Lamathe \cite{labelle} desenvolveram uma teoria geral de operadores diferenciais sobre espécies de estruturas, como ciclos ou diagramas de derivadas. O principal objetivo dessa dissertação é apresentar partes desta teoria.

\vspace{\onelineskip} 
\noindent 

\textbf{Palavras-chave}: combinatória, enumeração, espécies de estruturas, equações de recorrência, 
operadores diferenciais combinatórios.

}

\clearpage

\thispagestyle{empty}

{
	
	\sffamily
	
	{\Large\centering
		
		ABSTRACT
		
	}
	
~\vspace{1cm}

 In 1981, André Joyal \cite{joyal} provided a combinatorial interpretation of the algebra of formal power series, a central  gadget in the toolkit of enumerative combinatorics. In Joyal's theory of species of structures, combinatorial species (like permutations, graphs, partitions, etc.) are incarnated in endofunctors on the category of finite sets and bijections. Species can be added, multiplied, composed and differentiated; new species arise as solutions to functional and differential equations. 
 Moreover, everything achieved at the level of species can be directly translated into the language of generating series 
 for enumeration of labelled, as well as unlabelled structures. 

 More recently, Labelle and Lamathe \cite{labelle} developed a general theory of differential operators on species of structures, such as cycles or diagrams of derivatives. The main goal of this dissertation is to present some parts of this theory.

 \vspace{\onelineskip} 

\noindent \textbf{Keywords}: combinatorics, enumeration, species of structures, recurrence equations, combinatorial differential operators. 
	
}
\clearpage

\tableofcontents*

\clearpage

\setcounter{chapter}{-1}
\chapter{Introduction}

In their early mathematical lives, students of mathematics are presented with an intricate way of studying the variation of functions: Newton's (and Leibniz's) differential calculus. This comes with multiple utilities in the realm of physics; for example, the 
concept of velocity can only be properly understood through the notion of derivative.  
%study of velocity is strongly tied with derivation. 
As their studies progress, students are swarmed with abstractions, and the techniques from calculus quickly surpass 
the kingdom of `well behaved' real valued functions. 
%and operators on abstract spaces of functions.     
%are seen less as variations, and more as operators on abstract spaces of functions.

%This abstraction opens the mathematician's mind to the possibility of existence of multiple operators with similar properties, but in other contexts. One example is the study of variations in discrete spaces. In that sense, the concept of variation, and naturally differentiation, is tied with recurrence, as is studied in Chapter \ref{ch1}. The finite difference operator, which is the analogue to the differential in some particular discrete spaces, is used to find solutions to the recurrence issue, but can also be used to solve some classic problems, such as counting surjections between finite sets, and even a question on the geometry of Pythagorean triangles. Moreover, in the realm of power series, these recurrence relations can even be correlated with rational functions, as is studied in Chapter \ref{ch2}.

This opens the mathematician's mind to the possibility of existence of a multitude of (differential) operators with similar properties, but defined in other contexts and operating on different types of mathematical objects. One goal of this dissertation is precisely that: to present some differential operators that arise from the study of combinatorics. One example is the study of variations of functions defined on discrete spaces, where the concept of differentiation is closely tied to the notion of recurrence. 
The finite difference operator, which is the analogue of the differential in discrete spaces, can be used to find solutions to recurrence equations, but can also be used to solve some classical combinatorial problems, such as counting surjections between finite sets and evaluating power sums, and can even help answer a question about the geometry of Pythagorean triangles. 
A brief introduction to finite difference calculus is provided in Chapter~\ref{ch1}.

%The foremost gadget in the toolkit of enumerative combinatorics is the algebra of formal power series. 
%The central dogma of enumerative combinatorics can be resumed in the following set of instructions: 
%firstly, find a recurrence relation between the combinatorial objects that you wish to count; from the recurrence relation derive a functional (or differential) equation for the corresponding generating power series; next, solve the equation to obtain an explicit formula for the generating series; and finally, if you are lucky, you are ready to write down a formula for the coefficients of the generating series. The basic theory of formal power series is presented in Chapter \ref{ch2}.

 After turning their heads away from the shadows, students of mathematics quickly discover that what hitherto appeared to be different is essentially identical (Eureka, ${0.99 \ldots=1}$). Pólya's theory is the art of not counting the same thing twice.
 It can help us enumerate unlabelled graphs, multicoloured necklaces and even chemical isomers (see \cite{polya}). 
 Pólya's theory is explored in Chapter~\ref{ch3}, and its ideas, in particular the use of the \emph{Zyklenzeiger}, were key tools for the study of enumeration done by Joyal.

%Power series can also be used in combinatorics for enumeration. This was most notably done by George Pólya in \cite{polya}, where these were used to enumerate graphs and chemical isomers. This technique that he developed can be used to solve multiple problems in enumerative combinatorics; an example is the question of enumerating multicoloured necklaces. This is briefly explored in Chapter \ref{ch3}.

%Pólya's techniques can be extended to solve even more complex problems in enumeration. This was cleverly done by André Joyal in \cite{joyal}, where he defines the concept of species of structures, in which he expands the idea of generating series to enumerate labelled and unlabelled structures. Moreover, this also provides a tool for the analysis of such combinatorial objects. These are presented in Chapter \ref{ch4}.

By employing the unifying language of category theory, André Joyal \cite{joyal} presented a combinatorial interpretation
of the algebra of formal power series. In Joyal's theory of species of structures, combinatorial species (like permutations, graphs, partitions, etc.) are incarnated in endofunctors on the category of finite sets and bijections. Species can be added, multiplied, differentiated, composed, etc. The operations on species can then be directly translated into equivalent operations on generating series (for enumeration of labelled, as well as unlabelled structures). This movement circumvents the flow of information in the central dogma of enumerative combinatorics: structural relations of combinatorial objects are transcribed directly into algebraic relations between generating series. In Chapter \ref{ch4}, we begin to introduce Joyal's theory of species of structures.  
After developing the right intuition, the persistent reader will be pleasantly surprised to  
find the abstract technical machinery of the theory of species of structures being progressively reduced to a playful manipulation of pictures.     

%Joyal's idea has particular importance in giving a deep correlation between the recurrence of tree-like structures and functional or differential equations. This loops back to the original idea of this dissertation: the abstraction of differentiation. In correlating trees and differential equations, the original Newtonian derivative is replaced by a combinatorial differential operator that works in the much broader realm of (virtual, linear) species. This is explored in Chapter \ref{ch6}.

Joyal's approach to enumeration reveals a deep connection between the recursiveness of tree-like structures and certain types of functional (and differential) equations. This loops back to the original idea of this dissertation: the abstraction of differentiation. 
The original Newtonian derivative is replaced by a combinatorial differential operator, and 
(enriched) trees naturally appear as solutions of differential equations in the realm of (virtual, linear) species. 
This is explored in Chapters \ref{ch5} and \ref{ch6}.

More recently, Labelle and Lamathe \cite{labelle} have developed a theory of general combinatorial differential operators, such as cyclic or graphical arrangements of derivatives. In particular, they extended the difference operator, discussed in Chapter~\ref{ch1}, to the realm of species of structures. This provides a complete overhaul of differentiation, and can be seen as an analogue in combinatorics to what this hypothetical student of mathematics experienced in the first paragraphs of this introduction. Part of the theory of Labelle and Lamathe is presented in Chapter \ref{ch7}.

%More recently, Labelle and Lamathe \cite{labelle2} have succeed in extending these ideas by defining general combinatorial differential operators. These provide a complete abstraction of differentiation, as they combine multisort species and the aforementioned combinatorial differential operator into a new kind of differential operator; for example, this allows the concept of cyclic or graphical arrangements of derivatives. This also combines finite difference calculus with the theory of species. Part of their study is presented in Chapter \ref{ch7}.

The aim of this monograph is to trace the development of a general combinatorial theory of differentiation by following the footsteps of several notable mathematicians. We hope the reader finds this text enjoyable and inspiring.

\chapter*{Preliminary Notations and Definitions}

%In this short chapter we will be presenting some notations and definitions.

\begin{notn}
The set $\{1,2, \ldots ,n\}$ is denoted by $[n]$. We also define $[0]$ to be the empty set.
\end{notn}

\begin{notn}
Given two sets $A, B$, the set of functions from $A$ to $B$ is denoted by $B^A$.
\end{notn}

\begin{defi}
    The Kronecker delta $\delta_{n,k}$ is defined as $\delta_{n,k} = 1$ if $n=k$, and $\delta_{n,k} = 0$ otherwise.
\end{defi}

\begin{defi}
The polynomial $x^{\underline{k}} = x(x-1)\cdots(x-k+1)$, $k \geq 1$, is called the $k$-falling factorial of $x$. Also, $x^{\underline{0}} := 1$.
\end{defi}

\begin{defi}
The polynomial \(\binom{x}{k}\) $= \frac{1}{k!}x^{\underline{k}}$ is called $x$-choose-$k$, or the $x,k$-binomial. In particular, one can define $\binom{0}{k} = \delta_{0,k}$.
\end{defi}

\begin{notn}
The set of $k$-element subsets of a finite set $A$ is denoted by \(\binom{A}{k}\). Naturally, |\(\binom{A}{k}\)|$=$\(\binom{|A|}{k}\).
\end{notn}

\begin{defi}
    For $k \geq 2$, define the multinomial coefficient $\binom{n}{a_1,\ldots,a_k}$ as $$\binom{n}{a_1,\ldots,a_k} := \binom{n}{a_1}\binom{n-a_1}{a_2}\cdots\binom{n-a_1-\ldots-a_{k-1}}{a_k} = \frac{n!}{a_1!\cdots a_k!}.$$
\end{defi}

\chapter{Finite Difference Calculus}\label{ch1}

This chapter begins by defining the finite difference operator and some of its properties, and aims to give a brief overview of difference calculus, a discrete analogue to differential calculus. This is done by first presenting the operator itself, then proceeding to show its usefulness on the classical problem of counting surjections between finite sets. After that, we proceed to the study of difference calculus by showing an analogous to the Fundamental Theorem of Calculus, and also to differential equations. For more details on difference calculus, we refer the reader to \cite[Ch. 9]{wagner} and \cite[Ch. 3]{liu}.
\section{The Finite Difference Operator}

To work with differential calculus, one must first define the differential operator. In order to do calculus in discrete spaces, one must also define a linear operator with analogous properties.

\begin{defi}
Let $S \subseteq \C$ be a non-empty set such that, if $x \in S$, then $x+1 \in S$. Some good choices for $S$ are $\N,\Z$ and $\R$. When $S \subseteq \Z$, a function $f \in \C^S$ is often called a sequence. For $f \in \C^S$, define the difference operator $\dif$ as 
$$ \dif f(x) := f(x+1)-f(x).$$
\end{defi}

The difference operator has the following properties:

\begin{prop}
\label{8part}
Let $f,g \in \C^S$. Then
\end{prop}
\begin{enumerate}
    \item $\dif$ is a linear operator on $\C^S$;
    \item If $f$ is a constant function on $S$, then $\dif f \equiv 0$;
    \item $\dif x^n = \sum_{k=0}^{n-1}$ \(\binom{n}{k}\)$x^k$;
    \item $\dif x^{\underline{n}} = nx^{\underline{n-1}}$;
    \item $\dif \binom{x}{n} = \binom{x}{n-1};$
    \item $\dif 2^x = 2^x$;
    \item $\dif [f(x)g(x)] = f(x+1)\dif g(x) + g(x)\dif f(x)$
    
    $\quad\quad\quad\quad\quad\ \, = f(x)\dif g(x) + g(x+1)\dif f(x) $
    
    $\quad\quad\quad\quad\quad\ \, = f(x)\dif g(x) + g(x)\dif f(x) + \dif f(x) \dif g(x)$;
    \item $\displaystyle \dif\frac{f(x)}{g(x)} = \frac{g(x)\dif f(x) - f(x) \dif g(x)}{g(x)g(x+1)}$.
\end{enumerate}

\begin{proof}
\begin{enumerate}
    \item \emph{Addition:} $\dif [f(x)+g(x)] = f(x+1)+g(x+1)-f(x)-g(x) = \dif f(x) + \dif g(x)$.

    \emph{Scalar multiplication:} $\dif \lambda f(x) = \lambda f(x+1)-\lambda f(x) = \lambda [f(x+1)-f(x)] = \lambda \dif f(x)$.
    \item If $f \equiv c \in \C$, then $f(x+1) = f(x) = c$ for all $x \in S$, and $\dif f(x) = c-c = 0$.
    \item We have $(x+1)^n = \sum_{k=0}^n$\(\binom{n}{k}\)$x^k$, from where we get $\dif x^n = (x+1)^n - x^n = \sum_{k=0}^{n-1}$\(\binom{n}{k}\)$x^k$.
    \item $\dif x^{\underline{n}} = (x+1)^{\underline{n}} - x^{\underline{n}} = (x+1)x \cdots (x-n+2)- x(x-1)\cdots (x-n+1)$ 
    
    $\quad\quad\,= [x(x-1) \cdots (x-n+2)](x+1-x+n-1) = nx^{\underline{n-1}}.$
    \item $\displaystyle \dif\binom{x}{n} = \dif \frac{x^{\underline{n}}}{n!} = \frac{1}{n!}\dif x^{\underline{n}}=\frac{nx^{\underline{n-1}}}{n!} = \frac{x^{\underline{n-1}}}{(n-1)!} = \binom{x}{n-1}$.
    \item $\dif 2^x = 2^{x+1}-2^x = 2^x(2-1) = 2^x$. 
    \item $
        %\dif [f(x)g(x)] &= f(x+1)g(x+1)-f(x)g(x) \\
        %&= f(x+1)g(x+1) - f(x)g(x+1) + f(x)g(x+1) - f(x)g(x) \\&= f(x)\dif g(x) + g(x+1)\dif f(x); \\
        \dif [f(x)g(x)] = f(x+1)g(x+1)-f(x)g(x)$
        
        $\quad\quad\quad\quad\quad\ \, = f(x+1)g(x+1) - f(x+1)g(x) + f(x+1)g(x) - f(x)g(x)$
        
        $\quad\quad\quad\quad\quad\ \, = f(x+1)\dif g(x) + g(x)\dif f(x)$. %\\
        %\dif [f(x)g(x)] &= f(x+1)g(x+1)-f(x)g(x) \\
        %&= f(x+1)g(x+1) - f(x)g(x+1) + f(x)g(x+1) - f(x)g(x) \\
        %&={f(x)\dif g(x) + f(x+1)g(x+1) - f(x)g(x+1) + f(x+1)g(x)} \\
        %&\quad{-f(x+1)g(x) + f(x)g(x)-f(x)g(x)} \\
        %&= f(x)\dif g(x) + g(x)\dif f(x) + f(x+1)\dif g(x) - f(x)\dif g(x) \\
        %&= f(x)\dif g(x) + g(x)\dif f(x) + \dif f(x) \dif g(x).
        
    The other two identities can be proved in a similar way.
    \item $
        \displaystyle \dif \frac{f(x)}{g(x)} = \frac{f(x+1)}{g(x+1)} - \frac{f(x)}{g(x)} = \frac{f(x+1)g(x)-f(x)g(x+1)}{g(x)g(x+1)}$
        
        $\displaystyle\quad\quad\quad\,= \frac{f(x+1)g(x)-f(x)g(x)+f(x)g(x)-f(x)g(x+1)}{g(x)g(x+1)} = \frac{g(x)\dif f(x) - f(x) \dif g(x)}{g(x)g(x+1)}$.
\end{enumerate}
\end{proof}

Let us now present an application of the difference operator in the realm of geometry.

\begin{defi}
A triangle with a right angle is said to be Pythagorean when the lengths of its three sides are positive integers. A triplet of positive integers $(a,b,c)$, with $a< b<c$, is said to be a Pythagorean triplet when $a^2 + b^2 = c^2$.
\end{defi}

\begin{defi}
Two Pythagorean triplets $(a,b,c),(x,y,z)$ are said to be of the same class when there exists a positive integer $\lambda$ such that $(a,b,c) = (\lambda x, \lambda y, \lambda z)$.
\end{defi}

\begin{theorem}
There are infinitely many Pythagorean triangles with different shapes, that is, there are infinitely many Pythagorean triplet classes.
\end{theorem}
\begin{proof}
%From Theorem \ref{8part}.iii, one has $\dif x^2 = 2x+1$. Consider the triplet $(a,x,x+1)$. For every positive odd integer $a$, $a^2$ is also a positive odd integer, therefore for each positive odd integer $a \geq 3$ there is a positive integer $x = \frac{a^2-1}{2}$ such that $\dif x^2 = a^2$, and therefore for each positive odd integer $a \geq 3$ there is a Pythagorean triplet given by $(a,\frac{a^2-1}{2},\frac{a^2-1}{2}+1)$. It is also clear that, for two different odd numbers, these triplets will not be in the same class, as multiplying consecutive numbers also multiplies their distance.
From Proposition \ref{8part}.iii), $\dif x^2 = 2x+1$. Let $a \geq 3$ be a positive odd integer. Then, for $x = \frac{a^2 - 1}{2}$, one has $\dif x^2 = a^2$. In other words, ${(x+1)^2 = x^2 + a^2}$, and $(a, \frac{a^2 - 1}{2}, \frac{a^2 - 1}{2}+1)$ is a Pythagorean triplet. Also, if $b \geq 3$ is another positive odd integer, there is no integer $\lambda > 1$ such that $b = \lambda a$ and $\frac{b^2 - 1}{2} = \lambda\frac{a^2 - 1}{2}$; therefore one can pair positive odd integers $a \geq 3$ with Pythagorean classes of the form $(a, \frac{a^2 - 1}{2}, \frac{a^2 - 1}{2}+1)$, from where we get that there are infinitely many of them.
\end{proof}

Just like in the continuous case, one can define higher order difference operators inductively.

\begin{defi}
The $n$-th order difference operator is defined as follows: $\dif^0 f = f$, and $\dif^{n+1} f = \dif(\dif^n f)$.
\end{defi}

We still have linearity even in higher orders. By Proposition~\ref{8part}.iii), the difference operator reduces by one the degree of a polynomial. Hence, the $n$-th order difference operator reduces by $n$ the degree of a polynomial. 

By an easy induction, we get the following generalization of Proposition \ref{8part}.iv) and v).   

\begin{corollary}
\label{highord}
If $0 \leq k \leq n$, $\dif^k x^{\underline{n}} = n^{\underline{k}}x^{\underline{n-k}}$, and $\dif^k$\(\binom{x}{n}\) = \(\binom{x}{n-k}\).
\end{corollary}

In order to be able to work efficiently with higher order difference operators, we write $\dif$ as the difference of two simpler commuting operators.  

\begin{defi}
\label{shift}
The shift operator $E$ is defined as $Ef(x) = f(x+1)$, and higher order shift operators are defined by $E^n f(x) = f(x+n)$ for $n \in \N$. 
\end{defi}

\begin{defi}
The identity operator $I$ is defined as $If(x) = f(x)$.
\end{defi}

Clearly, $\dif = E - I$. With this we can see that the value of $\dif^k f(x)$ is determined by the values of $f(x), \ldots ,f(x+k)$.

\begin{theorem}
\label{teodk}
For all $k \in \N$, $\dif^kf(x) = \sum_{j=0}^k (-1)^{k-j}$\(\binom{k}{j}\)$f(x+j)$.
\end{theorem}
\begin{proof}
Since $I$ and $E$ commute, we have 
\begin{align*}
\dif^k f(x) = (E-I)^k f(x) &= \bigg[\sum_{j=0}^k \binom{k}{j} (-I)^{k-j} E^j\bigg] f(x) \\
&=\sum_{j=0}^k (-1)^{k-j} \binom{k}{j} E^j f(x) = \sum_{j=0}^k (-1)^{k-j} \binom{k}{j} f(x+j).   
\end{align*}
\end{proof}

The previous theorem can be used to prove a well-known orthogonality relation for binomial coefficients.

\begin{corollary}
For $n,k \in \N$, $\sum_{j=0}^k(-1)^{k-j}$\(\binom{k}{j}\)\(\binom{j}{n}\)$ = \delta_{k,n}$.
\end{corollary}
\begin{proof}
From Theorem \ref{teodk}, this sum is equal to $[\dif^k\binom{x}{n}]|_{x=0}$, and by Corollary \ref{highord}, it is equal to 
\(\binom{0}{n-k} =\delta_{k,n}\).
\end{proof}
%\newpage
\begin{notn}
Let $p(x) \in \C[x]$. Denote $p_k(m) := [\dif^k p(x)]|_{x=m}$.
\end{notn}

%Recall now that any sequence $p_0(x),p_1(x),p_2(x),\ldots \in \C[x]$, with $\deg p_i = i$ for each 
%$i \geq 0$, forms a basis for $\C[x]$. 
A polynomial $p(x) \in \C[x]$ of degree $n$ is fully determined by the values $p(0),\ldots ,p(n)$. From Theorem \ref{teodk}, we also have that these values determine $p_k(0)$, $0 \leq k \leq n$. Conversely, there is an interesting formula that can be used to recover the polynomial $p(x)$ from $p_k(0)$ ($0 \leq k \leq n$), giving a finite difference analogue to Taylor's theorem in the case of polynomials.
\newpage
\begin{theorem}
\label{taylor}
Let $p(x) \in \C[x]$ with $\deg p = n$. Then $p(x) = \sum_{k=0}^n p_k(0)$\(\binom{x}{k}\).
\end{theorem}
\begin{proof}
Since the binomial coefficients $\binom{x}{0},\binom{x}{1},\binom{x}{2},\ldots$ form a basis for $\C[x]$ (seen as a vector space over $\C$), we have that $p(x) = \sum_{j=0}^n a_j$\(\binom{x}{j}\) for some  $a_j \in \C$, $0 \leq j \leq n$. Then 
$$\dif^k p(x) = \sum_{j=0}^n a_j \dif^k \binom{x}{j} = \sum_{j=k}^n a_j \binom{x}{j-k},$$ 
from where we get $p_k(0) = [\dif^k p(x)]|_{x=0} = a_k$.
\end{proof}

\section{Counting Surjections}

The purpose of this section is to illustrate the usefulness of difference calculus to combinatorial analysis. We focus on the classical problem of counting surjections between finite sets.

%As discussed in Chapter 1, there is no simple closed formula for counting surjective functions between two finite sets. In this section, we will use the difference operator to find a (not so simple) formula for counting surjections. But first we must define two more linear operators.

Before taking on the task of counting surjections, we will first determine the number of injections between two finite sets.

\begin{theorem}
Let $A,B$ be finite sets. Then the number of injective functions from $A$ 
to $B$ is $|B|^{\underline{|A|}}$. In particular, if $|A|=|B|=n$, there are $n!$ bijections between $A$ and $B$.
\end{theorem}
\begin{proof}
Let $A = \{a_1, \ldots ,a_n\}$. If $|B| = k$, then for an injective function $f \in B^A$, there are $k$ choices for $f(a_1)$, then $k-1$ possible choices for $f(a_2)$, and so on. Therefore, there are $k(k-1) \cdots (k-n+1) = k^{\underline{n}} = |B|^{\underline{|A|}}$ possible injections from $A$ to $B$.

If $|A|=|B|=n$, then there are $n^{\underline{n}} = n!$ injections from $A$ to $B$. Naturally, the injections take the $n$ elements of $A$ into $n$ distinct elements of $B$, that is, onto all of $B$, and are therefore bijections.
\end{proof}

Unlike counting injections, counting surjections isn't such a triviality. 

\begin{notn}
The number $\sigma(n,k)$ denotes the amount of surjective functions from $[n]$ to $[k]$.
\end{notn}

%\newpage
\begin{theorem}
For all $n,k \in \N$, $\sigma(n,0) = \sigma(0,n) = \delta_{n,0}$ and $$ \sigma(n,k) = k\sigma(n-1,k-1)+k\sigma(n-1,k). $$
\end{theorem}
\begin{proof}
Consider possible distributions of $n$ labelled balls into $k$ labelled boxes with no box being left empty. There are $k\sigma(n-1,k-1)$ of those where the ball labelled $n$ is left alone, and $k\sigma(n-1,k)$ where it isn't.
\end{proof}

%\begin{lemma}
%\label{polyeq}
%Let $p(x),q(x)$ be two polynomials with degree $n$. If there exist $n+1$ elements $x_1,...,x_{n+1}$ such that, for all $i \in [n+1]$, $p(x_i) = q(x_i)$, then $p(x) = q(x)$.
%\end{lemma}
%\begin{proof}
%Since $p,q$ have degree $n$, $\deg p-q \leq n$. As polynomials of degree $n$ have at most $n$ roots, the polynomial $g(x) = p(x)-q(x)$ has at most $n$ roots, but the hypothesis says it has $n+1$ roots, therefore $g(x) \equiv 0$, and $p(x) = q(x)$.
%\end{proof}

This theorem gives us a way to count surjections using recursiveness, but it would also be interesting to have a closed form formula for them. For that, we must correlate the polynomials $x^n$ and $\binom{x}{k}$.

\begin{theorem}
\label{xnsig}
For all $n \in \N$, $x^n = \sum_{k=0}^n \sigma(n,k)$\(\binom{x}{k}\).
\end{theorem}
\begin{proof}
Since two distinct polynomials in $\C[x]$ can not take the same values on an infinite set,    
it is enough to prove that, for all $r \in \N \setminus \{0\}$, ${r^n = \sum_{k=0}^n \sigma(n,k)\binom{r}{k}}$. 
To see that this is the case, note that among all functions $f \in [r]^{[n]}$, $\sigma(n,k)$\(\binom{r}{k}\) enumerates those for which $|\textrm{Im}(f)| = k$, as one can choose \(\binom{r}{k}\) subsets of $[r]$ with $k$ elements, and for each of these subsets there are $\sigma(n,k)$ surjections from $[n]$ to it.
\end{proof}

We digress briefly in order to present an interesting application of the theorem above.

\begin{corollary}
The $n$-th order difference of the polynomial $x^n$ is constant equal to $n!$, that is, $\Delta^n x^n \equiv n!$.
\end{corollary}
\begin{proof}
%Immediate from Theorems \ref{xnsig} and \ref{taylor}, setting $k=n$, and the fact that $\dif^n$ reduces the degree of a polynomial by $n$.
From Theorem \ref{xnsig} we have $x^n = \sum_{k=0}^n \sigma(n,k) \binom{x}{k}$. Combining that with Theorem \ref{taylor}, we get $[\dif^n x^n]|_{x = 0} = \sigma(n,n) = n!$. Since $\dif^n$ reduces the degree of a polynomial by $n$, we have that $\dif^n x^n$ is a constant (equal to $n!$).
\end{proof}

Theorems~\ref{teodk} and \ref{taylor} give us an interesting way to enumerate surjective functions between two finite sets.
%\newpage
\begin{theorem}
\label{sigma}
Given $n,k \in \N$, $\sigma(n,k) = \sum_{j=0}^k (-1)^{k-j}$\(\binom{k}{j}\)$j^n$.
\end{theorem}
\begin{proof}
By Theorem \ref{xnsig}, $x^n = \sum_{k=0}^n \sigma(n,k)$\(\binom{x}{k}\). From Theorem \ref{taylor}, ${\sigma(n,k) = [\dif^k x^n]|_{x=0}}$, and Theorem \ref{teodk} gives us that $[\dif^k x^n]|_{x=0} = \sum_{j=0}^k (-1)^{k-j}$\(\binom{k}{j}\)$j^n$.
\end{proof}

\section{The Fundamental Theorem of Finite Difference Calculus}
Given a function $f \in \C^\N$, if $\dif F = f$, then we say that $F$ is an antidifference of $f$, and we write 
$F = \dif^{\langle-1\rangle}f$. Clearly, if $F$ is an antidifference of $f$, then for every $c \in \C$, $F+c$ is also an 
antidifference of $f$. Moreover, this is a full description of all antidifferences of $f$. Indeed, suppose that $G$ is an antidifference of $f$; then $\dif (F - G) \equiv 0$, and consequently $G = F + c$ for some $c \in \C$.

The following theorem is the discrete analogue of the fundamental theorem of calculus.

\begin{theorem}
\label{fundamental-theorem}
Let $f \in \C^{\N}$, and let $a,b \in \N$ with $a \leq b$. Then 
\begin{enumerate}
\item $f$ has an antidifference;
\item $\sum_{x=a}^b f(x) = F(b+1)-F(a)$, where $F$ is any antidifference of $f$.
\end{enumerate}
\end{theorem}
\begin{proof}
$i)$ It is easy to check that the function defined by $F(0):=0$ and $F(x):=\sum_{i=0}^{x-1} f(i)$ for $x \geq 1$ is an antidifference of $f$.
\medskip

$ii)$ From the hypothesis that $F = \dif^{\langle-1\rangle}f$, we have $$\sum_{x=a}^b f(x) = \sum_{x=a}^b \dif F(x) = \sum_{x=a}^b [F(x+1)-F(x)]$$ and this last sum telescopes to $F(b+1)-F(a)$.
\end{proof}

Let $p(x)=\sum_{k=0}^n a_k\binom{x}{k} \in \C[x]$ be a polynomial of degree $n \geq 0$.  
It follows from Proposition~\ref{8part}.v) that the polynomial $\sum_{k=0}^n a_k\binom{x}{k+1} \in \C[x]$ is an 
antidifference of $p(x)$. Therefore, every antidifference of a non-zero polynomial of degree $n \geq 0$ is a polynomial of degree $n+1$.

Let us now integrate the polynomial $(x+1)^r$. 

\begin{defi}
Let $r,n \in \N$. The $r$-th power sum is defined as $S_r(n) := \sum_{k=0}^n k^r$. For $S_0(n)$ to make sense, also consider $0^0 = 1$.
\end{defi}

Note that $S_r(x)$ is an antidifference of $(x+1)^r$. It follows from the discussion above that $S_r(x)$ is a polynomial of degree 
$r+1$. 

%\begin{theorem}
%The $r$-th power sum $S_r(n)$ is a polynomial of degree $r+1$.
%\end{theorem}
%\begin{proof}
%By \ref{8part}.iii), we have $(k+1)^{r+1} - k^{r+1} = \sum_{i=0}^r$\(\binom{r+1}{i}\)$k^i$. From this, we get $$ \sum_{k=0}^n [(k+1)^{r+1}-k^{r+1}] = \sum_{k=0}^n \sum_{i=0}^r\binom{r+1}{i}k^i = \sum_{i=0}^r \binom{r+1}{i} \sum_{k=0}^n k^i = \sum_{i=0}^r \binom{r+1}{i} S_i(n) $$

%Since the left part above telescopes to $(n+1)^{r+1}$, this gives us $(n+1)^{r+1} = \sum_{i=0}^r \binom{r+1}{i} S_i(n)$, and solving for $S_r(n)$ yields the recursive formula $$S_r(n) = \frac{1}{r+1}\left[(n+1)^{r+1} - \sum_{i=0}^{r-1} \binom{r+1}{i} S_i(n)\right]$$

%It is easy to see that $S_0(n) = n+1$, and this recursive formula then shows us that $S_r(n)$ is a polynomial of degree $r+1$.
%\end{proof}

\begin{theorem}
The $r$-th power sum can be written as $S_r(x) = \sum_{j=0}^r \sigma(r,j)$\(\binom{x+1}{j+1}\).
\end{theorem}
\begin{proof}
For $r=0$, this is trivially satisfied. For $r \geq 1$, write $S_r(x) = \sum_{j=0}^r a_{j+1}$\(\binom{x+1}{j+1}\). From Theorem \ref{taylor}, $a_{j+1} = [\dif^{j+1} S_r(x)]|_{x+1=0}$. We have $\dif S_r(x) = (x+1)^r$, from where we get $\dif^{j+1} S_r(x) = \dif^j (x+1)^r$. From Theorem \ref{teodk}, $$[\dif^j (x+1)^r]|_{x=-1} = \left[\sum_{k=0}^j (-1)^{j-k}\binom{j}{k}(x+1+k)^r\right]\Bigg|_{x=-1} = \sum_{k=0}^j (-1)^{j-k}\binom{j}{k}k^r$$
and from Theorem \ref{sigma}, $\sum_{k=0}^j (-1)^{j-k}\binom{j}{k}k^r = \sigma(r,j)$. We conclude that $a_{j+1} = \sigma(r,j)$.
\end{proof}

\section{Recurrence Equations}
\label{Recurrence Equations}
In usual calculus, the usage of derivatives leads to the concept of differential equations. In the same manner, the finite difference operator gives rise to the concept of recurrence equations. The goal of this section is to present linear recurrence equations and show a way to find closed form formulas for their solutions. 

All functions considered in this section belong to $\mathbb{C}^{\mathbb{N}}$.

\begin{defi}\label{141}
Let $d$ be a positive integer, $c_0, \ldots ,c_d, v_0, \ldots ,v_{d-1} \in \C$, with $c_0 c_d \neq 0$. A homogeneous linear recurrence equation of degree $d$ with constant coefficients is an equation of the form \begin{equation}\label{recequ} c_d f(n+d) + c_{d-1} f(n+d-1) + \ldots + c_1 f(n+1) + c_0 f(n) = 0 \end{equation} with initial conditions $ f(0) = v_0, f(1)=v_1, \ldots , f(d-1) = v_{d-1} $.
\end{defi}

In order to simplify the terminology, we refer to homogeneous linear recurrence equations with constant coefficients simply as recurrence equations. The name recurrence equation comes naturally from the fact that the initial conditions define recursively the solution: \begin{equation}\label{rec}
    f(n+d) = -\frac{1}{c_d}\sum_{k=0}^{d-1}c_kf(n+k), \, n \in \N.
\end{equation}

This recursiveness guarantees that, given initial conditions, there is always a unique solution to recurrence equations. This implies that the discrete analogue to Picard's theorem for finding solutions to differential equations is a triviality.

To unveil a method for describing solutions to a given recurrence equation, one must recall the operator $E = \dif + I$, as defined in \ref{shift}. Using the shift operator, one can rewrite equation \ref{recequ} as \begin{equation}\label{reqe}
    (c_dE^d + c_{d-1}E^{d-1}+ \ldots +c_1E+c_0I)(f) = \overline{0},
\end{equation}where $\overline{0}(n) = 0$ for all $n \in \N$.

\begin{notn}\label{l}
Let us denote by $L$ the linear operator $c_dE^d + \ldots +c_0I$, and by $N_L$ the set of all solutions to equation \ref{reqe}.
\end{notn}
\newpage
\begin{theorem}\label{nldim}
The set $N_L$ is a subspace of $\C^\N$ of dimension $d$.
\end{theorem}
\begin{proof}
The map $f \mapsto (f(0),\ldots ,f(d-1))$ is a linear map from $N_L$ to $\C^d$. It is injective by uniqueness of solutions, and is surjective by existence of solutions given initial conditions.
\end{proof}

Since the solutions to a recurrence equation form a vector space, it is only natural to ask for a basis. For that, we must first define the characteristic polynomial.

\begin{defi}
Given a linear operator $L$ as in \ref{l}, the characteristic polynomial $p_L(x)$ of $L$ is defined as 
$p_L(x) := c_dx^d + c_{d-1}x^{d-1} + \ldots + c_1x + c_0$.
\end{defi}
%\newpage
\begin{theorem}\label{plan}
Let $\alpha \in \C$. If $p_L(\alpha)=0$, then the function $f(n) = \alpha^n$ is a solution to $L(f)=\overline{0}$.
\end{theorem}
\begin{proof}
By hypothesis, $p_L(\alpha) = c_d \alpha^d + \ldots + c_1 \alpha + c_0 = 0$, therefore 
$$c_d \alpha^{n+d} + \ldots + c_1 \alpha^{n+1} + c_0 \alpha^n=\alpha^n p_L(\alpha)= 0$$ 
for all $n \in \N$, and $f(n) = \alpha^n$ is a solution to $L(f)=\overline{0}$.
\end{proof}

Now that we have an information on some solutions to recurrence equations, let us show that functions of the form $\alpha^n$ can be used to find a basis for $N_L$. We consider two different cases, one with all roots distinct and one more general.

\begin{theorem}\label{drbase}
Suppose that $p_L(x)$ has $d$ distinct roots $\alpha_1, \ldots ,\alpha_d$. Denote $f_j(n) = \alpha_j^n$. Then $\{f_1, \ldots ,f_d\}$ is a basis for $N_L$.
\end{theorem}
\begin{proof}
By Theorem \ref{plan}, each $f_j$ is an element of $N_L$, and Theorem \ref{nldim} shows that the dimension of $N_L$ is $d$. Therefore, it is sufficient to show that the functions $f_1, \ldots ,f_d$ are linearly independent.

Let $A_1 f_1(x) + \ldots + A_d f_d(x) = \overline{0}(x)$. Evaluating for $x=0, \ldots, d-1$, we get the following system of equations:

\begin{equation}\label{sys0}
    \begin{cases}
    A_1 + A_2 + \ldots + A_{d-1} + A_d = 0\\
    A_1 \alpha_1 + A_2 \alpha_2 + \ldots + A_{d-1} \alpha_{d-1} + A_d \alpha_d = 0\\
    \quad\vdots\\
    A_1 \alpha_1^{d-1} + A_2 \alpha_2^{d-1} + \ldots + A_{d-1} \alpha_{d-1}^{d-1} + A_d \alpha_d^{d-1} = 0
    \end{cases}
\end{equation}

This is equivalent to the following matrix equation:

\begin{equation}\label{sysm0}
    \begin{bmatrix}
    1&1&...&1&1\\
    \alpha_1&\alpha_2&...&\alpha_{d-1}&\alpha_d\\
    \vdots&\vdots&\ddots&\vdots&\vdots\\
    \alpha_1^{d-1}&\alpha_2^{d-1}&...&\alpha_{d-1}^{d-1}&\alpha_d^{d-1}
    \end{bmatrix}
    \begin{bmatrix}
    A_1\\A_2\\\vdots\\A_{d-1}
    \end{bmatrix}=
    \begin{bmatrix}
    0\\0\\\vdots\\0
    \end{bmatrix}
\end{equation}

System \ref{sys0} has $A_1 = \ldots = A_d = 0$ as its unique solution if and only if the $d \times d$ matrix in \ref{sysm0} is invertible. This is equivalent to it having non-zero determinant. But this matrix is a Vandermonde matrix, and its determinant is given by $\prod_{1 \leq i < j \leq d} (\alpha_j - \alpha_i)$, as seen in \cite[Theorem 9.6.4]{wagner}. Since $\alpha_i \neq \alpha_j$ for $i \neq j$ by hypothesis, the matrix is invertible, and therefore, the vectors $(f_i(0), \ldots ,f_i(d-1)), 1 \leq i \leq d,$ are linearly independent in $\C^d$. Moving backwards through the isomorphism presented in the proof of Theorem \ref{nldim}, this implies that the set $\{f_1,\ldots ,f_d\}$ is linearly independent.
\end{proof}

%From this last theorem we see that these geometric progressions form a basis for $N_L$, which gives us a neat way to find solutions to $L(f) = \overline{0}$ with given initial conditions. \textcolor{blue}{(Not a good sentence.)}

\begin{theorem}\label{finddr}
Let $v_0, \ldots ,v_{d-1} \in \C$, and suppose that $p_L(x)$ has $d$ distinct roots $\alpha_1,...,\alpha_d$. 
Then $f(n) = A_1 \alpha_1^n + \ldots + A_d \alpha_d^n$, with $A_1,\ldots ,A_d$ given by
\begin{equation*}
    \begin{cases}
    A_1 + A_2 + ... + A_{d-1} + A_d = v_0\\
    A_1 \alpha_1 + A_2 \alpha_2 + ... + A_{d-1} \alpha_{d-1} + A_d \alpha_d = v_1\\
    \quad\vdots\\
    A_1 \alpha_1^{d-1} + A_2 \alpha_2^{d-1} + ... + A_{d-1} \alpha_{d-1}^{d-1} + A_d \alpha_d^{d-1} = v_{d-1},
    \end{cases}
\end{equation*}
is the unique solution to $L(f)=\overline{0}$ with initial conditions $f(0)=v_0,\ldots ,f(d-1)=v_{d-1}$.
\end{theorem}
\begin{proof}
Theorem \ref{drbase} shows that all functions in $N_L$ are of the form $f(n) = A_1 \alpha_1^n + \ldots + A_d \alpha_d^n$ for some $A_1, \ldots ,A_d$. Uniqueness of solutions comes from the recursiveness in \ref{rec}. The values of $A_1,\ldots ,A_d$ are then simply given by evaluating $A_1 \alpha_1^n + \ldots + A_d \alpha_d^n$ for $n = 0, \ldots ,d-1$.
\end{proof}

Theorem \ref{finddr} gives us a method to find a closed form formula for some recurrent sequences. In particular, 
we can use the theorem to find a formula for the Fibonacci sequence.

\begin{corollary}
The $n$-th Fibonacci number is given by $F_n = \frac{1}{\sqrt{5}}[\Phi^{n+1}-\phi^{n+1}]$, where $\Phi = \frac{1+\sqrt{5}}{2}$ and $\phi = \frac{1-\sqrt{5}}{2}$.
\end{corollary}
\begin{proof}
It is known that $F_{n+2} = F_{n+1}+F_n$. This shows that, if $F(n) = F_n$, then $L(F) = \overline{0}$, with $L = E^2-E-I$. The two roots of $p_L(x) = x^2-x-1$ are $\Phi$ and $\phi$, so $F_n = A_1 \Phi^n + A_2 \phi^n$. Solving the system in Theorem \ref{finddr}, with initial conditions $v_0=v_1=1$, gives $A_1 = \frac{\Phi}{\sqrt{5}}$ and $A_2 = -\frac{\phi}{\sqrt{5}}$.
\end{proof}

Let us now consider the more general case where $p_L(x)$ may have roots with multiplicities.

\begin{theorem}
Let $p_L(x) = c_d \prod_{j=1}^k (x-\alpha_j)^{m_j}$, with all $m_j > 0$, all $\alpha_j$ pairwise distinct, and $m_1 + \ldots  + m_k = d$. Then a basis for $N_L$ is given by $\bigcup_{j=1}^k \{\alpha_j^n, n\alpha_j^n, \ldots ,n^{m_j -1}\alpha_j^n\}$.
\end{theorem}
\begin{proof}
We prove by induction in $m$ (with $0 \leq m < m_j$) that $(E-\alpha_jI)^{m+1}(n^m\alpha_j^n) = \overline{0}$.
The case $m=0$ being trivial, assume that $m \geq 1$.  

By the induction hypothesis, $(E-\alpha_jI)^{m}(n^r \alpha_j^n) = \overline{0}$ for all $0 \leq r \leq m - 1$, and by linearity, we have $(E-\alpha_jI)^{m}(p(n) \alpha_j^n) = \overline{0}$ for every polynomial $p(n)$ of degree at most $m-1$. 

Now it is easy to see that ${(E-\alpha_jI)(n^{m}\alpha_j^{n-1}) = p(n) \alpha_j^n}$ for some polynomial $p(n)$ of degree at most 
$m-1$. Hence, 
$$(E-\alpha_jI)^{m+1} (n^{m}\alpha_j^{n}) = \alpha_j(E-\alpha_jI)^{m}[(E-\alpha_jI)(n^{m}\alpha_j^{n-1})]
=\alpha_j(E-\alpha_jI)^{m}(p(n) \alpha_j^n)=\overline{0}.$$

It follows that $\{\alpha_j^n, n\alpha_j^n, \ldots ,n^{m_j -1}\alpha_j^n\}$ is a set of solutions to $L(f) = \overline{0}$, and it remains to prove that $\bigcup_{j=1}^k \{\alpha_j^n, n\alpha_j^n,\ldots,n^{m_j -1}\alpha_j^n\}$ is a linearly independent set. For that, let us use the same technique as in the proof of Theorem \ref{drbase}. 
Setting $h := d-1$, we need to show that the following matrix is invertible:  
\begin{equation*}
    \begin{bmatrix}
    0&0&...&1&...&0&0&...&1\\
    \alpha_1&\alpha_1&...&\alpha_1&...&\alpha_k&\alpha_k&...&\alpha_k\\
    2^{m_1 - 1}\alpha_1^2&2^{m_1 - 2}\alpha_1^2&...&\alpha_1^2&...&2^{m_k - 1}\alpha_k^2&2^{m_k - 2}\alpha_k^2&...&\alpha_k^2\\
    \vdots&\vdots&\ddots&\vdots&\ddots&\vdots&\vdots&\ddots&\vdots\\
    h^{m_1 - 1}\alpha_1^{h}&h^{m_1 - 2}\alpha_1^{h}&...&\alpha_1^{h}&...&h^{m_k - 1}\alpha_k^{h}&h^{m_1 - 2}\alpha_k^{h}&...&\alpha_k^{h}
    \end{bmatrix}
\end{equation*}

By \cite[Appendix 3.1]{liu}, this matrix has determinant equal to
\begin{equation*}
    [\prod_{1 \leq i \leq k} \alpha_i^{\binom{m_i}{2}}][\prod_{1 \leq i < j \leq k} (\alpha_i - \alpha_j)^{m_i m_j}].
\end{equation*}
Since $c_0 \neq 0$, all the roots of $p_L(x)$ are non-zero, and thus the determinant is non-zero. Therefore, the matrix is invertible and $\bigcup_{j=1}^k \{\alpha_j^n, n\alpha_j^n,...,n^{m_j -1}\alpha_j^n\}$ is a linearly independent set.
\end{proof}

Note as well that it is possible to build an analogue to Theorem \ref{finddr} for the general case. The method is precisely the same.

\chapter{Formal Power Series}\label{ch2}

In this chapter, we will introduce the algebra of complex formal power series. We begin this study by presenting Cauchy's algebra as the environment in which we'll work to develop this theory. We then present a topology to it, which allows us to develop a sense of convergence. After that, we define formal power series and see them as elements of our underlying environment. Finally, we look at derivatives, logarithms and exponentials in this context, as well as an interesting result about rational functions.
%We will define  derivatives, logarithms and exponentials work on them. 
For further reading, we recommend \cite[Ch. 13]{wagner}.

\section{The Cauchy Algebra}

Before providing the definition of Cauchy's algebra, let us recall the definition of algebra.

\begin{defi}
    An algebra $\mathcal{A}$ over a field $\mathbb{K}$ is a vector space equipped with a bilinear product.
    If the product is associative (respectively, commutative), we say that $\mathcal{A}$ is an associative 
    (respectively, commutative) algebra.
\end{defi}

\begin{defi}
    The set of complex sequences $\C^\N$ together with the operations
    \begin{itemize}
        \item[ ] $(F+G)(n) = F(n)+G(n)$
        \item[ ] $(F \cdot G)(n) = \sum_{k=0}^n F(k)G(n-k)$
        \item[ ] $(\lambda \cdot F)(n) = \lambda F(n)$
    \end{itemize}
    ($F,G \in \C^\N$ and $\lambda \in \C$) is an associative and commutative $\C$-algebra, called Cauchy's algebra. 
    The additive identity of Cauchy's algebra is $\overline{0}(n) \equiv 0$; the sequence $\overline{1}(n) = \delta_{0,n}$
    is a multiplicative identity. Naturally, powers of $F \in \C^{\N}$ are given by $F^0 = \overline{1}$ and $F^{n} = F^{n-1} \cdot F$.
\end{defi}

\begin{prop}
 Let $F, G \in \C^{\N}$. If $F,G \neq \overline{0}$, then $F \cdot G \neq \overline{0}$.
 In other words, $\C^{\N}$ is an integral domain. 
\end{prop}
\begin{proof}
    Let $r = \min\{k ; F(k) \neq 0\}$, ${s = \min\{k ; G(k) \neq 0\}}$. Then 
    $$(F \cdot G)(r+s) = \sum_{k=0}^{r+s} F(k)G(r+s-k) = F(r)G(s) \neq 0,$$ 
    as if $k < r$, then $F(k) = 0$, and if $k > r$, $r+s-k < s$ and $G(k) = 0$. Hence, 
    $F \cdot G \neq \overline{0}$.
\end{proof}
\newpage
\begin{prop}
    A function $F \in \C^{\N}$ has a multiplicative inverse if and only if $F(0) \neq 0$.
\end{prop}
\begin{proof}
    \emph{Necessity: } Suppose that $F$ has a multiplicative inverse $F^{-1}$. 
    Then 
    $$F^{-1}(0)F(0)=(F^{-1} \cdot F)(0)= \overline{1}(0) = 1,$$ 
    and thus $F(0) \neq 0$.

    \emph{Sufficiency: }Define the function $G \in \C^\N$ inductively by $(i)\; G(0) = 1/F(0)$, and \newline ${(ii)\; G(n+1) = -\frac{1}{F(0)} \sum_{k=1}^{n+1} F(k)G(n+1-k)}$. %Notice how this coincides with Definition \ref{141}, and $G$ is a solution to
    By definition, we have 
    \begin{equation*}
        F(0)G(n+1) + F(1)G(n) + \ldots + F(n)G(1) + F(n+1)G(0) = 0.
    \end{equation*}
    But the left hand side of this equation is exactly $(F \cdot G)(n+1)$, and thus, 
    $F \cdot G = \delta_{0,n} = \overline{1}$.
\end{proof}

It is also easy to see that if a function $F$ has a multiplicative inverse $F^{-1}$, then $F^r$, ${r \in \N \setminus \{0\}}$, has a multiplicative inverse, and $(F^r)^{-1} = (F^{-1})^r$. We write $F^{-r}$ for $(F^r)^{-1}$.

Consider now the set $\C^\N_1 := \{ F \in \C^\N ; F(0) = 1 \}$. This set has some properties that are analogous to those of the set $\R_+$ of positive real numbers. The following theorems illustrate some of those.

\begin{theorem}
    Let $F \in \C^\N_1$ and $r \in \N \setminus \{0\}$. Then there exists a unique $G \in \C^\N_1$, denoted $F^{\frac{1}{r}}$, such that $G^r = F$.
\end{theorem}
\begin{proof}
    We need to find $G \in \C^{\N}$ such that
    \begin{equation*}
        F(n) = \underset{n_i \geq 0}{\sum_{n_1 + \ldots + n_r = n}} G(n_1) \cdots G(n_r)
    \end{equation*}
    for all $n \geq 0$. We define $G$ recursively. For $n=0$, set $G(0) = 1$, and for $n=1$, $G(1) = \frac{F(1)}{r}$. 
    More generally, we can separate the case where $n_i = n$ for some $i=1, \ldots, r$ in the equation above, 
    and thus obtain a recursive formula for $G$:
    \begin{equation*}
        G(n) = \frac{1}{r} \Bigg[ F(n) - \underset{0 \leq n_i < n}{\sum_{n_1 + \ldots + n_r = n}} G(n_1) \cdots G(n_r) \Bigg].
    \end{equation*}
    The uniqueness of $G$ is now clear.
\end{proof}

\begin{corollary}
    Let $F \in \C^\N_1$, $r \in \N \setminus \{0\}$, and $q \in \Z$, there exists a unique $G \in \C^\N_1$, denoted $F^\frac{q}{r}$, such that $G^r = F^q$.
\end{corollary}
\begin{proof}
    Replace $F$ with $F^q$ in the proof of the previous theorem.
\end{proof}

\begin{theorem}
    Let $F,G \in \R^\N$, and $r \in \N \setminus \{0\}$. Suppose $F^r = G^r$. If $r$ is odd, then $F = G$, and if $r$ is even, then $F = G$ or $F = -G$.
\end{theorem}
\begin{proof}
    If either $F$ or $G$ is $\overline{0}$, then this is trivially satisfied. Otherwise, let $\xi_r = e^\frac{2\pi i}{r}$ be the primitive $r$-th root of unity. Then we have $\prod_{k=1}^r (F - \xi^k_r G) = F^r - G^r = \overline{0}$. If $\xi^k_r \notin \R$, then $F - \xi^k_r G \neq \overline{0}$, as $F$ and $G$  are both real-valued and non-null. If $r$ is odd, then $\xi^k_r$ is real if and only if $k = r$, and $F = G$. If $r$ is even, then $\xi^k_r$ is real if and only if $k = r$ or $k = \frac{r}{2}$, 
    and then $F = G$ or $F = -G$.
\end{proof}

\section{Power Series and Convergence}

In this section, we will provide a proper definition of formal power series.
We first topologize Cauchy's algebra $\C^{\N}$. 
%, as well as study convergence.

\begin{defi}
    The order function $o : \C^\N \to \N \cup \{\infty\}$ is defined as $o(\overline{0}) := \infty$ and $o(F) := \min\{ k ; F(k) \neq 0 \}$ when $F \neq \overline{0}$.
\end{defi}

The order function has the following properties.

\begin{prop}\label{227}
    For $F,G \in \C^\N$ and $\lambda \in \C \setminus \{0\}$,
    \begin{itemize}
        \item[ ] $\displaystyle o(F\cdot G) = o(F) + o(G) $
        \item[ ] $\displaystyle o(\lambda F) = o(F)$
        \item[ ] $\displaystyle o(F+G) \geq \min\{ o(F), o(G) \}$, and the equality holds whenever $o(F) \neq o(G)$.
    \end{itemize}
\end{prop}
\begin{proof}
If $F = \overline{0}$, then all identities trivially hold.
Thus we may assume that $F \neq \overline{0}$ and $G \neq \overline{0}$.
Then, we have 
\begin{align*}
 o(F \cdot G) &= \min\{ n ; \sum_{k=0}^n F(k)G(n-k) \neq 0 \} = \min\{ n = j+k ; F(j)G(k) \neq 0 \} \\      
&=\min\{ j ; F(j) \neq 0 \} + \min\{ k ; G(k) \neq 0 \} = o(F) + o(G);\\
o(\lambda F) &= \min\{ n ; \lambda F(n) \neq 0 \} = \min\{ n ; F(n) \neq 0 \} = o(F); \\
o(F + G) &= \min\{ n ; F(n) + G(n) \neq 0 \} \geq \min\{ n ; F(n) \neq 0 \text{ or } G(n) \neq 0 \} = \min\{ o(F), o(G) \}.
\end{align*}

Consider now the case when $o(F) \neq o(G)$. Without loss of generality assume that ${o(F) < o(G)}$. 
Then $o(F+G) = o(F) =: n$, as $F(n) + G(n) = F(n) + 0 = F(n)$ and $F(k) + G(k) = 0+0 = 0$ for all $k < n$.
\end{proof}

Now we define an absolute value on $\C^\N$.

\begin{defi}\label{228}
    Let $F \in \C^\N$. Then we define the absolute value of $F$ as $|F| := 2^{-o(F)}$, where $2^{-\infty} = 0$.
\end{defi}

\begin{prop}\label{229}
    The absolute value function has the following properties:
    \begin{enumerate}
        \item $|F| = 0$ if and only if $F = \overline{0}$;
        \item $|F \cdot G| = |F| \cdot |G|$;
        \item $|\lambda F| = |F|$ for all $\lambda \neq 0$;
        \item $|F+G| \leq \max\{|F|,|G|\}$, with the equality holding when $|F| \neq |G|$.
    \end{enumerate}
\end{prop}
\begin{proof}
    Immediate from Proposition~\ref{227}.
\end{proof}

\begin{defi}\label{2210}
    We define a distance mapping on $\C^\N$ as $d(F,G) := |F-G|$.
\end{defi}
\newpage
\begin{prop}
    Let $F,G,H \in \C^\N$. The distance mapping has the following properties:
    \begin{enumerate}
        \item $d(F,G) = 0$ if and only if $F = G$;
        \item $d(F,G) = d(G,F)$;
        \item $d(F,H) \leq d(F,G) + d(G,H)$;
        \item $d(F,H) \leq \max\{ d(F,G),d(G,H) \}$, with the equality holding when $d(F,G) \neq d(G,H)$.
    \end{enumerate}
\end{prop}
\begin{proof}
    Immediate from Proposition~\ref{229}.
\end{proof}

With this metric defined, we are prepared to deal with convergence. 
In what follows, $(F_n)_{n \in \N}$ denotes a sequence in $\C^\N$ and $F \in \C^\N$.

\begin{defi}
    We say that the sequence $(F_n)_{n \in \N}$ converges to $F$, and write $\displaystyle \lim_{n \to \infty} F_n = F$, if $\displaystyle \lim_{n \to \infty} d(F_n,F) = 0$. Since $d(F_n,F)$ tends to $0$ if and only if $o(F_n - F)$ tends to infinity, 
    we can define this limit in an equivalent manner by saying that the sequence $(F_n)_{n \in \N}$ converges to $F$ if, 
    for all $j \in \N$, there exists an $n_j \in \N$ such that $n \geq n_j$ implies $F_n(i) = F(i)$ for all $0 \leq i \leq j$.
\end{defi}
 
We can extend this definition to series.

\begin{defi}
    Denote $S_n := F_0 + \ldots + F_n$. If there exists $S \in \C^\N$ such that $\displaystyle \lim_{n \to \infty} S_n = S$, then we say that the series $\sum_{n \in \N} F_n$ converges to $S$, and write $\sum_{n \in \N} F_n = S$.
\end{defi}

\begin{theorem}
\label{criteria of convergence}
    The series $\sum_{n \in \N} F_n$ converges if and only if $\displaystyle \lim_{n \to \infty} F_n = \overline{0}$.
\end{theorem}
\begin{proof}
    \emph{Necessity: }Suppose that $\sum_{n \in \N}F_n = S$. Then for all $j \in \N$ there exists $n_j \in \N$ such that $n \geq n_j$ implies $S_n(j) = S(j)$. If $n \geq n_j + 1$, then ${F_n(j) = S_{n}(j) - S_{n-1}(j) = S(j) - S(j) = 0}$, and thus $\displaystyle \lim_{n \to \infty}F_n = \overline{0}$.

\medskip

    \emph{Sufficiency: }Suppose that $\displaystyle \lim_{n \to \infty}F_n = \overline{0}$. Then for each $j \in \N$ there exists a minimal $n_j \in \N$ such that $n \geq n_j$ implies $F_n(j) = 0$. Now define $S \in \C^{\N}$ by 
    $S(j) := F_0(j) + \ldots + F_{n_j - 1}(j)$. Clearly, for each $j \in \N$, we have $S(j)=S_n(j)$ for all $n \geq n_j$.
    It follows that $\sum_{n \in \N} F_n=S$.
\end{proof}

Note that this theorem is not true in the case of usual analysis. For instance,
the series $\sum_{n \in \mathbb{N}} \frac{1}{n}$ does not converge, even though 
${\displaystyle \lim_{n \to \infty}} \frac{1}{n}=0$. 

We may now start to think about what formal power series are \emph{per se}. For that, let us look at a particular function.

\begin{defi}
    The "indeterminate" function $X : \N \to \C$ is defined as $X(n) = \delta_{1,n}$. It is easy to check that $X^r(n) = \delta_{r,n}$.
\end{defi}

\begin{theorem}
\label{formal power series}
Let $(a_n)_{n \in \N}$ be a sequence of complex numbers. Then the series $\sum_{n \in N} a_n X^n$ converges.
Furthermore, if $(b_n)_{n \in \N}$ is another sequence of complex numbers, then
$$\sum_{n \in \N} a_n X^n=\sum_{n \in \N} b_n X^n \, \text{ implies } \, a_n=b_n \text{ for all } n \in \N.$$
\end{theorem}
\begin{proof}
Since $\displaystyle \lim_{n \to \infty} a_nX^n=\overline{0}$, it follows from Theorem~\ref{criteria of convergence} that 
$\sum_{n \in N} a_n X^n$ converges. Moreover, it is easy to see that 
$S=\sum_{n \in N} a_n X^n$ where $S(n)=a_n$ for all $n \in \N$. 
\end{proof}

\begin{defi}
    A series of the form $\sum_{n \in \N} a_n X^n$, with each $a_n \in \C$, is called a formal power series with coefficients in $\C$.
    We denote by $\C[[X]]$ the set of all formal power series with coefficients in $\C$.
\end{defi}

%\begin{defi}
%    A sequence $(S_n)_{n \in \N}$ in $\C^\N$ is said to converge pointwise to $S \in \C^\N$ if, for all $x \in \N$, it is the case that the sequence $(S_n(x))$ of complex numbers converges to $S(x)$ in $\C$.
%\end{defi}

\begin{theorem}
\label{representation as fps}
    For all $F \in \C^\N$,
    \begin{equation}\label{fnxn}
        F = \sum_{n \in \N} F(n)X^n.
    \end{equation}
Moreover, this representation of $F$ as a formal power series is unique.
\end{theorem}
\begin{proof}
%    Consider the sequence of functions $S_n = F(0)X^0 + \ldots + F(n)X^n$. Interpret equation \ref{fnxn} as saying that $(S_n)_{n \in \N}$ converges pointwise to $F$. For each $k \in \N$, the sequence $(F(n)X^n(k))_{n \in \N}$ has a unique non-zero term, that is, $F(k)X^k(k)$. From that we get that $S_n(k) = 0$ for all $n < k$, and $S_n(k) = F(k)$ for $n \geq k$. In other words, for every $k \in \N$ it is the case that $S_n(k) = F(k)$ for sufficiently large $n$, and we conclude that $(S_n)_{n \in \N}$ converges pointwise to $F$.
This follows immediately from Theorem~\ref{formal power series}.
\end{proof}

It follows from the theorem above that we can identify the set $\C^\N$ with the set $\C[[X]]$ of formal power series with coefficients in $\C$. We use the name polynomials for those power series  ${F(X)=\sum_{n \in \N} F(n)X^n \in \C[[X]]}$ for which $F(n)=0$ for all but finitely many $n$. We denote the set of polynomials by $\C[X]$.

We can now rewrite the operations of Cauchy's algebra in the notation of formal power series:

\begin{itemize}
    \item[ ] $\displaystyle \sum_{n \in \N} a_n X^n + \sum_{n \in \N} b_n X^n = \sum_{n \in \N} (a_n + b_n) X^n$
    \item[ ] $\displaystyle \sum_{n \in \N} a_n X^n \cdot \sum_{n \in \N} b_n X^n = \sum_{n \in \N} \Big( \sum_{k=0}^n a_k b_{n-k} \Big) X^n$
    \item[ ] $\displaystyle \lambda \sum_{n \in \N} a_n X^n = \sum_{n \in \N} (\lambda a_n) X^n$
\end{itemize}

One more equivalence that is not true in usual analysis, besides the one in Theorem \ref{criteria of convergence}, but holds in the formal power series case, is the equivalence of convergence and summability.
\begin{defi}
    A sequence $(a_n)_{n \in \N}$ is summable when (\emph{i}) $\sum_{n \in \N}a_n$ converges, and (\emph{ii}) $\sum_{n \in \N}a_{\sigma(n)}$ converges for any permutation $\sigma$ of $\N$ and $\sum_{n \in \N} a_{\sigma(n)} = \sum_{n \in \N} a_n$.
\end{defi}
%\newpage
\begin{theorem}
    A sequence $(F_n)_{n \in \N}$ in $\C^\N$ is convergent if and only if it is summable.
\end{theorem}
\begin{proof}
\emph{Sufficiency: }True by definition.

\emph{Necessity: }Consider the following matrices.
    \begin{equation*}
        M = \left[\begin{array}{cccc} 
        F_0(0) & F_0(1) & F_0(2) & \cdots \\
        F_1(0) & F_1(1) & F_1(2) & \cdots \\
        F_2(0) & F_2(1) & F_2(2) & \cdots \\
        \vdots & \vdots & \vdots & \ddots
        \end{array}\right]
        \quad\quad M^\sigma = \left[\begin{array}{cccc} 
        F_{\sigma(0)}(0) & F_{\sigma(0)}(1) & F_{\sigma(0)}(2) & \cdots \\
        F_{\sigma(1)}(0) & F_{\sigma(1)}(1) & F_{\sigma(1)}(2) & \cdots \\
        F_{\sigma(2)}(0) & F_{\sigma(2)}(1) & F_{\sigma(2)}(2) & \cdots \\
        \vdots & \vdots & \vdots & \ddots
        \end{array}\right]
    \end{equation*}
    Since $\displaystyle \lim_{n \to \infty} F_n = \overline{0}$, all columns of $M$ have only a finite number of non-zero entries. Since $M^\sigma$ comes from $M$ via a permutation of its rows, all columns of $M^\sigma$ also have only finitely many non-zero entries. Moreover,the column sums of the two matrices are precisely the same. We then have that 
    $\displaystyle \lim_{n \to \infty} F_{\sigma(n)} = \overline{0}$, and ${\sum_{n \in \N} F_n = \sum_{n \in \N} F_{\sigma(n)}}$.
\end{proof}
\newpage
\begin{prop}
    Let $(F_n)_{n \in \N}, (G_n)_{n \in \N}$ be summable sequences in $\C^\N$, and $H \in \C^\N$. Then $(F_n + G_n)_{n \in \N}$ and $(H \cdot F_n)_{n \in \N}$ are summable, with
    \begin{itemize}
        \item[ ] $\displaystyle \sum_{n \in \N} (F_n + G_n) = \sum_{n \in \N} F_n + \sum_{n \in \N} G_n$;
        \item[ ] $\displaystyle \sum_{n \in \N} (H \cdot F_n) = H \cdot \sum_{n \in \N} F_n$.
    \end{itemize}
\end{prop}
\begin{proof}
    We will not be demonstrating summability in this proof, as it is straightforward. 
    Let us then show that the equations hold.
    \begin{align*}
        \sum_{n \in \N} (F_n + G_n) &= \sum_{j \in \N} \left[ \sum_{n \in \N} F_n(j) + G_n(j) \right] X^j\\
        &= \sum_{j \in \N} \left[\sum_{n \in \N} F_n(j)\right] X^j  + \sum_{j \in \N}\left[\sum_{n \in \N} G_n(j)\right] X^j  = 
        \sum_{n \in \N} F_n + \sum_{n \in \N} G_n;\\
        \sum_{n \in \N} (H \cdot F_n) &= \sum_{j \in \N} \sum_{n \in \N}  \left[ \sum_{k=0}^j H(k)F_n(j-k) \right] X^j \\&= \sum_{j \in \N} \sum_{k=0}^j H(k) \left[ \sum_{n \in \N} F_n(j-k) \right] X^j = H \cdot \sum_{n \in \N} F_n.
    \end{align*}
\end{proof}

Consider now the set $\C^\N_0 := \{ F \in \C^\N ; F(0) = 0 \}$. This set can also be defined as the set of those elements 
$F \in \C^{\N}$ such that $|F| < 1$, as $F(0) = 0$ if and only if $o(F) > 0$, which is equivalent to saying $|F| = 2^{-o(F)} < 1$. Just like in classical analysis, we can define the geometric series on this set.

\begin{theorem}
    If $F \in \C^\N_0$, then $(F^n)_{n \in \N}$ is summable and the geometric series $\sum_{n \in \N} F^n$ converges to 
    $(\overline{1}-F)^{-1}$. In particular, $\sum_{n \in \N} X^n = (\overline{1}-X)^{-1}$.
\end{theorem}
\begin{proof}
    By Proposition \ref{229}.ii), $|F^n| = |F|^n$, and then $|F| < 1$ implies $\displaystyle \lim_{n \to \infty} |F^n| = 0$. This gives us $\displaystyle \lim_{n \to \infty} F^n = \overline{0}$, and thus $(F^n)_{n \in \N}$ is summable. Moreover,
    \begin{align*}
        (\overline{1}-F) \cdot \sum_{n \in \N} F^n &= \sum_{n \in \N} (\overline{1}-F)\cdot F^n = \sum_{n \in \N} (F^n - F^{n-1})\\ &= \lim_{n \to \infty} [(F^0 - F^1) + \ldots + (F^n - F^{n-1})] = \lim_{n \to \infty} (\overline{1}-F^{n-1}) = \overline{1}.
    \end{align*}
    Therefore, $\sum_{n \in \N} F^n = (\overline{1} - F)^{-1}$.
\end{proof}

\section{Formal Derivatives}

\begin{defi}
    The derivative $D : \C^\N \to \C^\N$ is defined by ${DF(n) := (n+1)F(n+1)}$ for all $n \in \N$. 
    In power series notation, this translates to
    \begin{equation*}
        D \sum_{n \in \N} F(n)X^n = \sum_{n \in \N} (n+1)F(n+1)X^n.
    \end{equation*}
    We sometimes write $F'$ instead of $DF$.
    Non-negative integral powers of $D$ are defined recursively by $D^0F = F$, $D^1F = DF$, and $D^{n}F = D(D^{n-1}F)$.
\end{defi}

Like the usual derivative operator, the formal derivative has the following properties.

\begin{prop}\label{232}
    Let $F, G \in \C^\N$, $\lambda \in \C$, $n \in \N \setminus \{0\}$, and $r \in \Q$. Then
    \begin{enumerate}
        \item $D(F+G) = DF+DG$;
        \item $D(\lambda F) = \lambda DF$;
        \item $D(F \cdot G) = DF \cdot G + F \cdot DG$;
        \item $D(F^n) = nF^{n-1} \cdot DF$;
        \item If $F$ is invertible, then $D(F^{-n}) = -n F^{-n-1} \cdot DF$;
        \item If $F \in \C^\N_1$, then $D(F^r) = rF^{r-1} \cdot DF$;
        \item If $F, G \in \C^\N_0$ and $DF = DG$, then $F = G$;
        \item $DF = \overline{0}$ if and only if $F = \lambda \cdot \overline{1}$ for some $\lambda \in \C$.
    \end{enumerate}
\end{prop}
\begin{proof}
    \begin{enumerate}
        \item $\displaystyle D\sum_{n \in \N} (F+G)(n) X^n = \sum_{n \in \N} (n+1)(F+G)(n+1) X^n$
        
        $\displaystyle\quad\quad\quad\quad\quad\quad\quad\quad\quad\quad\quad\quad\; = \sum_{n \in \N} (n+1)[F(n+1) + G(n+1)] X^n = DF + DG$;
        \item $D \displaystyle \sum_{n \in \N} \lambda F(n)X^n = \sum_{n \in \N} \lambda (n+1)F(n+1)X^n = \lambda \sum_{n \in \N} (n+1)F(n+1)X^n = \lambda DF$;
        \item $\displaystyle DF \cdot G = \sum_{n \in \N} (n+1)F(n+1) X^n \cdot \sum_{n \in \N} G(n)X^n = \sum_{n \in \N} \left[ \sum_{k=0}^n (k+1)F(k+1)G(n-k) \right]X^n$.

        Analogously, $\displaystyle F \cdot DG = \sum_{n \in \N} \left[ \sum_{k=0}^n (n-k+1) F(k)G(n-k+1) \right]X^n$. We then have
        \begin{align*}
            (DF \cdot G + F \cdot DG)(n) &= (n+1)F(0)G(n+1) + \sum_{k=1}^n (n-k+1)F(k)G(n-k+1) \\&\quad\quad+ \sum_{k=0}^{n-1} (k+1) F(k+1)G(n-k) + (n+1) F(n+1)G(0) \\&= (n+1)F(0)G(n+1) + \sum_{k=1}^{n} (n-k+1) F(k)G(n-k+1) \\&\quad\quad+ \sum_{k=1}^n k F(k)G(n-k+1) + (n+1) F(n+1)G(0) \\&= (n+1)F(0)G(n+1) + \sum_{k=1}^{n} (n-k+1+k) F(k)G(n-k+1) \\&\quad\quad+ (n+1) F(n+1)G(0) = \sum_{k=0}^{n+1} (n+1) F(k)G(n-k+1) \\&= (n+1)(F \cdot G)(n+1) = [D(F \cdot G)](n);
        \end{align*}
        \item The proof is by induction in $n$. For $n=1$, we have 
        $D(F^1) = DF = 1 \cdot\overline{1} \cdot DF$, as claimed. 
        Suppose that the equality holds for $n \leq k$. Then 
        \begin{equation*}
            D(F^{k+1}) = D(F^k) \cdot F + F^k \cdot DF = kF^{k-1} \cdot DF \cdot F + F^k \cdot DF = (k+1)F^k \cdot DF.
        \end{equation*}
        \item From $F^{-n} \cdot F^n = \overline{1}$, we get that $D(F^{-n}) \cdot F^n + F^{-n} \cdot nF^{n-1} \cdot DF = \overline{0}$. We then have $D(F^{-n}) \cdot F^n = -nF^{-1} \cdot DF$, which implies $D(F^n) = -nF^{-n-1} \cdot DF$.
        \item Suppose $r = p/q$. We have $D[(F^{r})^{q}] =  q(F^{r})^{q-1} \cdot D(F^r)$. But we also have that ${D[(F^{r})^{q}] = D(F^p) = pF^{p-1} \cdot DF}$. Consequently, ${q(F^r)^{q-1} D(F^r) = pF^{p-1} DF}$, and thus $D(F^r) = \frac{p}{q} F^{\frac{p}{q}-1} \cdot DF = rF^{r-1} \cdot DF$.
        \item If $DF = DG$, then $F(n) = G(n)$ for all $n \geq 1$. Since $F, G \in \C^\N_0$, $F(0) = G(0) = 0$, and we have that $F(n) = G(n)$ for all $n \in \N$. We then conclude that $F = G$.
        \item \emph{Necessity: }If $DF = \overline{0}$, then $F(n) = 0$ for all $n \geq 1$. Naturally, $F = \lambda.\overline{1}$ for some $\lambda \in \C$.

        \emph{Sufficiency: }$D(\lambda \cdot \overline{1}) = \lambda \cdot D\overline{1} = \lambda \cdot\overline{0} = \overline{0}$.
    \end{enumerate}
\end{proof}

We can now present the formal power series equivalent of Maclaurin (or Taylor) series.

\begin{theorem}\label{macl}
    Let $F \in \C^\N$. Then
    \begin{equation*}
        F = \sum_{n \in \N} \frac{1}{n!}(D^nF)(0) X^n.
    \end{equation*}
\end{theorem}
\begin{proof}
    It suffices to prove that $F(n) = \frac{1}{n!} (D^nF)(0)$ for all $n \in \N$. Applying the definition of derivative multiple times, we have that $(D^nF)(k) = (k+1)(k+2)\cdots(k+n)F(k+n)$, and thus $(D^nF)(0) = n!F(n)$.
\end{proof}

\begin{corollary}\label{234}
    Let $r \in \Q$ and $\lambda \in \C$. Then
    \begin{equation*}
        (\overline{1} + \lambda X)^r = \sum_{n \in \N} \binom{r}{n} \lambda^n X^n.
    \end{equation*}
\end{corollary}
\begin{proof}
    We have $[D(\overline{1} + \lambda X)](0) = \lambda$. Then, by applying Proposition \ref{232}.vi) $n$ times, we get $[D^n(\overline{1} + \lambda X)^r](0) = r^{\underline{n}} \lambda^n$. Finally, by Theorem \ref{macl}, 
    $$\displaystyle(\overline{1} + \lambda X)^r = \sum_{n \in \N} \frac{1}{n!} r^{\underline{n}} \lambda^n X^n = \sum_{n \in \N} \binom{r}{n} \lambda^n X^n.$$
\end{proof}

\begin{theorem}
    Let $(F_n)_{n \in \N}$ be a summable sequence in $\C^\N$. Then $(DF_n)_{n \in \N}$ is summable and
    \begin{equation}\label{235}
        D \sum_{n \in \N} F_n = \sum_{n \in \N} DF_n.
    \end{equation}
\end{theorem}
\begin{proof}
    Since $(F_n)_{n \in \N}$ is summable, $\displaystyle \lim_{n \to \infty} F_n = \overline{0}$. Therefore, for all $j \in \N$ there exists $n_j \in \N$ such that $n \geq n_j$ implies $F_n(j+1) = 0$, and also $(DF_n)(j) = (j+1)F_n(j+1) = 0$. It follows that $\displaystyle \lim_{n \to \infty} DF_n = \overline{0}$, and thus $(DF_n)_{n \in \N}$ is summable. 
    Moreover,
    \begin{align*}
        \left( D \sum_{n \in \N} F_n \right)(j) &= (j+1)\left( \sum_{n \in \N} F_n \right)(j+1) = (j+1) \sum_{n = 0}^{n_j} F_n(j+1) = \sum_{n=0}^{n_j} (j+1)F_n(j+1) \\&= \sum_{n \in \N} (j+1)F_n(j+1) = \sum_{n \in \N} (DF_n)(j) = \left( \sum_{n \in \N} DF_n \right)(j)
    \end{align*}
    and equation \ref{235} holds.
\end{proof}

\section{Formal Exponential and Logarithm}

Those readers familiar with complex analysis might 
recall the following power series expansions for the exponential and logarithmic functions:
\begin{itemize}
    \item[ ] $\displaystyle e^z = \sum_{n \in \N} \frac{1}{n!}z^n$
    \item[ ] $\displaystyle \log(1+z) = \sum_{n \geq 1} \frac{(-1)^{n+1}}{n} z^n$
\end{itemize}

In this section, we will present their equivalent forms in the context of formal power series.

\begin{defi}
    Let $F \in \C^\N$ be such that $|F| < 1$. Then the formal exponential function $\exp : \C^\N_0 \to \C^\N_1$ is defined as $\displaystyle \exp(F) := \sum_{n \in \N} \frac{1}{n!} F^n$.
\end{defi}

\begin{defi}
    Let $F \in \C^\N$ be such that $|F| < 1$. Then the formal logarithmic function $\log : \C^\N_1 \to \C^\N_0$ is defined as $\displaystyle \log(\overline{1} + F) := \sum_{n \geq 1} \frac{(-1)^{n+1}}{n} F^n$.
\end{defi}

%Like in usual analysis, their derivatives are given by the following.

\begin{prop}
    For all $F \in \C^\N_0$ and $G \in \C^\N_1$,
    \begin{enumerate}
        \item $D(\exp(F)) = \exp(F) \cdot DF$;
        \item $D(\log(G)) = G^{-1} \cdot DG$.
    \end{enumerate}
\end{prop}
\begin{proof}
    \begin{enumerate}
        \item $\displaystyle D(\exp(F)) = D\left( \sum_{n \in \N} \frac{1}{n!} F^n \right) = \sum_{n \geq 1} \frac{1}{(n-1)!} F^{n-1} \cdot DF$
        
        $\displaystyle \quad\quad\quad\quad\quad\quad\quad\quad\, = \sum_{n \in \N} \frac{1}{n!} F^n \cdot DF = \exp(F) \cdot DF$;
        \item Suppose $G = \overline{1} + H$ for some $H \in \C^\N_0$. Then
        \begin{align*}
            D(\log(G)) &= D \sum_{n \geq 1} \frac{(-1)^{n+1}}{n} H^n = \sum_{n \geq 1} (-1)^{n+1}H^{n-1} \cdot DH \\&= \sum_{n \in \N} (-1)^n H^n \cdot DH = (\overline{1} + H)^{-1} \cdot DH = G^{-1} \cdot DG
        \end{align*}
        as $DH = DG$.
    \end{enumerate}
\end{proof}

The following logarithmic properties still stand for the formal logarithm.
%\newpage
\begin{prop}
    Let $F, G \in \C^\N_1$. Then
    \begin{enumerate}
        \item $\log(F \cdot G) = \log(F) + \log(G)$;
        \item $\log(F^m) = m\log(F)$ for all $m \in \Z$;
        \item $\log(F^r) = r\log(F)$ for all $r \in \Q$.
    \end{enumerate}
\end{prop}
\begin{proof}
    \begin{enumerate}
        \item Since $\log(F \cdot G)$ and $\log(F) + \log(G)$ are elements of $\C^\N_0$, it is sufficient to prove that $D(\log(F \cdot G)) = D(\log(F) + \log(G))$. This holds, as
        \begin{align*}
            D(\log(F \cdot G)) &= (F \cdot G)^{-1} \cdot D(F \cdot G) = (G^{-1} \cdot F^{-1}) \cdot (G \cdot DF + F \cdot DG) \\&= F^{-1} \cdot DF + G^{-1} \cdot DG = D(\log(F) + \log(G)).
        \end{align*}
        \item Again, it is sufficient to prove that their derivatives coincide. We have
        \begin{align*}
            D(\log(F^m)) = F^{-m} \cdot D(F^m) = F^{-m} \cdot mF^{m-1} \cdot DF = mF^{-1} \cdot DF = D(m\log(F)).
        \end{align*}
        \item Suppose $r = p/q$. Let $F^r = G$, so $F^p = G^q$ and $\log(F^p) = \log(G^q)$. By the previous item, this is the same as $p\log(F) = q\log(F^r)$, and thus $\log(F^r) = r\log(F)$.
    \end{enumerate}
\end{proof}

Let us now see the injectivity of the formal logarithm.

\begin{lemma}\label{log0}
    The equality $\log(G) = \overline{0}$ holds if and only if $G = \overline{1}$.
\end{lemma}
\begin{proof}
    \emph{Necessity: }Since $\log(G) = \overline{0}$, we have $G^{-1} \cdot DG = 0$. For obvious reasons, $G^{-1} \neq \overline{0}$, and hence $DG = \overline{0}$. It follows that $G = \lambda \cdot\overline{1}$ for some $\lambda \in \C$. 
    Since the domain of the logarithm function is $\C^\N_1$, $G \in \C^\N_1$, and consequently, $G = \overline{1}$.

    \emph{Sufficiency: }$\displaystyle \log(\overline{1} + \overline{0}) = \sum_{n \geq 1} \frac{(-1)^{n+1}}{n} \overline{0}^n = \overline{0}$.
\end{proof}

\begin{theorem}
    The formal logarithmic function is injective.
\end{theorem}
\begin{proof}
    Suppose $\log(F) = \log(G)$ for $F,G \in \C^\N_1$. Then $\log(F \cdot G^{-1}) = \log(F) - \log(G) = \overline{0}$. From Lemma \ref{log0}, $F \cdot G^{-1} = \overline{1}$, and hence $F = G$.
\end{proof}

The following theorem shows that the inverse relation between logarithms and exponentials still stands in the context of formal power series.

\begin{theorem}
    Let $F \in \C^\N_0$ and $G \in \C^\N_1$. Then
    \begin{enumerate}
        \item $\log(\exp(F)) = F$;
        \item $\exp(\log(G)) = G$.
    \end{enumerate}
\end{theorem}
\begin{proof}
    \begin{enumerate}
        \item Since the codomain of the formal logarithmic function is $\C^\N_0$, it suffices to show that $D(\log(\exp(F))) = DF$. This holds, as
        \begin{align*}
            D(\log(\exp(F))) = [\exp(F)]^{-1} \cdot D(\exp(F)) = [\exp(F)]^{-1} \cdot \exp(F) \cdot DF = DF.
        \end{align*}
        \item From the previous item, $\log(\exp(\log(G))) = \log(G)$, and from the injectivity of the logarithmic function, we get $\exp(\log(G)) = G$.
    \end{enumerate}
\end{proof}

Let us now see a generalization of Corollary \ref{234}.

\begin{theorem}
    Let $F \in \C^\N_0$. Then, for all $r \in \Q$,
    \begin{equation*}
        (\overline{1} + F)^r = \sum_{n \in \N} \binom{r}{n} F^n.
    \end{equation*}
\end{theorem}
\begin{proof}
    The sequence $\left( \binom{r}{n}F^n \right)_{n \in \N}$ is clearly summable, as $\displaystyle \lim_{n \to \infty}|F^n| = 0$. Denote by ${G = \sum_{n \in \N} \binom{r}{n}F^n}$. Then $DG = DF \cdot \sum_{n \geq 1} n \binom{r}{n} F^{n-1}$, and
    \begin{align*}
        (\overline{1} + F) \cdot DG &= DF \cdot \sum_{n \geq 1} n \binom{r}{n} F^{n-1} + DF \cdot \sum_{n \geq 1} n \binom{r}{n} F^n \\&= DF \cdot \sum_{n \geq 1} n \binom{r}{n} F^{n-1} + DF \cdot \sum_{n \geq 2} (n-1) \binom{r}{n-1} F^{n-1} \\&= rDF + DF \cdot \sum_{n \geq 2} \left[ n \binom{r}{n} + (n-1) \binom{r}{n-1} \right] F^{n-1} \\&= rDF + rDF \cdot \sum_{n \geq 2} \binom{r}{n-1} F^{n-1} \\&= rDF \cdot \sum_{n \in \N} \binom{r}{n} F^n = rG \cdot DF.
    \end{align*}
    Multiplying both sides of the equation $(\overline{1} + F) \cdot DG = rG \cdot DF$ by $G^{-1} \cdot (\overline{1} + F)^{-1}$ yields $G^{-1} \cdot DG = r(\overline{1} + F)^{-1} \cdot DF$. It follows that
    \begin{align*}
        D(\log(G)) &= G^{-1} \cdot DG = r(\overline{1} + F)^{-1} \cdot DF = r(\overline{1} + F)^{-1} \cdot D(\overline{1} + F) \\&= D(r\log(\overline{1} + F)) = D(\log(\overline{1} + F)^r).
    \end{align*}
    Since $\log(G)$ and $\log(\overline{1} + F)^r$ are both in $\C^\N_0$, we may conclude that $\log(G) = \log(\overline{1} + F)^r$. The injectivity of the logarithmic function then guarantees $G = (\overline{1} + F)^r$.
\end{proof}

We conclude this section by using the theorem above to find another formula for the Fibonacci numbers.

\begin{exa}
    Let $F \in \C^\N$ be the function such that $F(n)$ is the $n$-th Fibonacci number. Then $F(0) = F(1) = 1$, and $F(n+2) = F(n) + F(n+1)$ for all $n \in \N$. We have
    \begin{align*}
        F &= \overline{1} + X + \sum_{n \geq 2} [F(n-2) + F(n-1)] X^n = \overline{1} + X + \sum_{n \geq 0} [F(n) + F(n+1)] X^{n+2} \\&= \overline{1} + X + X \sum_{n \geq 1} F(n)X^n + X^2 \sum_{n \geq 0} F(n) X^n = \overline{1} + X + X(F - \overline{1}) + X^2 F = \overline{1} + XF + X^2 F.
    \end{align*}
    This implies $\overline{1} = F(\overline{1} - X - X^2)$, from where $F = (\overline{1} - X - X^2)^{-1}$. Applying the 
    theorem above we get
    \begin{align*}
        F = \sum_{n \in \N} \binom{-1}{n} (-X-X^2)^{n} = \sum_{n \in \N} X^n(\overline{1} + X)^n = \sum_{n \in \N} \sum_{k = 0}^n \binom{n}{k} X^{n+k},
    \end{align*}
    which implies $\displaystyle F(n) = \sum_{k=0}^n \binom{n-k}{k}$.
\end{exa}

\section{The Fundamental Theorem of Rational Generating Functions}

In this section, we prove that a formal power series $\sum_{n=0}^{\infty} a_nX^n$ is a rational function if and only if the sequence of coefficients $(a_n)_{n \in \N}$ (eventually) satisfies a linear recurrence relation with constant coefficients. 

Recall that we denote by $E$ and $I$ the shift operator and the identity operator on $\C^{\N}$, respectively.

\begin{theorem}[Fundamental theorem of rational generating functions]
    Let $d$ be a positive integer, $c_0,\ldots,c_d$ be complex numbers such that $c_0 c_d \neq 0$, 
    $L = c_d E^d + \ldots +c_1 E + c_0 I$, and $p_L(X) = c_d X^d + \ldots + c_1 X + c_0 = c_d \prod_{i=1}^k (X - \alpha_i)^{m_i}$, with each $m_i > 0$, $m_1 + \ldots + m_k = d$, and $\alpha_1, \ldots, \alpha_k$ distinct complex numbers. 
    Consider the sets
    \begin{itemize}
        \item[ ] $N_L = \{ f : \N \to \C ; L(f)=\overline{0} \}$;
        \item[ ] $\displaystyle G_1 = \left\{ f : \N \to \C ; \sum_{n \in \N} f(n)X^n = \frac{q(X)}{X^d p_L(X^{-1})} \text{ for some }q \in \C[X], \deg q < d \right\}$;
        \item[ ] $\displaystyle G_2 = \left\{ f : \N \to \C ; \sum_{n \in \N} f(n)X^n = \sum_{i=1}^k \frac{g_i(X)}{(\overline{1} - \alpha_i X)^{m_i}} \text{ for some }g_i \in \C[X], \deg g_i < m_i \right\}$;
        \item[ ] $\displaystyle F = \left\{ f : \N \to \C ; f(n) = \sum_{i=1}^k \lambda_i(n) \alpha_i^n \text{ for some }\lambda_i \in \C[n], \deg \lambda_i < m_i \right\}$.
    \end{itemize}
    Then $N_L = G_1 = G_2 = F$.
\end{theorem}
\begin{proof}
    Firstly, note that these four sets are subspaces of the vector space $\C^\N$. Moreover, they are all $d$-dimensional, as
    \begin{enumerate}
        \item the map $f \mapsto (f(0),\ldots,f(d-1))$ is an isomorphism from $N_L$ to $\C^d$;
        \item if $q(X) = a_0 + \ldots + a_{d-1} X^{d-1}$, then the map $f \mapsto (a_0, \ldots, a_{d-1})$ is an isomorphism from $G_1$ to $\C^d$;
        \item if $g_i(X) = a_{i,0} + \ldots + a_{i, m_i - 1} X^{m_i - 1}$, then the map $f \mapsto (a_{1,0},\ldots,a_{1, m_1 - 1},a_{2,0}, \ldots, a_{k, m_k - 1})$ is an isomorphism from $G_2$ to $\C^d$;
        \item if $\lambda_i(n) = l_{i,0} + \ldots + l_{i, m_i - 1} n^{m_i - 1}$, then the map $f \mapsto (l_{1,0}, \ldots, l_{1, m_1 - 1}, l_{2,0}, \ldots, l_{k, m_k - 1})$ is an isomorphism from $F$ to $\C^d$.
    \end{enumerate}
    We already proved in Section~\ref{Recurrence Equations} that $N_L=F$.
    Therefore, in order to establish the remaining equalities, it suffices to show that 
    i) $G_1 \subseteq N_L$ and ii) $G_2 \subseteq G_1$. %and iii) $G_2 \subseteq F$.
    \begin{enumerate}
        \item If $f \in G_1$, then \begin{equation*}
            (c_d + \ldots + c_1 X^{d-1} + c_0 X^d) \sum_{n \in \N} f(n)X^n = q(X).
        \end{equation*}
        For all $n \in \N$, the coefficient of $X^{n+d}$ on the left-hand side of the equation above is equal to 
        $c_d f(n+d) + \ldots + c_1 f(n+1) + c_0 f(n)$, and on the right-hand side is equal to zero. It follows that 
        $L(f) = \overline{0}$ and $f \in N_L$.
        \item If $f \in G_2$, then
        \begin{equation*}
            \sum_{n \in \N} f(n) X^n = \sum_{i=1}^k \frac{g_i(X)}{(\overline{1} - \alpha_i X)^{m_i}}.
        \end{equation*}
        Adding the fractions on the right-hand side yields
        \begin{equation*}
            \sum_{n \in \N} f(n) X^n = \frac{c_d \sum_{j=1}^k g_j(X) \prod_{i \neq j} (\overline{1} - \alpha_i X)^{m_i}}{c_d \prod_{i = 1}^k (\overline{1} - \alpha_i X)^{m_i}}.
        \end{equation*}
        The denominator is equal to $X^d p_L(X^{-1})$. Moreover, since the degree of $g_j(X)$ is less than $m_j$, we have that the degree of the numerator is less than $d$, and thus we can take it as being our $q(X)$. Therefore, $f \in G_1$.
        %\item If $f \in G_2$, and $g_i(X) = a_{i,0} + \ldots + a_{i, m_i - 1} X^{m_i - 1}$, then
        %\begin{align*}
            %\sum_{n \in \N} f(n) X^n &= \sum_{i=1}^k \sum_{j=0}^{m_i - 1} a_{i,j} X^j (\overline{1} - \alpha_i X)^{-m_i} = \sum_{i=1}^k \sum_{j=0}^{m_i - 1} a_{i,j} \sum_{m \geq 0} \binom{-m_i}{m} (-\alpha_i)^m X^{m+j} \\
            %&=_{(n = m+j)} \sum_{i=1}^k \sum_{j=0}^{m_i - 1} a_{i,j} \sum_{n \geq j} \alpha_i^{-j} \binom{n - j + m_i - 1}{m_i - 1} \alpha_i^{n} X^n \\
            %&= \sum_{i=1}^k \sum_{n \geq 0} \alpha_i^n X^n \sum_{j=0}^{\min\{ n, m_i - 1 \}} a_{i,j} \alpha_i^{-j} \binom{n - j + m_i - 1}{m_i - 1} \\
            %&= \sum_{n \geq 0} \left( \sum_{i=1}^k \lambda_i(n) \alpha_i^n \right) X^n,
        %\end{align*}
        %where $\displaystyle \lambda_i(n) = \sum_{j=0}^{\min\{ n, m_i - 1 \}} a_{i,j} \alpha_i^{-j} \binom{n - j + m_i - 1}{m_i - 1}$ %is clearly a polynomial in $n$ with $\deg \lambda_i \leq m_i - 1$. Therefore, $f \in F$.
    \end{enumerate}
\end{proof}

\begin{exa}
    Let $F \in \C^\N$ be the function such that $F(n)$ is the $n$-th Fibonacci number. Denote $\Phi = \frac{1 + \sqrt{5}}{2}$, $\phi = \frac{1 - \sqrt{5}}{2}$, $A = \frac{\Phi}{\sqrt{5}}$, and $B = -\frac{\phi}{\sqrt{5}}$. Then the function $F$ can be characterized in the following ways.
    \begin{enumerate}
        \item $(E^2 - E - I)(F) = \overline{0}$, with $F(0) = F(1) = 1$;
        \item $\displaystyle \sum_{n \in \N} F(n) X^n = (\overline{1} - X - X^2)^{-1}$;
        \item $\displaystyle \sum_{n \in \N} F(n) X^n = \frac{A}{\overline{1} - \Phi X} + \frac{B}{\overline{1} - \phi X}$;
        \item $F(n) = A \Phi^n + B \phi^n$.
    \end{enumerate}
\end{exa}

\chapter{Pólya Theory}\label{ch3}

In this chapter we will introduce the \emph{Zyklenzeiger}, also known as the cycle index series, and use it to enumerate colourings. For that, we will give a brief presentation of Pólya theory, an interesting technique that allows us to count without repetition. We begin by studying some useful results on permutations, then follow by seeing Pólya's two theorems and using it to solve a problem of counting necklaces. We then conclude by using the techniques seen in the chapter to count graphs on unlabelled vertices, and we ask the reader to note this idea closely, as it will prove useful when trying to understand the concept of unlabelled structure in the next chapter. The \emph{Zyklenzeiger} will also prove very useful in the next chapters, when we introduce the concept of species of structures. This chapter assumes that the reader has some knowledge in group theory. For further reading, we refer the reader to \cite{polya}, \cite[Ch. 10]{wagner}, and \cite[App. 1]{bergeron2}.

\section{Permutations and Orbits}

Let $A$ be a non-empty finite set. Then the group of permutations of $A$ is denoted by $S(A)$ and is called the 
\emph{symmetric group} of $A$. When $A = [n]$, then we write $S_n$ instead of $S([n])$.

\begin{defi}
Given a finite set $A$, a group $G$ is said to be a permutation group of $A$ if $G$ is a subgroup of $S(A)$.
\end{defi}

\begin{defi}
Given a finite set $A$ and a permutation group $G$ of $A$, we say that two elements $x$ and $y$ in $A$ are congruent mod $G$ 
if there exists $g \in G$ such that $g(x)=y$. Congruence mod $G$ is an equivalence relation on $A$. The equivalence classes of this equivalence relation are the orbits of $G$. The orbit of $G$ that contains the element $x \in A$ is denoted by $\orbit_x$, and the subgroup of $G$ formed by the permutations that fix $x$ is denoted by $G_x$ and is called the stabilizer of $x$.
We write $\orbit(G)$ for the set of all orbits of $G$.
\end{defi}

\begin{theorem}[Orbit-Stabilizer]
 Let $G$ be a permutation group on a finite set $A$, and $x \in A$. Then 
 $$|\orbit_x|=[G:G_x].$$
\end{theorem}
\begin{proof}
    Consider the function $\psi : \orbit_x \to G/G_x, \, g(x) \mapsto gG_x$.
    Firstly, let us see that this function is well defined, that is, the choice of a permutation $g$ is irrelevant. If $g_1(x) = g_2(x)$, then $g_2^{-1} \circ g_1 (x) = x$; hence, $g_2^{-1} \circ g_1 \in G_x$, and thus $g_1 G_x = g_2 G_x$.

    Let us now see that this function is surjective. Clearly, if $g G_x \in G/G_x$, then $\psi(y)=gG_x$ where 
    $y = g(x)$.

    Finally, let us show that $\psi$ is injective. Let $y_1 = g_1(x)$ and $y_2 = g_2(x)$ be two elements in $\orbit_x$ such that $\psi(y_1) = \psi(y_2)$. This implies $g_1 G_x = g_2 G_x$, and hence $g_1^{-1} \circ g_2 \in G_x$. We then have $y_2 = g_2(x) = g_1 \circ g_1^{-1} \circ g_2(x) = g_1(x) = y_1$.
\end{proof}

%\begin{defi}
%\st{Given a finite set $A$ and $G < S(A)$, the orbits of $G$ are the equivalence classes of the equivalence relation defined above.}
%\end{defi}

%\begin{notn}
%\st{Given $A$ a finite set and $G < S(A)$, the set of all orbits of $G$ is denoted by $\orbit(G)$.}
%\end{notn}

%\begin{notn}
%Let $A$ be a finite set and $G < S(A)$. The orbit of $G$ that contains an element $x \in A$ is denoted by $\orbit_x$, and the subgroup of $G$ that fixes $x$ is denoted by $G_x$ and is called the stabilizer of $x$.
%\end{notn}

Recall now that every permutation of a finite set can be written as a unique (up to change of order) composition of disjoint cycles. 

\begin{notn}
Let $G$ be a permutation group on a finite set $A$, and $g \in G$. Then we denote by $\lam_i(g)$ the amount of cycles of length $i$ in the decomposition of $g$.
\end{notn}

With these in hand, we can show the first useful lemma.

\begin{lemma}[Cauchy-Frobenius]\label{notBurn}
Let $G$ be a permutation group on a finite set $A$. The number of orbits of $G$ is equal to the average number of fixed points of permutations in $G$, or in other words, $$ |\orbit(G)| = \frac{1}{|G|} \sum_{g \in G} \lam_1(g).$$
\end{lemma}
\begin{proof}
%Let $x$ be an element in $A$. Each class of $G/G_x$ is of the form ${H_y = \{g \in G ; g(x)=y \}}$, as for $g,h \in G$ such that $g(x)=h(x)=y$ for some $y \in A$, then $h^{-1}g(x) = h^{-1}(y) = x$ and $h^{-1}g \in G_x$. There is, therefore, a natural bijection that maps $y \in \orbit_x$ to  $H_y \in G/G_x$. From this, we have $|\orbit_x| = |G|/|G_x|$. Then
By the Orbit-Stabilizer theorem, we have   
\begin{align*}
\sum_{g \in G} \lam_1(g) &= \sum_{g \in G} \underset{g(x)=x}{\sum_{x \in A}} 1 = \sum_{x \in A} \sum_{g \in G_x} 1 = \sum_{x \in A} |G_x| = \sum_{x \in A} \frac{|G|}{|\orbit_x|} \\
    &= \sum_{O \in \orbit(G)} \underset{\orbit_x = O}{\sum_{x \in A}} \frac{|G|}{|O|} = \sum_{O \in \orbit(G)} |G| = |G| \cdot |\orbit(G)|.
\end{align*}
\end{proof}

Before moving on, let us present, as a corollary, a different proof to a well-known combinatorial result.

\begin{corollary}
There are $(n-1)!$ possible circular arrangements of $n$ distinctly colored beads.
\end{corollary}
\begin{proof}
Let $A$ be the set of all $n!$ visually distinct circular arrangements of $n$ distinctly coloured beads, and $G$ be the set of the $n$ possible clockwise rotations. The only rotation that has any fixed point is the identical rotation, which fixes everything. Rotational equivalence classes coincide with orbits of $G$, and by Lemma \ref{notBurn} we have $\frac{1}{n}(n!+0+...+0) = (n-1)!$ such classes.
\end{proof}

\section{Pólya's Theorems}

In this section, we will present Pólya's two theorems. For this, let us choose two sets $C,D$, with $|C|=k,|D|=n$, and $n,k \geq 2$. We call $C$ the set of colours, and $\vfi \in C^D$ a colouring of $D$.

\begin{defi}
Let $G$ be a permutation group on $D$. We define an equivalence relation on $C^D$ by 
$\vfi_1 \til \vfi_2 \iff (\exists  g \in G) \, \vfi_1 \circ g = \vfi_2$. The equivalence classes of $\til$ are called $G$-classes.
\end{defi}

\begin{exa}\label{322}
    As an example to the usage of this equivalence relation, let $D$ be the set of the $6$ equally spaced beads of a hexagonal necklace, $C=\{black, white\}$, and let $G = D_6$ be the hexagon dihedral group (seen as a permutation group on $D$). 
    Then the set of all visually distinct hexagonal necklaces with black or white beads can naturally be identified with $C^D$.   
    Note that two visually distinct necklaces are in the same $G$-class if and only if they are nothing more but different drawings 
    of what is essentially the same necklace. For instance, the following visually distinct hexagonal necklaces are equivalent.
    \begin{center}

\tikzset{every picture/.style={line width=0.75pt}} %set default line width to 0.75pt        


    \end{center}
\end{exa}

Let us build a relation between $G$-classes and a permutation group on the set of colourings.

\begin{defi}
Let $G$ be a permutation group on $D$. For each $g \in G$, $\vfi \in C^D$, define $\overline{g}(\vfi) := \vfi \circ g$, an endofunction on $C^D$. Since $g$ is invertible, $\vfi_1 \circ g = \vfi_2 \circ g \implies \vfi_1 = \vfi_2$; hence, $\overline{g}$ is injective, and thus a permutation on $C^D$. Define $\overline{G}$ as the set of all such $\overline{g}$.
\end{defi}

\begin{lemma}\label{cardG}
The set $\overline{G}$ is a permutation group on $C^D$, $|\overline{G}|=|G|$, and $G$-classes coincide with orbits of $\overline{G}$.
\end{lemma}
\begin{proof}
In order to see that $\overline{G}$ is a subgroup of $S(C^D)$, it suffices to show that, for $\overline{g}, \overline{h} \in \overline{G}$, we have $\overline{g}\circ\overline{h} \in \overline{G}$. But, for all $\vfi \in C^D$, $\overline{g}\circ\overline{h}(\vfi) = \vfi \circ h \circ g = \overline{h \circ g}(\vfi)$, and since $h \circ g \in G$, $\overline{h \circ g} \in \overline{G}$.

For $|\overline{G}| = |G|$, it is enough to show that the map $g \mapsto \overline{g}$ is injective, as $\overline{G}$ is the image of this map by definiton. Let $g,h \in G$ be distinct permutations. For some $x \in D$, $g(x) \neq h(x)$. Since $k \geq 2$, there is some $\vfi \in C^D$ such that $\vfi(g(x)) \neq \vfi(h(x))$, and thus $\overline{g}(\vfi) \neq \overline{h}(\vfi)$. Hence, $\overline{g} \neq \overline{h}$, and the map is injective.

And finally, let $\vfi_1,\vfi_2 \in C^D$. By definition, 
$$\vfi_1 \til \vfi_2 \iff (\exists g \in G) \,  \vfi_1 \circ g = \vfi_2 \iff (\exists g \in G) \, \overline{g}(\vfi_1) = \vfi_2$$
and so the $G$-classes coincide with the orbits of $\overline{G}$.
\end{proof}

%One more step is needed before enumerating $G$-classes. That is the introduction to the \emph{Zyklenzeiger}.

\begin{defi}\label{zyk}
The \emph{Zyklenzeiger} of a permutation group $G$ is the series $Z(G;x_1,x_2, \ldots)$ defined by
\begin{equation*}
    Z(G;x_1,x_2,\ldots) = \frac{1}{|G|} \sum_{g \in G} x_1^{\lam_1(g)} x_2^{\lam_2(g)} \ldots 
\end{equation*}
Naturally, since $G$ is a (finite) permutation group, the Zyklenzeiger is a polynomial
\begin{equation*}
    Z(G;x_1, \ldots ,x_n) = \frac{1}{|G|} \sum_{g \in G} x_1^{\lam_1(g)}\cdots x_n^{\lam_n(g)}.
\end{equation*}
\end{defi}

\begin{exa}\label{326}
    Consider the following hexagon as a notational standard.
    \begin{center}
        \begin{tikzcd}
                           & 1 \arrow[rd, no head] &                            \\
6 \arrow[ru, no head] &                            & 2 \arrow[d, no head]  \\
5 \arrow[u, no head]  &                            & 3 \arrow[ld, no head] \\
                           & 4 \arrow[lu, no head] &                           
\end{tikzcd}
    \end{center}
    Then the elements of $D_6$ are given by
    \begin{align*}
        &I = (1)(2)(3)(4)(5)(6) &S_{1,4} = (1)(4)(2\ 6)(3\ 5) \\
        &R = (1\ 2\ 3\ 4\ 5\ 6) &S_{2,5} = (2)(5)(1\ 3)(4\ 6) \\
        &R^2 = (1\ 3\ 5)(2\ 4\ 6) &S_{3,6} = (3)(6)(1\ 5)(2\ 4) \\
        &R^3 = (1\ 4)(2\ 5)(3\ 6) &S_{12,45} = (1\ 2)(3\ 6)(4\ 5) \\
        &R^4 = (1\ 5\ 3)(2\ 6\ 4) &S_{23,56} = (1\ 4)(2\ 3)(5\ 6) \\
        &R^5 = (1\ 6\ 5\ 4\ 3\ 2) &S_{34,61} = (1\ 6)(2\ 5)(3\ 4)
    \end{align*}
    It is then easy to see that $Z(D_6;x_1,\ldots,x_6) = \frac{1}{12} (x_1^6 + 4x_2^3 + 3x_1^2x_2^2 + 2x_3^2 + 2x_6)$.
\end{exa}

With these in hand, we can present Pólya's first theorem.

\begin{theorem}[Pólya's first theorem]\label{polya1}
The number of $G$-classes induced by $\sim$ is equal to $Z(G;k, \ldots ,k)$.
\end{theorem}
\begin{proof}
From Lemmas \ref{notBurn} and \ref{cardG}, we have
\begin{equation*}
    |\orbit(\overline{G})| = \frac{1}{|\overline{G}|} \sum_{\overline{g} \in \overline{G}} \lam_1(\overline{g}) = \frac{1}{|G|} \sum_{g \in G} |\{ \vfi \in C^D ; \vfi \circ g = \vfi \}| = \frac{1}{|G|} \sum_{g \in G} k^{\lam_1(g) + \ldots + \lam_n(g)}
\end{equation*}
since $\vfi \circ g = \vfi$ if and only if $\vfi$ is constant on the cycles of $g$, and $g$ has $\lam_1(g) + \ldots + \lam_n(g)$ cycles.
\end{proof}

\begin{exa}\label{328}
    Let us extend examples \ref{322} and \ref{326} by enumerating how many types of hexagonal necklaces there are with two colours of beads (namely, black and white). By Theorem \ref{polya1}, this number is equal to
    \begin{equation*}
        Z(D_6;2,\ldots,2) = \frac{1}{12}(2^6 + 4 \cdot 2^3 + 3 \cdot 2^2 \cdot 2^2 + 2 \cdot 2^2 + 2 \cdot 2) = \frac{156}{12} = 13,
    \end{equation*}
    i.e., there are $13$ equivalence classes of hexagonal necklaces with two colours of beads.
\end{exa}

Moving on, let us assume that $C=[k]$ and $D = [n]$. 
%One might wonder about ways for enumerating $G$-classes with specific properties. %One of these would be the pattern inventory polynomial, to be defined below.

\begin{defi}
Given $\vfi \in [k]^{[n]}$, the frequency type of $\vfi$ is given by $$(f_1(\vfi), \ldots ,f_k(\vfi)) = (|\vfi^{-1}(\{1\})|, \ldots ,|\vfi^{-1}(\{k\})|).$$
\end{defi}

\begin{lemma}\label{sameclass}
If two colourings are in the same $G$-class, then they have the same frequency type.
\end{lemma}
\begin{proof}
Let $\vfi_1$ and $\vfi_2$ be two colourings in the same $G$-class. Then $\vfi_1 \circ g = \vfi_2$ for some $g \in G$. 
Since $g$ is a permutation, it can be written as a product of disjoint cycles. 
Therefore, it suffices to show the result for the case when $g = (a_1 a_2 \ldots a_{j-1} a_j)$ is a single cycle.

Let $b_i=\vfi_1(a_i)$ for $1 \leq i \leq j$. Note that the $b_i's$ are not necessarily pairwise distinct. 
The image under $\vfi_1$ of the ordered set $\{ a_1, a_2, \ldots , a_{j-1}, a_j \}$ is the ordered multiset 
$\{ b_1, b_2, \ldots , b_{j-1}, b_j \}$. Naturally, the image of that same ordered set under $\vfi_1 \circ g$ is the ordered multiset $\{ b_2, b_3, ..., b_j, b_1 \}$, a simple reordering of the former. It is now clear that the frequency type is preserved.
\end{proof}

The above lemma allows us to talk about the frequency type of a $G$-class.
%, and with that define the pattern inventory polynomial.

\begin{defi}
The pattern inventory polynomial of $G$-classes is defined as $$ P(G;x_1, \ldots ,x_k) := \underset{f_j \geq 0}{\sum_{f_1 + \ldots + f_k = n}} G_{f_1 \ldots f_k} x_1^{f_1} \cdots x_k^{f_k} $$
where $G_{f_1 \ldots f_k}$ is the amount of $G$-classes with frequency type $(f_1, \ldots ,f_k)$.
\end{defi}

In order to find a proper way to enumerate the coefficients of the pattern inventory polynomial, we must first define weights and generalize Cauchy-Frobenius's Lemma.

\begin{defi}\label{wei}
Let $G$ be a permutation group on a finite set $A$, and $R$ a commutative ring with $\Q \subseteq R$. Then a map $w : A \to R$ is a weight function on $A$ with respect to $G$ if $w$ is constant on the orbits of $G$. This induces naturally a map $W : \orbit(G) \to R$.
\end{defi}

\begin{lemma}\label{notb2}
In the context of the definition above,
\begin{equation*}
    \sum_{O \in \orbit(G)} W(O) = \frac{1}{|G|} \sum_{g \in G} \underset{g(x)=x}{\sum_{x \in A}} w(x)
\end{equation*}
\end{lemma}
\begin{proof}
\begin{align*}
\sum_{g \in G} \underset{g(x)=x}{\sum_{x \in A}} w(x) &= \sum_{x \in A} w(x) \underset{g(x)=x}{\sum_{g \in G}} 1 = \sum_{x \in A} w(x)|G_x| = |G| \sum_{x \in A} \frac{w(x)}{|\orbit_x|} \\
&= |G| \sum_{O \in \orbit(G)} \sum_{x \in O} \frac{w(x)}{|O|} = |G| \sum_{O \in \orbit(G)} \sum_{x \in O} \frac{W(O)}{|O|} = |G| \sum_{O \in \orbit(G)} W(O). 
\end{align*}
\end{proof}

With these we can recuperate the pattern inventory polynomial from the \emph{Zyklenzeiger}.

\begin{theorem}[Pólya's second theorem]\label{polya2}
The pattern inventory polynomial is related to the Zyklenzeiger by the following formula:
\begin{equation*}
    P(G;x_1, \ldots ,x_k) = Z(G;\sum_{j=1}^k x_j,\sum_{j=1}^k x_j^2, \ldots , \sum_{j=1}^k x_j^n)
\end{equation*}
\end{theorem}
\begin{proof}
For each colouring $\vfi$, define $ w(\vfi) := \prod_{i=1}^n x_{\vfi(i)} = x_1^{f_1(\vfi)}\cdots x_k^{f_k(\vfi)}$.
By Lemmas \ref{cardG} and \ref{sameclass},
$G$-classes coincide with orbits of $\overline{G}$, and $w$ is constant on the orbits of $\overline{G}$. 
Naturally derive $W$ from $w$, as in Definition \ref{wei}. Then the pattern inventory polynomial is clearly given by 
$P(G; x_1,\ldots ,x_k) = \sum_{O \in \orbit(\overline{G})} W(O)$. From Lemma \ref{notb2},
\begin{equation}\label{pst1}
    \sum_{O \in \orbit(\overline{G})} W(O) = \frac{1}{|G|} \sum_{g \in G} \underset{\overline{g}(\vfi)=\vfi}{\sum_{\vfi \in [k]^{[n]}}} w(\vfi) = \frac{1}{|G|}\sum_{g \in G} \underset{\overline{g}(\vfi)=\vfi}{\sum_{\vfi \in [k]^{[n]}}} 
    x_1^{f_1(\vfi)} \cdots x_k^{f_k(\vfi)}.
\end{equation}
Since each $\vfi$ in the second sum is constant on cycles of $g$, we can write each $f_i(\vfi)$ as $1t_{i,1}+2t_{i,2}+ \ldots +nt_{i,n}$, where $t_{i, j}$ is the number of cycles of $g$ of length $j$ to which $\vfi$ assigns the colour $i$. 
We have 
\begin{align*}
    x_1^{f_1(\vfi)}x_2^{f_2(\vfi)}\cdots x_k^{f_k(\vfi)} &= x_1^{1t_{1,1}+2t_{1,2}+ \ldots +nt_{1,n}} x_2^{1t_{2,1}+2t_{2,2}+ \ldots +nt_{2,n}} \cdots x_k^{1t_{k,1}+2t_{k,2}+ \ldots +nt_{k,n}} \\
    &= (x_1^{t_{1,1}} x_2^{t_{2,1}} \cdots x_k^{t_{k,1}})(x_1^{2t_{1,2}} x_2^{2t_{2,2}} \cdots x_k^{2t_{k,2}})\cdots(x_1^{nt_{1,n}} x_2^{nt_{2,n}} \cdots x_k^{nt_{k,n}}).
\end{align*}
Since $t_{1,j}+t_{2,j}+ \ldots +t_{k,j} = \lam_j(g)$, there are $\binom{\lam_1(g)}{t_{1,1},\ldots,t_{k,1}}\binom{\lam_2(g)}{t_{1,2},\ldots,t_{k,2}}\cdots\binom{\lam_n(g)}{t_{1,n},\ldots,t_{k,n}}$ colourings $\vfi$ with associated decomposition  $(x_1^{t_{1,1}} x_2^{t_{2,1}} \cdots x_k^{t_{k,1}})(x_1^{2t_{1,2}} x_2^{2t_{2,2}} \cdots x_k^{2t_{k,2}})\cdots(x_1^{nt_{1,n}} x_2^{nt_{2,n}} \cdots x_k^{nt_{k,n}})$, as these count the ways to choose which cycle gets which colour. Grouping accordingly, we then have that equation \ref{pst1} is equal to
\begin{equation*}
    \frac{1}{|G|}\sum_{g \in G} \left( \sum_{t_1 + \ldots + t_k = \lam_1(g)} \binom{\lam_1(g)}{t_{1},\ldots,t_{k}} x_1^{t_1} \cdots x_k^{t_k} \right) %\left( \sum_{t_1 + \ldots + t_k = \lam_2(g)} \binom{\lam_2(g)}{t_{1},\ldots,t_{k}} x_1^{2t_1} \cdots x_k^{2t_k} \right)
    \cdots \left( \sum_{t_1 + \ldots + t_k = \lam_n(g)} \binom{\lam_n(g)}{t_{1},\ldots,t_{k}} x_1^{nt_1} \cdots x_k^{nt_k} \right)
\end{equation*}
%Since, for each such decomposition of $\lam_j(g)$, there are $\binom{\lam_j(g)}{t_{1,j},t_{2,j}, \ldots ,t_{k,j}} = \binom{\lam_j(g)}{t_{1,j}}\binom{\lam_j(g)-t_{1,j}}{t_{2,j}}\ldots \binom{\lam_j(g)-t_{1,j}-...-t_{k-1,j}}{t_{k,j}}$ possible choices for which $j$-cycles are of which color, it is clear that formula \ref{pst1} is equal to
\begin{equation*}
    = \frac{1}{|G|} \sum_{g \in G} (x_1 + ...+ x_k)^{\lam_1(g)} \cdots (x_1^n + ... + x_k^n)^{\lam_n(g)} = Z(G;\sum_{j=1}^k x_j, \ldots, \sum_{j=1}^k x_j^n)
\end{equation*}
\end{proof}

\begin{exa}
    Let us conclude example \ref{328} by evaluating the distribution of hexagonal necklaces regarding the number of black and white beads. For that, let $b$ be the variable that counts black beads and $w$ be the variable that counts white beads. Then the polynomial that gives us this distribution, i.e. the pattern inventory, is given by
    \begin{align*}
        P(D_6;b,w) &= Z(D_6;b+w,b^2+w^2,\ldots,b^6+w^6) \\
        &= \frac{1}{12}[(b+w)^6 + 4(b^2+w^2)^3 +3(b+w)^2(b^2+w^2)^2 + 2(b^3+w^3)^2 + 2(b^6+w^6)] \\
        &= 1b^6 + 1b^5w^1 + 3b^4w^2 + 3b^3w^3 + 3b^2w^4 + 1b^1w^5 + 1w^6.
    \end{align*}
\end{exa}

\section{Isomorphic Graphs}\label{isograph}

In this section, we will use the theorems proved in the last section to count isomorphism types of graphs.

\begin{defi}
A graph on a finite set $V$ is an ordered pair $(V,E)$, where $E \subseteq \binom{V}{2}$. Elements of $V$ are called vertices, and elements of $E$ are called edges.
\end{defi}

\begin{defi}\label{isogr}
Two graphs $(V_1,E_1), (V_2,E_2)$ are said to be isomorphic if there is a bijection $\sigma : V_1 \to V_2$ such that $\{u,v\} \in E_1 \iff \{ \sigma(u),\sigma(v) \} \in E_2$. %If $V_1 = V_2$, 
We say that two isomorphic graphs have the same isomorphism type.
\end{defi}

Let us now assume that all graphs are on a set $V = [n]$ with $n \geq 3$. 
%It becomes clear, upon further inspection, that isomorphism types of graphs coincide with graphs on unlabeled vertices. 
Our goal here is to enumerate isomorphism types of graphs on $V=[n]$ using Pólya theory. For that, let us associate to each graph $([n],E)$ a colouring $\vfi : \binom{[n]}{2} \to \{yes,no\}$, where $\vfi(\{u,v\}) = yes$ if $\{u,v\} \in E$, and $\vfi(\{u,v\}) = no$ otherwise. Denote as usual by $S_n$ the symmetric group on $[n]$ and by $S(\binom{[n]}{2})$ the symmetric group on $\binom{[n]}{2}$.

\begin{defi}
For each $\sigma \in S_n$ and $\{u,v\} \in \binom{[n]}{2}$, define $g_\sigma(\{u,v\}) := \{\sigma(u),\sigma(v)\}$.
\end{defi}

Since $\sigma : [n] \to [n]$ is injective, $g_\sigma : \binom{[n]}{2} \to \binom{[n]}{2}$ is also injective, and therefore a permutation. Since $g_{\sigma_1 \circ \sigma_2} = g_{\sigma_1} \circ g_{\sigma_2}$, the map $\sigma \mapsto g_\sigma$ is a 
homomorphism from $S_n$ into $S(\binom{[n]}{2})$. Furthermore, this homomorphism is injective. Indeed, $g_{\sigma_1}(\{1,k\}) = g_{\sigma_2}(\{1,k\})$ for all  $k = 2,\ldots ,n$ implies $\sigma_1(j) = \sigma_2(j)$ for all $j \in [n]$. 
(Note that for $n=2$ the homomorphism is not injective, as there are exactly two permutations such that $g_{\sigma_1}(\{1,2\}) = g_{\sigma_2}(\{1,2\})$, namely, $\sigma_1 = Id_2$ and $\sigma_2 = (1\ 2)$.) 

\begin{notn}
Denote by $S_n^{(2)}$ the image of the map $\sigma \mapsto g_\sigma$, and call it the pair group.
\end{notn}

The pair group allows us to see our problem from the perspective of Pólya theory. We can then rewrite Definition \ref{isogr} as

\begin{defi}
Two graphs $([n],E_1),([n],E_2)$ are of the same isomorphism type when their respective colourings $\vfi_1,\vfi_2$ are of the same $S_n^{(2)}$-class, that is, there exists $\sigma \in S_n$ such that $\vfi_1 = \vfi_2 \circ g_\sigma$.
\end{defi}

The following two theorems are then consequences of Pólya's two theorems.

\begin{theorem}\label{336}
The number of isomorphism types of graphs on $n$ vertices is given by \newline$Z(S_n^{(2)};x_1,\ldots ,x_{\binom{n}{2}})|_{x_i = 2}$.
\end{theorem}
\begin{proof}
Immediate from Theorem \ref{polya1}.
\end{proof}

\begin{theorem}\label{337}
Let $\mathcal{G}_j$ denote the number of isomorphism types of graphs on $n$ vertices with exactly $j$ edges. Then $$\sum_{j=0}^{\binom{n}{2}} \mathcal{G}_j x^j = Z(S_n^{(2)};1+x,1+x^2, \ldots ,1+x^{\binom{n}{2}})$$
\end{theorem}
\begin{proof}
Immediate from Theorem \ref{polya2}, replacing the variable related to $no$ with 1, and the variable related to $yes$ with $x$.
\end{proof}

Now the last step to enumerate these isomorphism types is to calculate the \emph{Zyklenzeiger} of $S_n^{(2)}$. There is an algorithm to do so, and it comes from a relation with the \emph{Zyklenzeiger} of $S_n$. From \cite[Theorem 10.2.3]{wagner}, we have
\begin{equation}\label{zyksn}
    Z(S_n;x_1, \ldots ,x_n) = \sum_{\lam_1+2\lam_2+ \ldots +n\lam_n = n} \frac{x_1^{\lam_1}x_2^{\lam_2} \cdots x_n^{\lam_n}}{1^{\lam_1}\lam_1!2^{\lam_2}\lam_2! \cdots n^{\lam_n}\lam_n!}
\end{equation}

The following few lemmas will present the relations between permutations in $S_n$ and in $S_n^{(2)}$. From now on, assume that $\sigma$ and $g_\sigma$ are as before.

\begin{notn}
Denote by $a$ mod $m$ the unique $r \in [m]$ such that $a \equiv r$ mod $m$.
\end{notn}

\begin{lemma}\label{sn21}
If $m \in \N$, each cycle $(v_1 v_2 \ldots v_{2m+1})$ of the disjoint cycle decomposition of $\sigma$ gives rise in $g_\sigma$ to $m$ cycles of length $2m+1$, where the entries of the latter cycles are pairs of vertices from $\{v_1,v_2,\ldots,v_{2m+1}\}$.
\end{lemma}
\begin{proof}
Consider $\{v_j,v_k\}$ with $j,k$ distinct in $[2m+1]$. We claim that the cycle of $g_\sigma$ containing $\{v_j,v_k\}$ is
\begin{equation*}
    (\{v_{j+i-1 \!\!\!\!\mod 2m+1},v_{k+i-1 \!\!\!\!\mod 2m+1}\})_{1 \leq i \leq 2m+1}
\end{equation*}
It is clear that $g_\sigma(\{v_{j+i-1 \!\!\mod 2m+1},v_{k+i-1 \!\!\mod 2m+1}\}) = \{v_{j+i \!\!\mod 2m+1},v_{k+i \!\!\mod 2m+1}\}$. To see that this cycle is of length $2m+1$, we must show that if $i,i' \in [2m+1]$ and $$\{v_{j+i-1 \!\!\mod 2m+1},v_{k+i-1 \!\!\mod 2m+1}\} = \{v_{j+i'-1 \!\!\mod 2m+1},v_{k+i'-1 \!\!\mod 2m+1}\}$$ then $i=i'$. Consider the only two possible cases: 
\begin{enumerate}
\item $v_{j+i-1 \!\!\mod 2m+1} = v_{k+i'-1 \!\!\mod 2m+1}$ and $v_{k+i-1 \!\!\mod 2m+1} = v_{j+i'-1 \!\!\mod 2m+1}$, or
\item $v_{j+i-1 \!\!\mod 2m+1} = v_{j+i'-1 \!\!\mod 2m+1}$.
\end{enumerate}
But (i) implies $j+i-1 \equiv k+i'-1$ mod $2m+1$ and $k+i-1 \equiv j+i'-1$ mod $2m+1$, from where $j \equiv k$ mod $2m+1$ and $j=k$, contradicting the hypothesis; from (ii) we have ${j+i-1 \equiv j+i'-1}$ mod $2m+1$, from where $i \equiv i'$ mod $2m+1$, and $i = i'$.

Since we are taking the disjoint cycle decomposition of $g_\sigma$, these cycles partition the set $\binom{\{v_1,v_2,...,v_{2m+1}\}}{2}$ into blocks of size $2m+1$, therefore there are $\binom{2m+1}{2}/(2m+1) = m$ cycles.
\end{proof}

\begin{lemma}\label{sn22}
If $m \in \N \setminus \{0\}$, each cycle $(v_1 v_2 ... v_{2m})$ of $\sigma$ gives rise in $g_\sigma$ to $m-1$ cycles of length $2m$ and one cycle of length $m$, where the entries of the latter cycles are pairs of vertices from $\{ v_1,v_2,...,v_{2m} \}$.
\end{lemma}
\begin{proof}
Consider $\{v_j,v_k\}$ with $j,k$ distinct in $[2m]$. If $j \equiv k$ mod $m$, then the cycle containing $\{v_j,v_k\}$ is $(\{v_1,v_{m+1}\}\{v_2,v_{m+2}\}\ldots\{v_m,v_{2m}\})$. If not, then by an argument analogue to that in the proof of Lemma \ref{sn21} the cycle containing $\{v_j,v_k\}$ is $(\{v_{j+i-1 \!\!\!\!\mod 2m},v_{k+i-1 \!\!\!\!\mod 2m}\})_{1 \leq i \leq 2m}$. To determine the number of such cycles, we see that these cycles partition $\binom{\{v_1,v_2,...,v_{2m}\}}{2}$ into one part of size $m$ and $x$ parts of size $2m$, from where $m+2mx = \binom{2m}{2}$, and $x = m-1$.
\end{proof}

\begin{lemma}\label{sn23}
If $r,s \in \N \setminus \{0\}$, each pair $\{ (u_1\ldots u_r),(v_1\ldots v_s) \}$ of cycles in $\sigma$ gives rise in $g_\sigma$ to $\gcd(r,s)$ cycles of length $\lcm(r,s)$, where entries of the latter cycles are of the form $\{u_j,v_k\}$.
\end{lemma}
\begin{proof}
It's clear that $g_\sigma(\{u_{j+i-1 \!\!\mod r},v_{k+i-1 \!\!\mod s}\}) = \{u_{j+i \!\!\mod r},v_{k+i \!\!\mod s}\}$, ${i \in [\lcm(r,s)]}$, and these pairs are clearly all distinct. Hence the $rs$ pairs in ${\{ u_1,\ldots,u_r\}\times\{v_1,\ldots,v_s \}}$ get partitioned into parts of length $\lcm(r,s)$, therefore there are $rs/\lcm(r,s) = \gcd(r,s)$ such cycles.
\end{proof}

Now we have the necessary tools to build the \emph{Zyklenzeiger} of $S_n^{(2)}$.
\newpage
\begin{theorem}\label{zyksn2}
The Zyklenzeiger of $S_n^{(2)}$ is derived from the Zyklenzeiger of $S_n$ in the following manner:
\begin{enumerate}
    \item A factor $x_{2m+1}$ in a term of the Zyklenzeiger of $S_n$ gives rise to a factor $x_{2m+1}^m$ in the Zyklenzeiger of $S_n^{(2)}$;
    \item A factor $x_{2m}$ in a term of the Zyklenzeiger of $S_n$ gives rise to a factor $x_m x_{2m}^{m-1}$ in the Zyklenzeiger of $S_n^{(2)}$;
    \item A factor $x_r x_s$ in a term of the Zyklenzeiger of $S_n$ gives rise to a factor $x_{\lcm(r,s)}^{\gcd(r,s)}$ in the Zyklenzeiger of $S_n^{(2)}$.
\end{enumerate}
\end{theorem}
\begin{proof}
The injectivity of the map $\sigma \mapsto g_\sigma$ gives $|S_n| = |S_n^{(2)}|$. The theorem then follows immediately from Lemmas \ref{sn21}, \ref{sn22}, \ref{sn23}, and Definition \ref{zyk}.
\end{proof}

Let us now conclude this section by calculating the number of graphs on four unlabelled vertices using Theorem \ref{336} and also their distribution according to the amount of edges using Theorem \ref{337}. From formula \ref{zyksn}, we have
\begin{equation*}
    Z(S_4;x_1,x_2,x_3,x_4) = \frac{1}{24}(x_1^4 + 6x_1^2x_2 + 8x_1x_3 + 3x_2^2 + 6x_4)
\end{equation*}
From Theorem \ref{zyksn2}, we have the \emph{Zyklenzeiger} of $S_4^{(2)}$ by the following:
\begin{enumerate}
    \item Each of the four $x_1$ in $x_1^4$ gives rise to a factor 1, and each pair $\{ x_1,x_1 \}$ gives rise to a factor $x_1$; therefore $x_1^4$ gives rise to a factor $x_1^6$.
    \item Each of the two $x_1$ in $x_1^2x_2$ gives rise to a factor 1, the $x_2$ gives rise to a factor $x_1$, the pair $\{x_1,x_1\}$ gives rise to a factor $x_1$ and each pair $\{x_1,x_2\}$ gives rise to a factor $x_2$; therefore $6x_1^2x_2$ gives rise to a factor $6x_1^2x_2^2$.
    \item The $x_1$ in $x_1x_3$ gives rise to a factor $1$, the $x_3$ gives rise to a factor $x_3$, and the pair $\{ x_1,x_3 \}$ gives rise to a factor $x_3$; therefore $8x_1x_3$ gives rise to a factor $8x_3^2$.
    \item Each $x_2$ in $x_2^2$ gives rise to a factor $x_1$, and the pair $\{ x_2,x_2 \}$ gives rise to a factor $x_2^2$; therefore $3x_2^2$ gives rise to a factor $3x_1^2x_2^2$.
    \item The $x_4$ gives rise to a factor $x_2x_4$; therefore $6x_4$ gives rise to a factor $6x_2x_4$.
\end{enumerate}
From this, we have
\begin{equation*}
    Z(S_4^{(2)};x_1,x_2,x_3,x_4,x_5,x_6) = \frac{1}{24} (x_1^6 + 9x_1^2x_2^2 + 8x_3^2 + 6x_2x_4).
\end{equation*}
We then have the amount of graphs on four unlabelled vertices given by
\begin{equation*}
    \frac{1}{24} (x_1^6 + 9x_1^2x_2^2 + 8x_3^2 + 6x_2x_4)|_{x_i = 2} = \frac{1}{24} (2^6 + 9\cdot2^4 + 8\cdot2^2 + 6\cdot2^2) = 11,
\end{equation*}
and their distribution regarding the number of edges is given by the polynomial
\begin{equation*}
    \frac{1}{24} [(1+x)^6 + 9(1+x)^2(1+x^2)^2 + 8(1+x^3)^2 + 6(1+x^2)(1+x^4)] =
\end{equation*}
\begin{equation*}
    \frac{1}{24} [24+24x+48x^2+72x^3+48x^4+24x^5+24x^6] = 1+1x+2x^2+3x^3+2x^4+1x^5+1x^6.
\end{equation*}

\chapter{Species of Structures}\label{ch4}

In this chapter, we present an introduction to the concept of structure, species of structures, and some operations on the space of species. We begin by giving a brief introduction to category theory, as it bears the key concept of functor which will be used to define species. We then give a proper definition of species of structures and show the series that are used to enumerate them. After that, we show some operations that can be used to relate species, such as addition, product, and composition. Finally, we present the concepts of multisort and virtual species, which will prove useful in the subsequent chapters. The main references for this chapter are \cite{joyal}, \cite{bergeron2}, and \cite{bergeron1}.

\section{Category Theory}

In order to fully understand the concept of species of structures, it is useful to have some knowledge of category theory. 
%For that, this section brings a light introduction to the topic. \textcolor{blue}{Maybe you should change the last sentence.}

\begin{defi}
    A category $\cat$ consists of the following mathematical entities: 
    \begin{enumerate}
        \item a class $\mathbf{Ob}(\cat)$, whose elements are called the objects of $\cat$;
        \item %a family $\mathbf{Mor}(\cat)$ associating, 
             for every pair $A,B \in \mathbf{Ob}(\cat)$, a set $\mathbf{Mor}(\cat)(A,B)$, whose elements are called the 
             morphisms from $A$ to $B$;
        \item for each $A,B,C \in \mathbf{Ob}(\cat)$, a composition mapping $$\mathbf{Mor}(\cat)(B,C) \times \mathbf{Mor}(\cat)(A,B) \to \mathbf{Mor}(\cat)(A,C), \quad (g, f) \mapsto g \circ f;$$
        \item for each $A \in \mathbf{Ob}(\cat)$, an identity morphism $Id_A \in \mathbf{Mor}(\cat)(A,A).$
    \end{enumerate}
    The composition of morphisms is required to satisfy the following conditions:
    \begin{enumerate}
        \item[v)] for $A,B,C,D \in \mathbf{Ob}(\cat)$, $f \in \mathbf{Mor}(\cat)(A,B)$, $g \in \mathbf{Mor}(\cat)(B,C)$, $h \in \mathbf{Mor}(\cat)(C,D)$, $$ h \circ (g \circ f) = (h \circ g) \circ f;$$
        \item[vi)] for $A,B,C \in \mathbf{Ob}(\cat)$, $f \in \mathbf{Mor}(\cat)(A,B)$, $g \in \mathbf{Mor}(\cat)(C,A)$, $$ f \circ Id_A = f, \ \ \  Id_A \circ g = g.$$
    \end{enumerate}
\end{defi}

%\textcolor{blue}{(Write the examples as a list; in each case, explicitely explain the composition.
%Also, mention the categories of vector spaces and topological spaces.)}
%Some examples of categories are $\mathbf{Grp}$, the category whose objects are groups and morphisms are group homomorphisms; $\mathbf{Rng}$, the category whose objects are rings and morphisms are ring homomorphisms; $\E$, the category whose objects are finite sets and morphisms are all functions between them; and $\B$, the category whose objects are finite sets and morphisms are all 
%bijective functions between them.

\begin{exa}
    Here are some examples of categories. In each example, the composition of morphisms is the usual composition of functions. 
    \begin{itemize}
        \item The category $\mathbf{Grp}$ of groups and group homomorphisms;
        \item The category $\mathbf{FVect_\mathbb{K}}$ of finite dimensional vector spaces over a field $\mathbb{K}$ 
        and linear transformations;
        \item The category $\mathbf{Sp}$ of topological spaces and continuous maps;
        \item The category $\E$ of finite sets and functions;
        \item The category $\B$ of finite sets and bijections.
    \end{itemize}
\end{exa}

Let $\mathcal{C}$ be a category, and let $A, B \in \mathbf{Ob}(\mathcal{C})$. We replace the cumbersome notation 
${f \in \mathbf{Mor}(\cat)(A,B)}$ with the more intuitive $f:A \to B$.

%Let us now categorize some morphisms.

%\begin{defi}
%    Let $\cat$ be a category and $f : A \to B$ a morphism in $\cat$. The morphism $f$ is called a monomorphism if and only if for all $g,h : C \to A$ such that $f \circ g = f \circ h$ it follows that $g=h$.
%\end{defi}
%\begin{defi}
 %   Let $\cat$ be a category and $f : A \to B$ a morphism in $\cat$. The morphism $f$ is called an epimorphism if and only if for all $g,h : B \to C$ such that $g \circ f = h \circ f$ it follows that $g=h$.
%\end{defi}
\begin{defi}
    Let $\cat$ be a category and $f : A \to B$ a morphism in $\cat$. 
    The morphism $f$ is called an isomorphism if and only if there exists a morphism $g : B \to A$ such that $f \circ g = Id_B$ and $g \circ f = Id_A$. It is easy to see that if such $g$ exists, then it is unique. Hence, we are justified to call $g$ the inverse of $f$, and to write $g=f^{-1}$.
\end{defi}

%\textcolor{blue}{(What are the isomorphisms in each of the above categories?)}
%We can now define some relations between two categories, namely functors.
For example, the isomorphisms in $\mathbf{FVect}_\mathbb{K}$ are the linear isomorphisms between vector spaces, and those in $\mathbf{Sp}$ are the homeomorphisms between topological spaces. Moreover, every morphism in $\B$ is an isomorphism.

\begin{defi}
    Let $\mathscr{A},\mathscr{B}$ be categories. A functor $F$ from $\mathscr{A}$ to $\mathscr{B}$ is given by the following data:
    \begin{enumerate}
        \item a function $F_{\mathbf{Ob}} : \mathbf{Ob}(\mathscr{A}) \to \mathbf{Ob}(\mathscr{B})$, called the object part of $F$;
        \item a family of functions $F_{\mathbf{Mor}} = \{ F_{C,D} :  \mathbf{Mor}(\mathscr{A})(C,D) \to \mathbf{Mor}(\mathscr{B})(F[C],F[D]) \}_{C, D \in \mathbf{Ob}(\mathscr{A})}$, called the morphism part of $F$, satisfying $F[Id_C] = Id_{F[C]}$ and  $F[g \circ f] = F[g] \circ F[f]$.
    \end{enumerate}
    If $\mathscr{A} = \mathscr{B}$, then we say that $F$ is an endofunctor of $\mathscr{A}$. For instance, for every category 
$\mathscr{A}$, we have an identity functor $Id: \mathscr{A} \to \mathscr{A}$ which fixes all objects and all morphisms. Also, in order to simplify the notation, we will call $F_\mathbf{Ob}$ and $F_{C,D}$ simply by $F$, and the context will make it clear if we refer to the object part or the morphism part of the functor.
\end{defi}

A nontrivial example of an endofunctor is the functor $F: \mathbf{FVect}_\mathbb{K} \to \mathbf{FVect}_\mathbb{K}$
defined on objects by $F(V)=V^{**}$, where $V^{**}=\mathrm{Hom}_\mathbb{K}(\mathrm{Hom}_\mathbb{K}(V, \mathbb{K}), \mathbb{K})$ is the double dual of the vector space $V$; for a linear transformation $T:V \to W$ between finite dimensional $\mathbb{K}$-vector spaces, 
set $F(T)=T^{**}$, where $T^{**}: V^{**} \to W^{**}$ is the linear transformation induced by $T$.
(More precisely, first define $T^*:\mathrm{Hom}_\mathbb{K}(W, \mathbb{K}) \to \mathrm{Hom}_\mathbb{K}(V, \mathbb{K})$ by 
$T^*(f)=f \circ T$, and then let $T^{**}(\phi):=T^* \circ \phi$ for all $\phi \in V^{**}$.)
%\textcolor{blue}{Give an example of a functor. For instance, the identity functor and a functor that takes a vector space to its double dual. Define the notion of an endofunctor. Show that functors take isomorphisms to isomorphisms. What does this say for functors from $\B$?}
\begin{prop}
    Functors preserve isomorphisms, that is, if $\cat, \mathscr{D}$ are two categories, ${F : \mathscr{C} \to \mathscr{D}}$ is a functor, $A,B \in \mathbf{Ob}(\cat)$, and $\sigma : A \to B$ is an isomorphism in $\cat$, then $F[\sigma]$ is an isomorphism in $\mathscr{D}$. Moreover, $F[\sigma]^{-1}=F[\sigma^{-1}]$.
\end{prop}
\begin{proof}
    Functors preserve composition and identities, so 
    \begin{align*}
    &F[\sigma] \circ F[\sigma^{-1}]=F[\sigma \circ \sigma^{-1}]=F[Id_B]=Id_{F[B]}, \text{ and } \\
    &F[\sigma^{-1}] \circ F[\sigma]=F[\sigma^{-1} \circ \sigma]=F[Id_A]=Id_{F[A]}.    
    \end{align*}
    Therefore, $F[\sigma]$ is an isomorphism with inverse $F[\sigma^{-1}]$.
\end{proof}

To conclude, we define the concept of natural transformation.

\begin{defi}
    Let $\mathscr{A}, \mathscr{B}$ be categories, and $F,G$ be functors from $\mathscr{A}$ to $\mathscr{B}$. A natural transformation from $F$ to $G$ is a family of morphisms $\tau = \{ \tau_C : F[C] \to G[C] \}_{C \in \mathbf{Ob}(\mathscr{A})}$ such that
\[\begin{tikzcd}
{F[C]} \arrow{r}{\tau_C} \arrow{d}[left]{F[f]} & {G[C]} \arrow{d}{G[f]} \\
{F[D]} \arrow{r}[below]{\tau_D}                     & {G[D]}                    
\end{tikzcd}\]
commutes for all $f \in \mathbf{Mor}(\mathscr{A})(C,D)$. A natural transformation is called a natural isomorphism when $\tau$ is a family of isomorphisms.
\end{defi}
%\textcolor{blue}{(Define the notion of a natural isomorphism. Give an example. One standard example is the natural isomorphisms from the identity functor on the category of vector spaces to the functor that takes a vector space to its double dual.)}
Let $F: \mathbf{FVect}_\mathbb{K} \to \mathbf{FVect}_\mathbb{K}$ be the double dual functor.
For a finite dimensional $\mathbb{K}$-vector space $V$, let $\tau_V:V \to V^{**}$ be the evaluation map 
defined by $\tau_V(v)[\phi]=\phi(v)$ for all $v \in V$ and $\phi \in V^{*}=\mathrm{Hom}_\mathbb{K}(V, \mathbb{K})$.
Then the family of linear transformations 
$$\tau = \{ \tau_V : Id[V] \to F[V] \}_{V \in \mathbf{Ob}(\mathbf{FVect}_\mathbb{K})}$$ 
defines a natural isomorphism from the identity functor on $\mathbf{FVect}_\mathbb{K}$ to the double dual functor. 
In contrast, even though 
every finite dimensional vector space $V$ is isomorphic to its dual $V^*$, there does not exist a natural way of
identifying $V$ and $V^*$.   

With all of these definitions in hand, we can move ahead and start the study of species of structures.

\section{Species of Structures and Their Associated Series}

The concept of structure is the very root of the subsequent studies. Let us first give an informal definition of structure.

\begin{defi}
   A structure is a pair $s = (U, \gamma)$, where $U$ is a finite set and $\gamma$ is a construction performed on $U$. We call $U$ the underlying set of $s$.
\end{defi}

This definition is somewhat vague, as $\gamma$ could be anything built on the underlying set $U$. Let us then give a few examples.

%\newpage

\begin{exa}
    Let $U = \{ a,b,c,d,e,f \}$, and $\gamma = (\{d\},\{ \{d,a\},\{d,c\},\{c,b\},\{c,f\},\{c,e\} \})$. Then $s = (U,\gamma)$ is the following rooted tree.
    \[\begin{tikzcd}
a & b                                          & f                                                            & e \\
  &                                            & c \arrow[lu, no head] \arrow[u, no head] \arrow[ru, no head] &   \\
  & \circled{$d$} \arrow[luu, no head] \arrow[ru, no head] &                                                              &  
\end{tikzcd} \]
\end{exa}

\begin{exa}
    Let $U = \{ a,b,c,d,e,f \}$, and $\gamma = \{ (c,b),(b,a),(a,d),(d,e),(e,f),(f,c) \}$. Then $s=(U,\gamma)$ is the following cycle.
    \[ \begin{tikzcd}
                         & d \arrow[ld, bend right] &                          \\
e \arrow[d]              &                          & a \arrow[lu, bend right] \\
f \arrow[rd, bend right] &                          & b \arrow[u]              \\
                         & c \arrow[ru, bend right] &                         
\end{tikzcd} \]
\end{exa}

\begin{exa}
    Let $U = \{ a,b,c,d,e,f \}$, and $\gamma = \{\{ a,b,e \},\{ c,f\} , \{ d\}\}$. Then $s = (U, \gamma)$ is a partition of the set $U$.
\end{exa}

%\textcolor{blue}{Maybe is good to give one example that is not a graph.}

Let us now begin the study of species.

\begin{defi}
    A species of structures is an endofunctor $F$ on the category $\B$ of finite sets and bijections. In some cases, it also makes sense to consider $F$ as a functor from $\B$ to $\E$. An $F$-structure is an element $s \in F[U]$ for some finite set $U$. Given a bijection $\sigma:U \to V$ between finite sets, $F[\sigma]$ is called the transport of $F$-structures along $\sigma$.
\end{defi}

Two possible representations of generic $F$-structures are as follows:

\begin{center}
\end{center}

%We can also generalize the concept of isomorphism type defined in Section \ref{isograph}.

\begin{defi}
    Let $s_1 \in F[U], s_2 \in F[V]$ be two $F$-structures. A bijection $\sigma : U \to V$ is an isomorphism from $s_1$ to $s_2$ if $s_2 = F[\sigma](s_1)$. These two structures are then said to be of the same isomorphism type. If $s_1 = s_2$, then $\sigma$ is called an automorphism.
\end{defi}

%Let us then see an example of how these isomorphisms happen in the species of graphs.

\begin{exa}\label{13425}
    The species $\mathcal{G}$ of simple graphs is defined as follows. For a finite set $U$, we set 
    $\mathcal{G}[U] = \{ (U,E) ; E \subseteq \binom{U}{2} \}$, and given a bijection $\sigma : U \to V$ 
    between finite sets, we define $\mathcal{G}[\sigma]: \mathcal{G}[U] \to \mathcal{G}[V]$ by 
    $\mathcal{G}[\sigma]((U, E))=(V, \{\{\sigma(x), \sigma(y)\} ; \{x, y\} \in E\})$ for each $(U, E) \in \mathcal{G}[U]$.
    It is easy to check that $\mathcal{G}$ is indeed an endofunctor on $\B$.
    The following picture, where $U = V = [5]$ and $\sigma = (1\ 3\ 4\ 2\ 5)$, will make the definition of $\mathcal{G}[\sigma]$ clearer.
    \begin{equation*}
      \begin{tikzcd}
                                                  & 2 \arrow[r, no head] \arrow[dd, no head, dotted, bend right] & 4 \arrow[r, no head] \arrow[ld, no head] \arrow[dd, no head, dotted] & 5 \arrow[dd, no head, dotted] \\
1 \arrow[ru, no head] \arrow[dd, no head, dotted] & 3 \arrow[l, no head] \arrow[dd, no head, dotted, bend left]  &                                                                      &                               \\
                                                  & 5 \arrow[r, no head]                                         & 2 \arrow[r, no head]                                                 & 1                             \\
3 \arrow[ru, no head]                             & 4 \arrow[ru, no head] \arrow[l, no head]                     &                                                                      &                              
\end{tikzcd}  
    \end{equation*}
\end{exa}

In what follows, we have some examples of species of structures and the notation used for them throughout this text. The morphism part will be omitted, as it can be deduced naturally (use Example \ref{13425} as a guideline).

\begin{notn}
    The following is a list of species.
    \begin{itemize}
        \item the species $\mathcal{A}$ of rooted trees;
        \item the species $\mathcal{G}$ of simple graphs;
        \item the species $\mathfrak{a}$ of trees;
        \item the species Par of set partitions;
        \item the species $\wp$ of subsets;
        \item the species End of endofunctions;
        \item the species $\mathcal{S}$ of permutations;
        \item the species $\mathcal{C}$ of cyclic permutations;
        \item the species $L$ of linear (or total) orders;
        \item the species $E$ of sets, \emph{i.e.} $E[U] = \{U\}$ for every finite set $U$;
        \item the species $E_n$ of sets with cardinality $n$, \emph{i.e.} $E_n[U] = \{U\}$ if $|U| = n$ and $E_n[U] = \emptyset$ otherwise;
        \item the species $X$ of singletons, \emph{i.e.} $X[U] = \{U\}$ if $|U| = 1$ and $X[U] = \emptyset$ otherwise;
        \item the species $1$, characteristic of the empty set, \emph{i.e.} $1[\emptyset]=\{\emptyset\}$ and $1[U] = \emptyset$ if $|U| > 0$;
        \item the empty species 0, defined as $0[U] = \emptyset$ for all finite sets $U$.
    \end{itemize}
\end{notn}

Let us take a brief intermission to talk about permutations.

\begin{notn}
    Let $\sigma$ be a permutation of the finite set $U$. Then the cycle type of $\sigma$ is the sequence $(\sigma_1, \sigma_2, \ldots)$, where each $\sigma_j$ denotes the number of cycles of size $j$ in the (unique) disjoint cycle decomposition of $\sigma$. In particular, $\mathbf{fix}(\sigma) = \sigma_1$ is the number of elements of $U$ fixed by $\sigma$. Since we work with finite sets, we can ignore the zeroes that end the sequence and denote the cycle type by $(\sigma_1, \ldots, \sigma_n)$, where $n$ is greater than or equal to the maximum length of a cycle of $\sigma$.
\end{notn}

\begin{lemma}\label{permconj}
    Let $\sigma, \tau \in S_n$ be permutations. Then $\perm[\sigma](\tau) = \sigma \tau \sigma^{-1}$.
\end{lemma}
\begin{proof}
    Let the (unique) disjoint cycle decomposition of $\tau$ be $(t_1 \ldots t_{k_1})(t_{k_1 + 1} \ldots t_{k_2})\cdots$. We want to prove the equation \begin{equation*}
        \sigma \tau \sigma^{-1}= \perm[\sigma](\tau) := (\sigma(t_1) \ldots \sigma(t_{k_1}))(\sigma(t_{k_1 + 1}) \ldots \sigma(t_{k_2}))\cdots.
    \end{equation*}
   %Since they are both bijections, it suffices to show that they have the same outcome 
    Given $j \in [n]$, since $\sigma$ is bijective, we have $j = \sigma(i)$ for some $i \in [n]$.
    Without loss of generality, assume $j = \sigma(t_1)$. From the right-hand side, we have ${\perm[\sigma](\tau)(\sigma(t_1)) = \sigma(t_2)}$. 
    On the other hand, from $\tau(t_1) = t_2$, we get ${\sigma \tau \sigma^{-1} (\sigma(t_1)) = \sigma \tau (t_1) = \sigma(t_2)}$.
\end{proof}

We will now start to find a way to generalize what was done with graphs in Section \ref{isograph}. To do so, we must define three series and the concept of labelled and unlabelled structures. Since all transports are bijections, the cardinality of $F[U]$ depends only on that of $U$, and so one can assume $U = [n]$ for some $n \in \mathbb{N}$. To simplify notation, we will write $F[n]$ instead of $F[[n]]$.

\begin{defi}
    Let $F$ be a species of structures. A labelled $F$-structure is an element $s$ of $F[U]$ for some finite set $U$.
\end{defi}

\begin{defi}
    Let $F$ be a species of structures. An unlabelled $F$-structure is an isomorphism type 
    of structures of $F[U]$ for some finite set $U$.
\end{defi}

\begin{defi}
    Let $F$ be a species of structures. The (exponential) generating series of $F$ is the power series
    \begin{equation*}
        F(x) = \sum_{n=0}^\infty f_n \frac{x^n}{n!}
    \end{equation*}
    where $f_n = |F[n]|$. This series enumerates labelled $F$-structures.
\end{defi}

\begin{defi}
    Let $F$ be a species of structures. The (isomorphism) type generating series of $F$ is the power series
    \begin{equation*}
        \widetilde{F}(x) = \sum_{n=0}^\infty \widetilde{f_n}x^n
    \end{equation*}
    where $\Tilde{f_n}$ is the number of isomorphism types (unlabelled $F$-structures) of $F[n]$. More precisely, for $s_1, s_2 \in F[n]$, write $s_1 \sim s_2$ if and only if $s_1$ and $s_2$ are isomorphic 
    $F$-structures on $[n]$; then $\sim$ is an equivalence relation on $F[n]$, and $\widetilde{f_n}$ is exactly the number of equivalence classes of $\sim$.
\end{defi}

\begin{defi}
    Let $F$ be a species of structures. The \emph{Zyklenzeiger} of $F$ is the series
    \begin{equation*}
        Z_F(x_1,x_2,x_3, \ldots.) = \sum_{n=0}^\infty \frac{1}{n!} \sum_{\sigma \in S_n} \mathbf{fix}(F[\sigma]) x_1^{\sigma_1} x_2^{\sigma_2} x_3^{\sigma_3} \ldots 
    \end{equation*}
    where $S_n$ is the symmetric group on $[n]$.
\end{defi}

Without knowledge of operations on species, it is only possible to calculate these series in some simple cases. 
%In the following examples, we will show some of these.

\begin{exa}
    The following generating series are verifiable by direct enumeration.
    \begin{itemize}
        \item $L(x) = \frac{1}{1-x}$, as there are $n!$ possible linear orders of $n$ elements;
        \item $\perm(x) = \frac{1}{1-x}$, as there are $n!$ permutations in the symmetric group on $[n]$;
        \item $\cyc(x) = -\log(1-x)$, as there are $(n-1)!$ cycles on $n$ elements and 
        ${-\log(1-x) = \sum_{n=1}^\infty \frac{x^n}{n}}$;
        \item $E(x) = e^x$, as there is only one element in $E[U]$ for any finite set $U$;
        \item $\wp(x) = e^{2x}$, as there are $2^n$ subsets of $[n]$;
        \item $X(x) = x$;
        \item $1(x) = 1$;
        \item $0(x) = 0$;
        \item $\mathcal{G}(x) = \sum_{n=0}^\infty 2^{\binom{n}{2}} \frac{x^n}{n!}$, as there are $2^{\binom{n}{2}}$ possible choices for the set of edges;
        \item End$(x) = \sum_{n=0}^\infty n^n \frac{x^n}{n!}$, as there are $n^n$ functions in $[n]^{[n]}$.
    \end{itemize}
\end{exa}

\begin{exa}\label{extil}
    The following type generating series are verifiable by direct enumeration.
    \begin{itemize}
        \item $\widetilde{L}(x) = \frac{1}{1-x}$, as all linear orders on $n$ elements are isomorphic;
        \item $\widetilde{\perm}(x) = \sum_{n=0}^\infty p(n) x^n = \prod_{k=1}^\infty \frac{1}{1-x^k}$, where $p(n)$ is the number of partitions of $n$, as two permutations are isomorphic if and only if they have the same cycle type;
        %\textcolor{blue}{Explain this better.}
        \item $\widetilde{\cyc}(x) = \frac{x}{1-x}$ by the same argument as above;
        \item $\widetilde{E}(x) = \frac{1}{1-x}$, as there is only one element in $E[U]$ for any finite set $U$;
        \item $\widetilde{X}(x) = x$;
        \item $\widetilde{1}(x) = 1$;
        \item $\widetilde{0}(x) = 0$.
    \end{itemize}
\end{exa}

Note that for linear orders and permutations the generating series are the same but the type generating series are different. This difference is due to the nature of isomorphisms.

\begin{exa}\label{zykspe}
    The following are \emph{Zyklenzeiger} of some species.
    \begin{itemize}
        \item $Z_0(x_1,x_2,...) = 0$;
        \item $Z_1(x_1,x_2,...) = 1$;
        \item $Z_X(x_1,x_2,...) = x_1$;
        \item $Z_L(x_1,x_2,...) = \frac{1}{1-x_1}$;
        \item $Z_\perm(x_1,x_2,...) = \prod_{k=1}^\infty \frac{1}{1-x_k}$;
        \item $Z_E(x_1,x_2,...) = \exp(\sum_{i=1}^\infty \frac{x_i}{i})$.
    \end{itemize}
\end{exa}
\begin{proof}
We will only show the proof for the last three cases, as the first three are trivial.
    \begin{itemize}
        \item Given any linear order, it is only fixed by the identity permutation, and therefore $\mathbf{fix}(L[\sigma]) \neq 0$ if and only if $\sigma = Id$. We then have $$Z_L(x_1,x_2,...) = \sum_{n=0}^\infty \frac{1}{n!} \sum_{\sigma \in S_n} \mathbf{fix}(L[\sigma]) x_1^{\sigma_1} x_2^{\sigma_2}... = \sum_{n=0}^\infty \frac{1}{n!} \mathbf{fix}(L[Id_n])x_1^n = \sum_{n=0}^\infty x_1^n = \frac{1}{1-x_1}.$$
        \item From Lemma \ref{permconj}, a permutation $\tau$ in $\perm[n]$ is fixed by $\perm[\sigma]$, $\sigma \in S_n$, if and only if 
        $\sigma \tau \sigma^{-1} = \tau$, or equivalently, $\sigma \tau  =\tau \sigma$. Therefore, 
        $\mathbf{fix}(\perm[\sigma])$ is equal to the order of the centralizer of $\sigma$ in $S_n$, which in turn is equal to
        $\displaystyle \frac{n!}{|\mathcal{C}_{\sigma}|}$, where $\mathcal{C}_{\sigma}$ denotes the conjugacy class of $\sigma$.
        Note that two permutations are conjugate in $S_n$ if and only if they have the same cycle type.
        By \cite[Th. 6.6.6]{wagner}, there are $\displaystyle \frac{n!}{\prod i^{a_i}a_i!}$ permutations in $S_n$ with cycle type $(a_1,a_2,...,a_n)$. Consequently, $\mathbf{fix}(\perm[\sigma])=\prod_{j \geq 1} j^{\sigma_j} \sigma_j!$.
        It follows that
        %takes cycles of $\tau$ into cycles of $\tau$ (of same size). Therefore $\mathbf{fix}(\perm[\sigma]) = \prod_{j \geq 1} j^{\sigma_j} \sigma_j!$, with $j^{\sigma_j}$ counting the intracyclic part, and $\sigma_j!$ counting the intercyclic part. From \cite[Th. 6.6.6]{wagner} we have that there are $\displaystyle \frac{n!}{\prod i^{a_i}a_i!}$ elements of $S_n$ with cycle type $(a_1,a_2,...,a_n)$, and hence 
        \begin{align*}
            Z_\perm(x_1,x_2, \ldots) &= \sum_{n=0}^\infty \frac{1}{n!} \sum_{\sigma \in S_n} \mathbf{fix}(\perm[\sigma])x_1^{\sigma_1}x_2^{\sigma_2}\ldots \\ 
            &= \sum_{n=0}^\infty \frac{1}{n!} \underset{+na_n=n}{\sum_{a_1+2a_2+}} \Bigg[\frac{n!}{\prod_{i=1}^n i^{a_i}a_i!} (\prod_{i=1}^n i^{a_i}a_i!) x_1^{a_1} \cdots x_n^{a_n} \Bigg] \\&= \sum_{n=0}^\infty \underset{+na_n=n}{\sum_{a_1+2a_2+ \ldots }} x_1^{a_1} \cdots x_n^{a_n} = \prod_{k=1}^\infty \frac{1}{1-x_k}.
        \end{align*}
        \item Given any permutation $\sigma$ on a finite set $U$, $E[\sigma]$ fixes the unique element of $E[U]$, from where we have \begin{align*}
            Z_E(x_1,x_2, \ldots) &= \sum_{n=0}^\infty \frac{1}{n!} \sum_{\sigma \in S_n} x_1^{\sigma_1}x_2^{\sigma_2} \ldots = \sum_{n=0}^\infty \frac{1}{n!}\underset{+n\sigma_n=n}{\sum_{\sigma_1+2\sigma_2+ \ldots}} \frac{n!}{\prod_{i=1}^n i^{\sigma_i}\sigma_i!} x_1^{\sigma_1}x_2^{\sigma_2} \cdots x_n^{\sigma_n} \\&= \sum_{\sigma_1 \geq 0, \sigma_2 \geq 0, \ldots} \prod_{i \geq 1} \frac{x_i^{\sigma_i}}{i^{\sigma_i}\sigma_i!} = \prod_{i \geq 1} \sum_{k \geq 0} \frac{(x_i / i)^{k}}{k!} = \prod_{i \geq 1} \exp(\frac{x_i}{i}) = \exp(\sum_{i=1}^\infty \frac{x_i}{i}).
        \end{align*}
    \end{itemize}
\end{proof}

Just like in the previous chapter, the \emph{Zyklenzeiger} plays a fundamental part in connecting different enumeration techniques. In the study of species, it connects labelled and unlabelled enumeration.

\begin{theorem}
    For any species of structures $F$, we have $F(x) = Z_F(x,0,0,...)$, and \newline$\widetilde{F}(x) = Z_F(x,x^2,x^3,...)$.
\end{theorem}
\begin{proof}
    For the first equality, we have $Z_F(x,0,0,...) = \sum_{n \geq 0} \frac{1}{n!} \sum_{\sigma \in S_n} \mathbf{fix}(F[\sigma])x^{\sigma_1}0^{\sigma_2 + \sigma_3 + ...}$, and for each $n$ this second summation is non-zero if and only if $\sigma_1 = n$, that is, $\sigma = Id_n$; thus $Z_F(x,0,0,...) = \sum_{n \geq 0} \frac{1}{n!} \mathbf{fix}(F[Id_n])x^n = \sum_{n \geq 0} \frac{1}{n!} |F[n]| x^n = F(x)$.

    For the second equality, consider the equivalence relation $\til$ given by $s \til t \iff s,t$ are isomorphic. Then we use Lemma \ref{notBurn} to show $$ Z_F(x,x^2,x^3,...) = \sum_{n \geq 0} \frac{1}{n!} \sum_{\sigma \in S_n} \mathbf{fix}(F[\sigma])x^n = \sum_{n \geq 0} |F[n]/\til| x^n = \widetilde{F}(x). $$
\end{proof}

\section{Equality and Addition of Species}

In the next few sections we will present some operations on species of structures and use them to build and enumerate more complex species. 
We will assume that the reader is familiar with Möbius inversion (cf. \cite[Ch. 8]{wagner}). Firstly, one might ask when two species of structures are essentially the same, in the sense that they have the same combinatorial properties. 
%This comes from the concept of combinatorial equality and is an essential idea for the rest of this article.

\begin{defi}
    Two species $F,G$ are said to be (combinatorially) equal, and one writes $F \simeq G$ or $F = G$, if there exists a natural isomorphism between them. One can also say that $F$ and $G$ are isomorphic.
\end{defi}

Note, though, that having $|F[n]| = |G[n]|$ for all $n \in \N$ is not a sufficient condition for $F$ and $G$ to be isomorphic: for instance, $L$ and $\perm$ have this property but are not isomorphic.

Let us now begin presenting some operations on species. The first operation is addition.

\begin{defi}
    Let $F,G$ be two species of structures. The species $F+G$ is called the sum of $F$ and $G$, and is defined in the following manner: an $(F+G)$-structure on a finite set $U$ is an $F$-structure on $U$ xor a $G$-structure on $U$. More formally, $$(F+G)[U] = F[U] \sqcup G[U]$$ $$(F+G)[\sigma](s) = \begin{cases}
        F[\sigma](s), \text{if } s \in F[U],\\
        G[\sigma](s), \text{if } s \in G[U].
    \end{cases} $$
\end{defi}

In pictures, this operation can be visualised as
\begin{center}

\tikzset{every picture/.style={line width=0.75pt}} %set default line width to 0.75pt        


\end{center}

The definition of $F+G$ requires $F[U] \cap G[U] = \emptyset$. If this is not the case, one must first create disjoint copies of these sets. One can do this by replacing $F[U]$ with $F[U] \times \{1\}$, and $G[U]$ with $G[U] \times \{2\}$. Then we define $(F+G)[U] = (F[U] \times \{1\}) \sqcup (G[U] \times \{2\})$.

The operation of addition is associative and commutative (up to isomorphism). More so, the species 0 is the neutral element of this operation, that is, for any species of structures $F$, ${F+0=0+F=F}$.

\begin{theorem}
    Let $F,G$ be species of structures. Then
    \begin{itemize}
        \item[ ] $(F+G)(x) = F(x) + G(x);$
        \item[ ] $\widetilde{(F+G)}(x) = \widetilde{F}(x) + \widetilde{G}(x);$
        \item[ ] $Z_{F+G}(x_1,x_2,...) = Z_F(x_1,x_2,...) + Z_G(x_1,x_2,...)$
    \end{itemize}
\end{theorem}
\begin{proof}
    Immediate from the definition.
\end{proof}

This concept of addition can be extended to infinite sums in the following case.

\begin{defi}
    A family $\{ F_i \}_{i \in I}$ of species is said to be summable if for any finite set $U$, we have $F_i[U] \neq \emptyset$  
    for only finitely many $i$. In this case,
    \begin{itemize}
        \item[ ] $ (\sum_{i \in I} F_i)[U] = \bigsqcup_{i \in I} (F_i[U] \times \{i\}) $
        \item[ ] $ (\sum_{i \in I} F_i)[\sigma](s,i) = (F_i[\sigma](s),i) $
        \item[ ] $ (\sum_{i \in I} F_i)(x) = \sum_{i \in I} F_i(x) $
        \item[ ] $ \widetilde{(\sum_{i \in I} F_i)}(x) = \sum_{i \in I} \widetilde{F_i}(x) $
        \item[ ] $ Z_{\sum_{i \in I} F_i}(x_1,x_2,...) = \sum_{i \in I} Z_{F_i}(x_1,x_2,...) $
    \end{itemize}
\end{defi}

The definition of summable families of species allows us to define a canonical decomposition for any species of structures. 
%analogous to that of $E$ given by the $E_n$.

\begin{defi}
    For every species of structures $F$ there is an enumerable summable family of species $\{ F_n \}_{n \in \N}$, given by $F_n[U] = F[U]$ if $|U| = n$, and $F_n[U] = \emptyset$ otherwise, with the transports induced naturally by those of $F$. 
    We then have $F = F_0 + F_1 + F_2 + \ldots$. We also denote by $F_+$ the species $F_1 + F_2 + \ldots$.
\end{defi}

One more use of addition is the introduction of natural numbers as species of structures.

\begin{defi}
    For all $\mathbf{n} \in \N$, one can define the species $n := 1 + 1 + \ldots + 1$, adding the species $1$ exactly $\mathbf{n}$ times. The species $n$ has $\mathbf{n}$ structures on the empty set and none on sets of positive cardinality, and these structures are all non-isomorphic.
\end{defi}

\section{The Two Products}

The second operation that will be presented is the superposition, also known as the Cartesian product.

\begin{defi}
    Let $F,G$ be two species of structures. Then $F \times G$, the superposition of $F$ and $G$, is defined as follows: an $(F \times G)$-structure on a finite set $U$ is a pair $s = (f,g)$, where $f \in F[U]$ and $g \in G[U]$, and the transport is given by $(F \times G)[\sigma](f,g) = (F[\sigma](f),G[\sigma](g))$. These can be visualized by the following image.
    \begin{center}

\tikzset{every picture/.style={line width=0.75pt}} %set default line width to 0.75pt        


    \end{center}
\end{defi}

In order to relate the series of $F \times G$ to the series of $F$ and $G$, we need to define the Hadamard product of series.   
\begin{defi}
    Let $\mathcal{X} = (x_1,x_2,...)$ be a sequence of variables. Given a sequence of natural numbers 
    $\mathcal{N} = (n_1,n_2,...)$ containing only finitely many non-zero entries, let 
    $\mathcal{X}^\mathcal{N} = x_1^{n_1} x_2^{n_2} \ldots$ and $\text{aut}(\mathcal{N}) = 1^{n_1}n_1! 2^{n_2}n_2! \ldots$.
    Now if $f(\mathcal{X}) = \sum_{\mathcal{N}} f_\mathcal{N} \frac{\mathcal{X}^\mathcal{N}}{\text{aut}(\mathcal{N})}$ and $g(\mathcal{X}) = \sum_{\mathcal{N}} g_\mathcal{N} \frac{\mathcal{X}^\mathcal{N}}{\text{aut}(\mathcal{N})}$, we define 
    \begin{equation*}
        (f \times g)(\mathcal{X}) = \sum_{\mathcal{N}} f_\mathcal{N} g_\mathcal{N} \frac{\mathcal{X}^\mathcal{N}}{\text{aut}(\mathcal{N})}.
    \end{equation*}
\end{defi}

%This then extends Cartesian product to series in the following sense.
Let $F$ be a species of structures, $n \in \mathbb{N}$ and $\sigma_1, \sigma_2 \in S_n$.
Suppose that $\sigma_1$ and  $\sigma_2$ are conjugate in $S_n$. Then there exists $\tau \in S_n$ such that 
$\tau \sigma_1 \tau^{-1}=\sigma_2$. Since $F$ is a functor, we have that $F[\tau] F[\sigma_1] F[\tau]^{-1}= F[\sigma_2]$, and thus
$F[\sigma_1]$ and $F[\sigma_2]$ are conjugate permutations on $F[n]$. As a consequence of Lemma \ref{permconj}, conjugate permutations have the same cycle 
structure, from where we get that $\mathbf{fix}(F[\sigma_1])=\mathbf{fix}(F[\sigma_2])$. Therefore, 
grouping together conjugate permutations, we can rewrite the \emph{Zyklenzeiger} as
\begin{equation}\label{z48}
    Z_F(x_1,x_2,\ldots) = \sum_{n \geq 0} \underset{+na_n=n}{\sum_{a_1+2a_2+...}} \mathbf{fix}(F[\sigma])\frac{x_1^{a_1} \cdots x_n^{a_n}}{\prod_{i=1}^n i^{a_i}a_i!},
\end{equation}   
where $\sigma \in S_n$ is any permutation with cycle type $(a_1, a_2, \ldots, a_n)$. We
thus obtain a description of the \emph{Zyklenzeiger} consistent with the notation of the definition above.

\begin{theorem}
    Let $F,G$ be species of structures. Then
    \begin{itemize}
        \item[ ] $(F \times G)(x) = F(x) \times G(x)$
        \item[ ] $\widetilde{(F \times G)}(x) = (Z_F \times Z_G)(x,x^2,x^3,...)$
        \item[ ] $Z_{F \times G}(x_1,x_2,...) = Z_F(x_1,x_2,...) \times Z_G(x_1,x_2,...)$
    \end{itemize}
\end{theorem}
\begin{proof}
    It is sufficient to prove the last equality, which is trivial, as an $(F \times G)$-structure is fixed if and only if each of the two structures which compose it are fixed, from where $\mathbf{fix}((F \times G)[\sigma]) = \mathbf{fix}(F[\sigma])\mathbf{fix}(G[\sigma])$, and then
    \begin{align*}
        Z_{F \times G}(x_1,x_2,\ldots) &= \sum_{n \geq 0} \underset{+na_n=n}{\sum_{a_1+2a_2+...}} \mathbf{fix}((F \times G)[\sigma]) \frac{x_1^{a_1} \cdots x_n^{a_n}}{\prod_{i=1}^n i^{a_i}a_i!} \\&= \sum_{n \geq 0} \underset{+na_n=n}{\sum_{a_1+2a_2+...}} \mathbf{fix}(F[\sigma])\mathbf{fix}(G[\sigma]) \frac{x_1^{a_1} \cdots x_n^{a_n}}{\prod_{i=1}^n i^{a_i}a_i!} \\&= Z_F(x_1,x_2,\ldots) \times Z_G(x_1,x_2,\ldots).
    \end{align*}
\end{proof}

%\[ Z_F = \sum_{n \geq 0} \sum_{n_1 + ... + n_k = n} \mathbf{fix}(F[\sigma]) \frac{x_1^{n_1} \cdots x_k^{n_k}}{1^{n_1}n_1! \cdots k^{n_k}n_k!} \]

The neutral element of superposition is the set species $E$, that is, for any species $F$, $E \times F = F \times E = F$, and also $E_n \times F = F \times E_n = F_n$. However, the law of cancellation is not valid in the case of Cartesian products, that is, $A \times B = A \times C$ does not imply $B = C$. For example, one can show that $L \times L = L \times \perm$, but $L \neq \perm$.

Let us now move on to the next operation: the product.

\begin{defi}
    Let $F,G$ be species of structures. The species $F \cdot G$ (also $FG$) is called the product of $F$ and $G$, and is defined as follows: an $FG$-structure on a finite set $U$ is an ordered pair $s = (f,g)$, where $f \in F[U_1]$, $g \in G[U_2]$, $U_1 \cup U_2 = U$ and $U_1 \cap U_2 = \emptyset$. In other terms,
    \begin{equation*}
        (FG)[U] = \sum_{U_1 \sqcup U_2 = U} (F[U_1] \times G[U_2])
    \end{equation*}
    and the transport is given by
    \begin{equation*}
        (FG)[\sigma](f,g) = (F[\sigma|_{U_1}](f),G[\sigma|_{U_2}](g))
    \end{equation*}\newpage
    In pictures, an $FG$-structure can be seen in these two manners:
    \begin{center}

\tikzset{every picture/.style={line width=0.75pt}} %set default line width to 0.75pt        


    \end{center}
\end{defi}

The product of species is associative and commutative (up to isomorphism), and the species $1$ is the neutral element, that is, 
$1F = F1 = F$. More so, $0F = F0 = 0$, and the distributive property of multiplication over addition holds.

\begin{theorem}
    Let $F,G$ be species of structures. Then
    \begin{itemize}
        \item[ ] $ (FG)(x) = F(x)G(x) $
        \item[ ] $ \widetilde{(FG)}(x) = \widetilde{F}(x)\widetilde{G}(x) $
        \item[ ] $ Z_{FG}(x_1,x_2,...) = Z_F(x_1,x_2,...)Z_G(x_1,x_2,...) $
    \end{itemize}
\end{theorem}
\begin{proof}
    It is sufficient to prove the third equality. Observe that an $FG$-structure is fixed by a permutation if and only if its two parts are both fixed. Hence, in order for a permutation $\sigma$ of $[n] = U_1 \sqcup U_2$ to fix a structure in 
    $F[U_1] \times F[U_2]$, we must have $\sigma(U_1) = U_1$ and $\sigma(U_2) = U_2$. Using that there are $\binom{n}{i}$ ways to choose a subset $U_1$ with cardinality $i$ from $[n]$, we get 
    \begin{align*}
        &Z_{FG}(x_1,x_2,...) = \sum_{n \geq 0} \frac{1}{n!} \sum_{U_1 \sqcup U_2 = [n]} \sum_{\sigma^1 \in S(U_1)} \sum_{\sigma^2 \in S(U_2)} \mathbf{fix}(F[\sigma^1])\mathbf{fix}(G[\sigma^2]) x_1^{\sigma_1^1 + \sigma_1^2} x_2^{\sigma_2^1 + \sigma_2^2}\ldots \\
        &= \sum_{n \geq 0} \frac{1}{n!} \underset{i,j \geq 0}{\sum_{i+j=n}} \binom{n}{i} \sum_{\sigma^1 \in S_i} \mathbf{fix}(F[\sigma^1]) x_1^{\sigma_1^1} x_2^{\sigma_2^1}... \sum_{\sigma^2 \in S_j} \mathbf{fix}(G[\sigma^2]) x_1^{\sigma_1^2} x_2^{\sigma_2^2}\ldots \\
        &= \sum_{i \geq 0} \sum_{j \geq 0} \frac{1}{i!j!} \sum_{\sigma^1 \in S_i} \mathbf{fix}(F[\sigma^1]) x_1^{\sigma_1^1} x_2^{\sigma_2^1}... \sum_{\sigma^2 \in S_j} \mathbf{fix}(G[\sigma^2]) x_1^{\sigma_1^2} x_2^{\sigma_2^2} \ldots \\
        &= \sum_{i \geq 0}\frac{1}{i!} \sum_{\sigma^1 \in S_i} \mathbf{fix}(F[\sigma^1]) x_1^{\sigma_1^1} x_2^{\sigma_2^1}... \sum_{j \geq 0} \frac{1}{j!} \sum_{\sigma^2 \in S_j} \mathbf{fix}(G[\sigma^2]) x_1^{\sigma_1^2} x_2^{\sigma_2^2}... = Z_F(x_1,x_2,...)Z_G(x_1,x_2,...).
    \end{align*}
\end{proof}

We can now use this last theorem to enumerate structures that decompose as products.

\begin{exa}
    A permutation can be seen as a set of fixed points alongside a derangement (a permutation without fixed points).
\begin{center}

\tikzset{every picture/.style={line width=0.75pt}} %set default line width to 0.75pt        


    \end{center}
    We can then write $\perm = E \cdot \mathrm{Der}$, and since we already have the series for $\perm$ and $E$, use this to find the series for $\mathrm{Der}$ and prove the well-known formula for derangements: $\displaystyle {d_n = \sum_{i=0}^n \frac{(-1)^in!}{i!}}$.\newpage We have
    \begin{itemize}
        \item[ ] $\frac{1}{1-x} = e^x \text{Der}(x)$
        \item[ ] $\prod_{k \geq 1} \frac{1}{1-x^k} = \frac{1}{1-x} \widetilde{\text{Der}}(x)$
        \item[ ] $\prod_{k \geq 1} \frac{1}{1-x_k} = e^{x_1 + \frac{x_2}{2} + \frac{x_3}{3} + \ldots} Z_{\text{Der}}(x_1,x_2, \ldots)$
    \end{itemize}
    and by simple calculations we get
    \begin{itemize}
        \item[ ] Der$(x) = \frac{e^{-x}}{1-x}$
        \item[ ] $\widetilde{\text{Der}}(x) = \prod_{k \geq 2} \frac{1}{1-x^k}$
        \item[ ] $Z_{\text{Der}}(x_1,x_2, \ldots) = e^{-(x_1 + \frac{x_2}{2} + \frac{x_3}{3} + \ldots)}\prod_{k \geq 1} \frac{1}{1-x_k}$
    \end{itemize}
    From the first formula above, we deduce
    \begin{equation*}
        \sum_{n \geq 0} d_n \frac{x^n}{n!} = \left( \sum_{i \geq 0} (-1)^i \frac{x^i}{i!} \right) \left( \sum_{j \geq 0} x^j \right) \implies d_n = \sum_{i=0}^n \frac{(-1)^i n!}{i!}.
    \end{equation*}
\end{exa}

\begin{exa}
    Given a set $U$, a subset $V \subseteq U$ can also be seen as an ordered pair $(V,U \setminus V)$. This gives the combinatorial equality $\wp = E \cdot E = E^2$, and from that the formula $\wp(x) = e^{2x}$ and another proof to the fact that a set with $n$ elements has $2^n$ possible subsets. We can also consider the subspecies $\wp^{[k]} = E_k \cdot E$ of subsets with cardinality $k$, from where we have $\wp^{[k]}(x) = e^x \frac{x^k}{k!}$, which expands to
    \begin{equation*}
        \sum_{n \geq 0} |\wp^{[k]}[n]|\frac{x^n}{n!} = \sum_{m\geq0} \frac{x^{m+k}}{m!k!} = \sum_{n \geq k} \frac{x^n}{k!(n-k)!} \implies |\wp^{[k]}[n]| = \frac{n!}{k!(n-k)!} = \binom{n}{k},
    \end{equation*}
    coming back to the definition of binomial coefficients. Since $\wp = \sum_{k=0}^n \wp^{[k]}$, we have a proof for the well-known identity $\sum_{k=0}^n \binom{n}{k} = 2^n$.
\end{exa}

\begin{exa}
    Given a species $F$, one has $nF = F + F + \ldots + F$ ($n$ terms). %This comes from the fact that $n$ is the species that creates $n$ structures from the empty set, and no structure from any other set. This justifies the introduction of natural numbers into the space of species.
\end{exa}

\begin{exa}
    A linear order of $n$ elements can be seen as an ordered $n$-tuple of singletons, from where we have the identity $L_k = X^k$.
\end{exa}

\section{Composition of Species}

Moving forward, we have the next operation: composition.

\begin{defi}
    Let $F,G$ be two species of structures with $G[\emptyset] = \emptyset$. The species $F \circ G$ (also $F(G)$), called the (partitional) composite of $G$ in $F$, is defined as follows: an $F \circ G$-structure on a finite set $U$ is a triplet $s = (\pi,\vfi,\gamma)$ with
    \begin{enumerate}
        \item $\pi$ a partition on $U$;
        \item $\vfi$ an $F$-structure on the set of parts of $\pi$;
        \item $\gamma = \{ \gamma_p \}_{p \in \pi}$ a family of $G$-structures, where for each $p \in \pi$, $\gamma_p$ is a $G$-structure on $p$.
    \end{enumerate}
    In other words,
    \begin{equation*}
        (F \circ G)[U] = \sum_{\pi \in \text{Par}[U]} (F[\pi] \times \prod_{p \in \pi} G[p] )
    \end{equation*}
    and the transport is given by
    \begin{equation*}
        (F \circ G)[\sigma](\pi,\vfi,\gamma) = (\overline{\pi},\overline{\vfi},\{\overline{\gamma}_{\overline{p}}\}_{\overline{p} \in \overline{\pi}}),
    \end{equation*}
    where
    \begin{enumerate}
        \item $\overline{\pi}$ is the partition obtained by transport of $\pi$ along $\sigma$;
        \item for each $\overline{p} = \sigma(p) \in \overline{\pi}$, $\overline{\gamma}_{\overline{p}}$ is the $G$-structure obtained by $G$-transport of $\gamma_p$ along $\sigma|_{p}$;
        \item $\overline{\vfi}$ is the $F$-structure obtained by $F$-transport of $\vfi$ along the bijection $\overline{\sigma} : \pi \to \overline{\pi}$ induced by $\sigma$.
    \end{enumerate}
    An $F \circ G$-structure can also be interpreted as an $F$-assembly of $G$-structures. When $F = E$, we simply say assembly instead of $E$-assembly. These are more easily understood with pictures.
    \begin{center}

\tikzset{every picture/.style={line width=0.75pt}} %set default line width to 0.75pt        


    \end{center}
\end{defi}

The neutral element of composition is the species $X$ of singletons, that is, for all species $F$, we have $F = F(X) = X(F)$. 
Moreover, composition is associative (up to isomorphism).

The definition of composition is naturally extended to the exponential generating series and the \emph{Zyklenzeiger}, though it is not extended to the type generating series in such a simple manner. The proof of this is not as direct as in the other cases.

\begin{theorem}
    Let $F,G$ be species of structures with $G[\emptyset] = \emptyset$. Then
    \begin{itemize}
        \item[ ] $(F \circ G)(x) = F(G(x))$
        \item[ ] $\widetilde{(F \circ G)}(x) = Z_F(\widetilde{G}(x),\widetilde{G}(x^2),\widetilde{G}(x^3), \ldots )$
        \item[ ] $Z_{F \circ G}(x_1,x_2,x_3,\ldots) = Z_F(Z_G(x_1,x_2,x_3,\ldots),Z_G(x_2,x_4,x_6, \ldots),Z_G(x_3,x_6,x_9, \ldots),\ldots)$
    \end{itemize}
\end{theorem}
\begin{proof}
    It is enough to prove the third equality. This requires a characterization of the \emph{Zyklenzeiger}, and can be found in \cite[Section 4.3]{bergeron2}.
\end{proof}

Let us now see some examples of composition of species.
%\newpage
\begin{exa}
    An endofunction can be seen as a permutation of rooted trees, as illustrated in the following picture.
    \begin{center}

\tikzset{every picture/.style={line width=0.75pt}} %set default line width to 0.75pt        


    \end{center}
    This translates, in mathematical terms, to $\text{End} = \perm \circ \rooted$.
\end{exa}

\begin{exa}\label{axea}
    A rooted tree can be seen as a singleton and an assembly of rooted trees, that is, $\rooted = X \cdot E(\rooted)$.
    \begin{center}

\tikzset{every picture/.style={line width=0.75pt}} %set default line width to 0.75pt        

%
    \end{center}
\end{exa}

\begin{exa}
    A partition can be seen as a set of non-empty sets, that is, $\text{Par} = E(E_+)$. This gives us a formula for enumerating partitions: $\text{Par}(x) = e^{e^x - 1}$.
\end{exa}

\begin{defi}
    The Möbius function $\mu : \N \setminus \{0\} \to \Z$ is defined as
    \begin{equation*}
        \mu(n) = \begin{cases}
            1\text{, if }n = 1;\\
            (-1)^r\text{, if }n\text{ is a product of }r\text{ distinct primes};\\
            0\text{, otherwise}.
        \end{cases}
    \end{equation*}
    This function satisfies the following property.
    \begin{equation*}
        \sum_{d|n} \mu(d) = \begin{cases}
            1,\text{ if }n=1;\\
            0,\text{ otherwise.}
        \end{cases}
    \end{equation*}
    This is clear in the case $n=1$. For $n = p_1^{e_1} \cdots p_r^{e_r} > 1$, we have 
    \begin{equation*}
        \sum_{d | n} \mu(d) = \sum_{d | p_1 \cdots p_r} \mu(d) = \sum_{k=0}^r (-1)^k \binom{r}{k} = 0.
    \end{equation*}
\end{defi}

\begin{lemma}[Möbius Inversion, series version]
    Let $a(x_1,x_2,\ldots), b(x_1,x_2,\ldots)$ be two power series. Denote $a_k = a(x_k,x_{2k},\ldots), b_k = b(x_k,x_{2k},\ldots)$. Then
    \begin{equation*}
        b_1 = \sum_{k \geq 1} \frac{1}{k} a_k \iff a_1 = \sum_{k \geq 1} \frac{\mu(k)}{k} b_k.
    \end{equation*}
\end{lemma}
\begin{proof}
    Assume the left-hand side equation holds. Then
    \begin{equation*}
        \sum_{k \geq 1} \frac{\mu(k)}{k} b_k = \sum_{k \geq 1} \frac{\mu(k)}{k} \sum_{d \geq 1} \frac{1}{d} a_{dk} = \sum_{k \geq 1} \mu(k) \sum_{d \geq 1} \frac{1}{dk} a_{dk} = \sum_{n \geq 1} \frac{1}{n} a_n \sum_{k|n} \mu(k) = a_1.
    \end{equation*}

    Now assume the right-hand side equation holds. Then
    \begin{equation*}
        \sum_{k \geq 1} \frac{1}{k} a_k = \sum_{k \geq 1} \frac{1}{k} \sum_{d \geq 1} \frac{\mu(d)}{d} b_{dk} = \sum_{k \geq 1} \sum_{d \geq 1} \frac{\mu(d)}{dk} b_{dk} = \sum_{n \geq 1} \frac{1}{n} b_n \sum_{d|n} \mu(d) = b_1.
    \end{equation*}
\end{proof}

\begin{exa}\label{permacyc}
    A permutation is an assembly of cycles, that is, $\perm = E(\cyc)$. This gives us another proof for the identity $\cyc(x) = -\log(1-x)$, as it comes from $\frac{1}{1-x} = \perm(x) = e^{\cyc(x)}$. This also allows us to calculate the \emph{Zyklenzeiger} of the species of cycles by using Möbius inversion.
    \begin{equation*}
        \prod_{k \geq 1} \frac{1}{1-x_k} = \exp \left( \sum_{k \geq 1} \frac{1}{k} Z_\cyc(x_k,x_{2k},...) \right) \implies \log\left( \prod_{k \geq 1} \frac{1}{1-x_k}\right) = \sum_{k \geq 1}\frac{1}{k} Z_\cyc(x_k,x_{2k},...) 
    \end{equation*}
    \begin{equation*}
        \implies Z_\cyc(x_1,x_2,...) = \sum_{k \geq 1} \frac{\mu(k)}{k} \log \left( \prod_{s \geq 1} \frac{1}{1-x_{ks}} \right) = \sum_{k \geq 1} \sum_{s \geq 1} \frac{\mu(k)}{k} \log \frac{1}{1-x_{ks}}.
    \end{equation*}
    This last series can be rewritten as 
    \begin{equation*}
        \sum_{k \geq 1} \sum_{d|k} \frac{\mu(k/d)}{k/d} \log \frac{1}{1-x_k} = \sum_{k \geq 1} \frac{1}{k} \log \left( \frac{1}{1-x_k} \right) \sum_{d|k} {d\mu(k/d)} = \sum_{k \geq 1} \frac{\phi(k)}{k} \log \frac{1}{1-x_k},
    \end{equation*}
    where $\phi$ is Euler's totient function. By taking $\widetilde{\cyc}(x)$ as in Example \ref{extil}, this also gives us the neat identity
    \begin{equation*}
        \frac{x}{1-x} = \sum_{k \geq 1} \frac{\phi(k)}{k} \log \frac{1}{1-x^k}.
    \end{equation*}
\end{exa}

Composition of species also gives us a proper way to define connected structures.

\begin{defi}
    Let $F$ be a species of structures. Then the species $F^c$ of connected $F$-structures is defined as the species such that $F = E(F^c)$.
\end{defi}

This naturally gives us the following series.

\begin{itemize}
    \item[ ] $F(x) = e^{F^c(x)}$
    \item[ ] $\widetilde{F}(x) = \exp \sum_{k \geq 1} \frac{1}{k} \widetilde{F^c}(x^k)$
    \item[ ] $Z_F(x_1,x_2,...) = \exp \sum_{k \geq 1} \frac{1}{k} Z_{F^c}(x_k,x_{2k},...)$
\end{itemize}

By using Möbius inversion just like in Example \ref{permacyc}, we can properly calculate the series of any species of connected structures from the original species.

\begin{itemize}
    \item[ ] $F^c(x) = \log F(x)$
    \item[ ] $\widetilde{F^c}(x) = \sum_{k \geq 1} \frac{\mu(k)}{k} \log \widetilde{F}(x^k)$
    \item[ ] $Z_{F^c}(x_1, x_2, \ldots) = \sum_{k \geq 1} \frac{\mu(k)}{k} \log Z_F(x_k,x_{2k}, \ldots)$.
\end{itemize}

It is also worth mentioning that there exists a second type of composition of species, namely functorial composition, denoted by $F$\scalebox{0.75}{$\square$}$G$. The interested reader can find more about it in \cite[Section 2.4]{bergeron1}.

\section{Derivatives and Pointing}

Next comes the operation of derivatives.

\begin{defi}
    Let $F$ be a species of structures. The derivativee of $F$, denoted $F'$ or $\frac{d}{dX}F(X)$, is defined as follows: an $F'$-structure on a finite set $U$ is an $F$-structure on the set $U^+ := U \sqcup \{ * \}$, where $* = *_U$ is an element outside of $U$. In other words, $F'[U] = F[U^+]$, and the transport along $\sigma : U \to V$ is given by $F'[\sigma] = F[\sigma^+]$, where $\sigma^+ : U^+ \to V^+$ is such that $\sigma^+|_U = \sigma$ and $\sigma^+(*_U) = *_V$. In pictures, this translates to
    \begin{center}

\tikzset{every picture/.style={line width=0.75pt}} %set default line width to 0.75pt        


    \end{center}
    One can also see an $F'$-structure as an $F$-structure with a hole, associating $*$ with the removal of a vertex.
\end{defi}

Naturally, the derivative extends to the exponential generating series and the \emph{Zyklenzeiger}.

\begin{theorem}
    Let $F$ be a species of structures. Then
    \begin{itemize}
        \item[ ] $F'(x) = \frac{d}{dx} F(x)$
        \item[ ] $\widetilde{F'}(x) = \left( \frac{\del}{\del x_1} Z_F \right)(x,x^2,x^3,...)$
        \item[ ] $Z_{F'}(x_1,x_2,x_3,...) = \left( \frac{\del}{\del x_1} Z_F \right)(x_1,x_2,x_3,...)$
    \end{itemize}
\end{theorem}
\begin{proof}
    It is sufficient to prove the third equality. We have 
    \begin{equation}\label{zf'1}
        Z_{F'}(x_1,x_2,x_3,\ldots) = \sum_{n \geq 0} \frac{1}{n!} \sum_{\sigma \in S_n} \mathbf{fix}(F'[\sigma]) x_1^{\sigma_1} x_2^{\sigma_2} \ldots =  \sum_{n \geq 0} \frac{1}{n!} \sum_{\sigma \in S_n} \mathbf{fix}(F[\sigma^+]) x_1^{\sigma_1} x_2^{\sigma_2} \ldots
    \end{equation}
    Consider a sequence $\alpha = (a_1,a_2, \ldots )$ of non-negative integers. We say that $\alpha$ partitions $n$, and write $\alpha \vdash n$, if $a_1 + 2a_2 + \ldots + na_n = n$ and $a_k = 0$ for $k > n$. If $\alpha \vdash n$, we use the notation $\sigma_\alpha$ to denote a permutation on $[n]$ with cycle type $\alpha$, and write $\sigma_\alpha^+$ for the permutation on $[n+1]$ that has cycle type $(a_1 + 1, a_2, a_3, \ldots)$ and is equal to $\sigma_\alpha$ when restricted to $[n]$. Then
    \begin{equation*}
        (\ref{zf'1}) = \sum_{n \geq 0} \sum_{\alpha \vdash n}\mathbf{fix}(F[\sigma_\alpha^+]) \frac{1}{\prod i^{a_i}a_i!} x_1^{a_1} x_2^{a_2}... = \frac{\del}{\del x_1}\sum_{n \geq 0} \sum_{\alpha \vdash n+1}\mathbf{fix}(F[\sigma_\alpha]) \frac{1}{\prod i^{a_i}a_i!} x_1^{a_1} x_2^{a_2}...
    \end{equation*}
    \begin{equation*}
         = \frac{\del}{\del x_1}\sum_{n \geq 1} \sum_{\alpha \vdash n}\mathbf{fix}(F[\sigma_\alpha]) \frac{1}{\prod i^{a_i}a_i!} x_1^{a_1} x_2^{a_2}... = \frac{\del}{\del x_1}\sum_{n \geq 1} \frac{1}{n!} \sum_{\alpha \vdash n}\mathbf{fix}(F[\sigma_\alpha]) \frac{n!}{\prod i^{a_i}a_i!} x_1^{a_1} x_2^{a_2}...
    \end{equation*}
    \begin{equation*}
         = \frac{\del}{\del x_1}\sum_{n \geq 1} \frac{1}{n!} \sum_{\sigma \in S_n}\mathbf{fix}(F[\sigma]) x_1^{\sigma_1} x_2^{\sigma_2}...  = \frac{\del}{\del x_1}(Z_F - \text{const.})(x_1,x_2,...) = \frac{\del}{\del x_1}Z_F(x_1,x_2,...)
    \end{equation*}
    as the only non-zero parts of that derivative sum are those where $a_1 > 0$, and we can see the new $\sigma_\alpha$ as being the old $\sigma_\alpha^+$.
\end{proof}

\begin{theorem}
    The addition, product, and chain rules for derivatives still stand in the case of derivatives of species.
\end{theorem}
\begin{proof}
    The addition rule states that $(F+G)' = F'+G'$. This follows directly from the definition of addition of species.

    The product rule states that $(FG)' = F'G + FG'$. To see this note that in an $(FG)'$ structure $s = (f,g)$, the $*$ element must either be in $f$ or in $g$.

    The chain rule states that $(F \circ G)' = F'(G) \cdot G'$. This can be best explained with a picture:
    \begin{center}

\tikzset{every picture/.style={line width=0.75pt}} %set default line width to 0.75pt        


    \end{center}
\end{proof}

Let us now look at some concrete examples of derivatives of structures.

\begin{exa}
    The derivative of a cycle is a linear order, that is, $L = \cyc'$. This can be seen by removing the element $*$, or conversely by adding $*$ in between the maximum and the minimum. The following picture shows this in a clearer way.
    \begin{center}
        \begin{tikzcd}
                         & a \arrow[ld, bend right] &                          &   &                          & a \arrow[ld, bend right] \\
b \arrow[rd, bend right] &                          & * \arrow[lu, bend right] & = & b \arrow[rd, bend right] &                          \\
                         & c \arrow[ru, bend right] &                          &   &                          & c                       
\end{tikzcd}
    \end{center}
\end{exa}

\begin{exa}
    The derivative of a set is a set, that is, $E' = E$. This gives a combinatorial version of the classical identity $\frac{d}{dx}e^x = e^x$.
\end{exa}

\begin{exa}
    The derivative of a linear order is an ordered pair of linear orders, that is, $L' = L^2$. The following picture shows this in a clearer manner.
    \begin{center}

\tikzset{every picture/.style={line width=0.75pt}} %set default line width to 0.75pt        


    \end{center}
\end{exa}

\begin{exa}\label{dtf}
    The derivative of a tree is a (rooted) forest, that is, a set of rooted trees. This translates to $\tree' = E(\rooted)$. This can be better seen in the following picture.
    \begin{center}

\tikzset{every picture/.style={line width=0.75pt}} %set default line width to 0.75pt        

%
    \end{center}
\end{exa}

We can now move on to the last operation: pointing.

\begin{defi}
    Let $F$ be a species of structures. Then the species $F^\bullet$, called $F$-dot, is defined as follows: an $F^\bullet$-structure on a finite set $U$ is a pair $s = (f,u)$, where $f \in F[U]$ and $u \in U$. Transport along $\sigma$ is given by $F^\bullet[\sigma](f,u) = (F[\sigma](f),\sigma(u))$. The pair $(f,u)$ is called a pointed $F$-structure. Graphically, this can be represented in the following manner.
    \begin{center}

\tikzset{every picture/.style={line width=0.75pt}} %set default line width to 0.75pt        


    \end{center}
    The enumeration of pointed structures satisfies the equation $|F^\bullet[n]| = n|F[n]|$. The operation of pointing corresponds to the differential operator $x \frac{d}{dx}$, which can be seen more clearly in the following theorem.
\end{defi}

\begin{theorem}
    Let $F$ be a species of structures. Then
    \begin{itemize}
        \item[ ] $F^\bullet(x) = x \frac{d}{dx} F(x)$
        \item[ ] $\widetilde{F^\bullet}(x) = x \left( \frac{\del}{\del x_1} Z_F \right)(x,x^2,x^3,...)$
        \item[ ] $Z_{F^\bullet}(x_1,x_2,x_3,...) = x_1 \left( \frac{\del}{\del x_1} Z_F \right) (x_1,x_2,x_3,...)$
    \end{itemize}
\end{theorem}
\begin{proof}
    It is sufficient to prove the third equality. For that, let $\sigma \in S_n$ be a permutation. An $F^\bullet$-structure is fixed by $F^\bullet[\sigma]$ if and only if the underlying $F$-structure is fixed, as well as the pointed element. That means that, for each $F$-structure fixed by $F[\sigma]$, there are $\sigma_1$ fixed $F^\bullet$-structures, that is, $\mathbf{fix}(F^\bullet[\sigma]) = \sigma_1 \mathbf{fix}(F[\sigma])$. Therefore,
    \begin{align*}
        Z_{F^\bullet}(x_1,x_2, \ldots) &= \sum_{n \geq 0} \frac{1}{n!} \sum_{\sigma \in S_n} \mathbf{fix}(F^\bullet[\sigma]) x_1^{\sigma_1} x_2^{\sigma_2} \ldots = \sum_{n \geq 0} \frac{1}{n!} \sum_{\sigma \in S_n} \sigma_1 \mathbf{fix}(F[\sigma])x_1^{\sigma_1}x_2^{\sigma_2} \ldots \\
        &= x_1 \sum_{n \geq 0} \frac{1}{n!} \sum_{\sigma \in S_n} \mathbf{fix}(F[\sigma])\sigma_1x_1^{\sigma_1 - 1}x_2^{\sigma_2} \ldots = x_1 \left( \frac{\del}{\del x_1} Z_F \right)(x_1,x_2, \ldots).
    \end{align*}
\end{proof}

There are two ways to express pointing in terms of other operations. The first one, and the clearer one, is to see a pointed $F$-structure as a pairing of an $F'$-structure and a singleton, which becomes clear from the following picture.
\begin{center}

\tikzset{every picture/.style={line width=0.75pt}} %set default line width to 0.75pt        


\end{center}
This translates to the equation $F^\bullet = X \cdot F'$, which gives us a clear visualization of the theorem above. The second way to see pointing through other operations comes from the Cartesian product. Since pointing an $(F \times G)$-structure can be seen as pointing either structure and then pairing them together, we have the distributive property $(F \times G)^\bullet = F^\bullet \times G = F \times G^\bullet$. From that, since $F = F \times E$ for any species $F$, we have $F^\bullet = (F \times E)^\bullet = F \times E^\bullet = F \times (X \cdot E)$, and we can write pointing as a combination of the two products.

Let us now see a few examples of pointing in action.

\begin{exa}
    A rooted tree can be seen as a tree with a pointed element, that is, the root. From that, we have the combinatorial equality $\rooted = \tree^\bullet$.
\end{exa}

\begin{exa}\label{vapp}
    Pointing can be used to enumerate the number $\alpha_n$ of trees on $n$ vertices. Consider the species $\mathcal{V} = \tree^{\bullet\bullet}$ of vertebrates. This gives the equality $\nu_n = n^2 \alpha_n$.
    
    A $\mathcal{V}$-structure can also be seen as a tree with two (not necessarily distinct, but ordered) selected points. This creates the equality $\mathcal{V} = L_+(\rooted)$, as can be seen in the picture below.
    \begin{center}

\tikzset{every picture/.style={line width=0.75pt}} %set default line width to 0.75pt        


    \end{center}
    Since we only care about the cardinality, we can replace $L_+$ with the equipotent $\perm_+$, and since $\perm(\rooted) = \text{End}$, we have that, for $n > 0$, there are as many vertebrates as endofunctions; thus $\nu_n = n^n$. Coming back to trees, we have that, for $n > 0$, there are $n^{n-2}$ trees on $n$ vertices, and therefore $n^{n-1}$ rooted trees.
\end{exa}

\section{Multisort Species}

In this section, we will present the concept of multisort species, which will be useful in the next few chapters. 
%and show how the series and operations happen on them.

\begin{defi}
    A $k$-(multi)set is a $k$-tuple $U = (U_1, \ldots ,U_k)$ of sets. An element of $U_i$ is said to be of sort $i$.
    We say that $U$ is a finite $k$-set if $U_i$ is a finite set for each $i=1, \ldots, k$. 
\end{defi}

\begin{defi}
    A multifunction $f : (U_1, \ldots ,U_k) \to (V_1, \ldots ,V_k)$ between $k$-sets is a $k$-tuple of functions 
    $f = (f_1, \ldots ,f_k)$ such that $f_i: U_i \to V_i$ for each $i=1, \ldots, k$. 
    Composition of multifunctions is done componentwise.
    We say that $f$ is a bijective multifunction if each component function $f_i$ is bijective.
\end{defi}

\begin{defi}
    A $k$-sort species is a functor $F$ from the category $\B^k$ of finite $k$-sets and bijective multifunctions to the category $\B$ of finite sets and bijections. 
    In other words, $F$ is a rule that assigns 
    \begin{enumerate}
        \item to each finite $k$-set $U = (U_1, \ldots,U_k)$, a finite set $F[U_1, \ldots ,U_k]$;
        \item to each bijective multifunction $\sigma = (\sigma_1, \ldots ,\sigma_k) : (U_1, \ldots ,U_k) \to (V_1, \ldots ,V_k)$, a bijection $F[\sigma_1,\ldots,\sigma_k] : F[U_1, \ldots,U_k] \to F[V_1,\ldots,V_k]$ such that the functoriality conditions are satisfied, \emph{i.e.},  
        $$F[\tau \circ \sigma] = F[\tau] \circ F[\sigma] \quad \text{and} \quad {F[Id_U] = Id_{F[U]}}.$$
    \end{enumerate}
\end{defi}

\begin{defi}
    The $k$-sort species $X_i$ of singletons of sort $i$ is defined as follows:
    \begin{equation*}
        X_i[U] = \begin{cases}
            \{U\}\text{, if } |U_i| = 1\text{ and }|U_j| = 0\text{ for all }j \neq i;\\
            \emptyset\text{, otherwise}.
        \end{cases}
    \end{equation*}
    Sometimes it is more convenient to use $X,Y,Z,T$ instead of $X_1,X_2,X_3,X_4$.
\end{defi}

\begin{notn}
    If $F$ is a $k$-sort species, it is sometimes useful to denote it as $F(X_1,\ldots,X_k)$.
    This notation is compatible with the composition introduced later in this section.
\end{notn}

\begin{exa}
    Another example of $k$-sort species is the species of tri-colored simple graphs, that is, graphs whose vertices can be of three distinct colors. In this case, transport along a multifunction corresponds to relabelling vertices while preserving their color.
\end{exa}

Let us now see how the operations introduced in the previous section extend to $k$-sort species. For that, we must first define dissections and partitions.

\begin{defi}
    A $k$-set dissection of a $k$-set $U$ is a pair of $k$-sets $(V,W)$ such that ${V_i \cup W_i = U_i}$ and $V_i \cap W_i = \emptyset$ for each $i=1, \ldots, k$. We denote by $\text{Dis}[U]$ the set of dissections of $U$.
\end{defi}

\begin{defi}
    A $k$-set partition $\pi$ of a $k$-set $U=(U_1, \ldots, U_k)$ is a partition of the set $U_1 \sqcup \ldots \sqcup U_k$. Each class $C \in \pi$ can be seen as a $k$-set with $C_i = C \cap U_i$. We denote by Par$[U]$ the set of partitions of $U$.
\end{defi}

We are now ready to define operations on $k$-sort species.

\begin{defi}
    The addition of two $k$-sort species is defined by taking disjoint unions, 
    just like in the case of regular species.
\end{defi}

\begin{defi}
    The product of two $k$-sort species on a $k$-set $U$ is given by
    \begin{equation*}
        (FG)[U] = \sum_{(V,W) \in \text{Dis}[U]} (F[V] \times G[W]).
    \end{equation*}
\end{defi}

\begin{defi}
    The Cartesian product of two $k$-sort species on a $k$-set $U = (U_1,\ldots,U_k)$ is given by $(F \times G)[U_1,\ldots U_k] = F[U_1,\ldots,U_k] \times G[U_1,\ldots,U_k]$. %, with pairings respecting the sort of the elements. 
    The following picture will make this clearer.
    \begin{center}

\tikzset{every picture/.style={line width=0.75pt}} %set default line width to 0.75pt        


    \end{center}
\end{defi}

\begin{defi}
    Let $F(X_1, \ldots ,X_m)$ be an $m$-sort species, and $G_1,\ldots ,G_m$ be $m$ $k$-sort species. Then the composition $F(G_1, \ldots ,G_m)$ is the $k$-sort species defined by
    \begin{equation*}
        F(G_1, \ldots ,G_m)[U] = \underset{\chi : \pi \to [m]}{\sum_{\pi \in \text{Par}[U]}} (F[\chi^{-1}] \times \underset{C \in \chi^{-1}(j)}{\prod_{j \in [m]}}G_j[C])
    \end{equation*}
    where, for each function $\chi : \pi \to [m]$, $\chi^{-1}$ is the $m$-set $(\chi^{-1}(1), \ldots ,\chi^{-1}(m))$. A description in words is that an $F(G_1,\ldots ,G_m)$-structure is an $F$-structure in which each element of sort $X_j$ has been turned into a $G_j$-structure.
\end{defi}

\begin{exa}
    Consider two sorts of elements $X$ and $Y$. An $(X+Y)$-structure is a singleton of either sort. If $F$ is a $1$-sort species, an $F(X+Y)$-structure is an $F$-structure whose underlying set 
    $U=U_1 \sqcup U_2$ is composed of elements of two sorts, namely, $X$ and $Y$. 
    Indeed, by the definition of composition, we have 
    \begin{align*}
        F(X + Y)[U] = \underset{\chi : \pi \to [1]}{\sum_{\pi \in \text{Par}[U]}} (F[\chi^{-1}] \times \underset{C \in \chi^{-1}(j)}{\prod_{j \in [1]}} (X+Y)[C]) = F[U_1 \sqcup U_2] \times \prod_{z \in U_1 \sqcup U_2} (X+Y)[\{z\}],
    \end{align*}
    as the only partition where this product doesn't vanish is the maximal partition.
\end{exa}

\begin{exa}
    The species $\cyc(X+Y)$ of cycles with beads of two colors can be described by the equation $\cyc(X+Y) = \cyc(X) + \cyc(YL(X))$. This reflects that if the cycle has at least one element of color $Y$, then it can be seen as a cycle of chains where the only element of sort $Y$ is the first one. The following picture will make this equation clearer.
    \begin{center}

\tikzset{every picture/.style={line width=0.75pt}} %set default line width to 0.75pt        


    \end{center}
\end{exa}

\begin{defi}
    For a $k$-sort species, one can define partial derivatives in the following way:
    \begin{equation*}
        \left(\frac{\del}{\del X_j} F\right)[U_1, \ldots ,U_k] = F[U_1, \ldots ,U_j \sqcup \{*_j\}, \ldots ,U_k]
    \end{equation*}
    The usual rules of calculus still apply. For example, in the case of $2$-sort species,
    \begin{equation*}
        \frac{\del}{\del X} F(G,H) = \frac{\del F}{\del X}(G,H) \cdot \frac{\del G}{\del X} + \frac{\del F}{\del Y}(G,H) \cdot \frac{\del H}{\del X}.
    \end{equation*}
\end{defi}

\begin{defi}
    The partial pointing of a $k$-sort species is defined as $F^{\bullet_i} = X_i \frac{\del}{\del X_i} F$.
\end{defi}

For each operation, the transport of structures is defined in the natural manner, as for regular ($1$-sort) species.

We will focus here on 2-sort species, that is, species with two different sorts of elements, but the series shown below can be generalized in the natural manner to $k$-sort species. 
%One example of such species are trees where leaves are colored differently from the rest of the vertices. 
We simplify notation by writing $F[[n],[k]]$ as simply $F[n,k]$.
\begin{defi}
    Let $F(X,Y)$ be a 2-sort species. Then the generating series $F(x,y)$, the type generating series $\widetilde{F}(x,y)$, and the \emph{Zyklenzeiger} $Z_F(x_1,x_2, \ldots ;y_1,y_2, \ldots )$ are defined as
    \begin{itemize}
        \item[ ] $\displaystyle F(x,y) = \sum_{n,k \geq 0} |F[n,k]| \frac{x^n}{n!} \frac{y^k}{k!}$
        \item[ ] $\displaystyle \widetilde{F}(x,y) = \sum_{n,k \geq 0} |F[n,k]/\til| x^n y^k$
        \item[ ] $\displaystyle Z_F(x_1,x_2, \ldots ;y_1,y_2, \ldots) = \sum_{n,k \geq 0} \frac{1}{n!k!} \sum_{\sigma \in S_n, \tau \in S_k} \mathbf{fix}(F[\sigma,\tau])x_1^{\sigma_1}x_2^{\sigma_2} \cdots y_1^{\tau_1}y_2^{\tau_2} \ldots$
    \end{itemize}
    where $\til$ is the equivalence relation defined by $s \til t \iff (\exists \sigma) F[\sigma](s) = t$. More so, the following relations stand, and the passage to series is compatible with operations in the same sense as before.
    \begin{itemize}
        \item[ ] $F(x,y) = Z_F(x,0, \ldots ;y,0, \ldots)$
        \item[ ] $\widetilde{F}(x,y) = Z_F(x,x^2, \ldots;y,y^2, \ldots)$
    \end{itemize}
\end{defi}

%Let us see an example of 2-sort species in practice.

\section{Virtual Species}

In this section, we will provide a way to cover a flaw in the space of species: the absence of properly defined subtraction and division operations. 

Recall that a semiring is an algebraic structure similar to a ring, but without the requirement that each element must have an additive inverse. For example, the natural numbers with the standard operations of addition and multiplication form a semiring.
Another example is the set of species (or more precisely, the set of equivalence classes of naturally isomorphic species) 
with the operations of addition and multiplication. 

\begin{defi}
    Consider the semiring $(\text{Spe},+,\cdot)$ of species of structures. By analogy to the construction of $\Z$ from $\N$, we define a virtual species as being an element of the quotient set
    \begin{equation*}
        \text{Virt} = (\text{Spe} \times \text{Spe})/\til
    \end{equation*}
    where the equivalence relation $\til$ is defined by
    \begin{equation*}
        (F,G) \til (H,K) \iff F + K = G + H.  
    \end{equation*}
    We denote by $F-G$ the class of $(F,G)$ in regards to $\til$.
\end{defi}

The fact that $\til$ is an equivalence relation relies on the cancellation law for addition of species.

\begin{prop}
    The set \emph{Virt} of virtual species constitutes a commutative ring $($\emph{Virt}$,+,\cdot)$, with the operations defined by
    \begin{itemize}
        \item[ ] $(F-G) + (H-K) = (F+H) - (G + K)$
        \item[ ] $(F-G) \cdot (H-K) = (FH+GK)-(FK+GH)$
    \end{itemize}
    and neutral elements $0 = (0-0)$ and $1 = (1-0)$. The additive inverse of $(F-G)$ is $(G-F)$.
Moreover, there is an obvious injective homomorphism of semirings $\emph{Spe} \to \emph{Virt}$, $F \mapsto F - 0$.
\end{prop}

\begin{defi}
    Let $\Phi = F-G$ be a virtual species. Then the series associated to $\Phi$ are defined by
    \begin{itemize}
        \item[ ] $\Phi(x) = F(x) - G(x)$
        \item[ ] $\widetilde{\Phi}(x) = \widetilde{F}(x) - \widetilde{G}(x)$
        \item[ ] $Z_\Phi(x_1,x_2,\ldots) = Z_F(x_1,x_2, \ldots) - Z_G(x_1,x_2, \ldots)$
    \end{itemize}
\end{defi}

It is easy to verify that the above definitions don't depend on the choice of representative species $F$ and $G$.
Moreover, the properties for product of series still stand, that is, if $\Phi,\Psi$ are virtual species, then
\begin{itemize}
    \item[ ] $(\Phi\Psi)(x) = \Phi(x)\Psi(x)$
    \item[ ] $\widetilde{(\Phi\Psi
)}(x) = \widetilde{\Phi}(x)\widetilde{\Psi}(x)$
    \item[ ] $Z_{\Phi\Psi}(x_1,x_2, \ldots ) = Z_\Phi(x_1,x_2, \ldots)Z_\Psi(x_1,x_2, \ldots)$
\end{itemize}

With subtraction properly defined, our next task is to define the multiplicative inverse of a species of structures, which naturally gives us a sense of division. 
%For that, we must define what a summable virtual species is.
\begin{defi}
    Let $F$ and $G$ be species of structures. We call $G$ a subspecies of $F$, and write $G \subseteq F$, if
    \begin{enumerate}
        \item for any finite set $U$, $G[U] \subseteq F[U]$;
        \item for any bijection $\sigma : U \to V$ between finite sets, $G[\sigma] = F[\sigma]|_{G[U]}$.
    \end{enumerate}
\end{defi}

For example, the species $\tree$ of trees and the species $\mathcal{G}^c$ of simple connected graphs are both subspecies 
of the species $\mathcal{G}$ of simple graphs.

\begin{defi}
    Two species $F$ and $G$ are said to be unrelated if their only subspecies in common is the empty species. A virtual species $\Phi = F-G$ is said to be in reduced form if $F$ and $G$ are unrelated.
\end{defi}

Every species can be written in reduced form. We denote the reduced form by ${\Phi = \Phi^+ - \Phi^-}$, and call 
$\Phi^+$ and $\Phi^-$ the positive and negative parts of $\Phi$, respectively. 
The proof of this fact requires the study of molecular decompositions, and can be found in \cite[Section 2.6]{bergeron2}.

\begin{defi}
    A family $\{ \Phi_i \}$ of virtual species is summable if each of the two families of species $\{ \Phi_i^+ \}, \{\Phi_i^-\}$ is summable. In this case,
    \begin{equation*}
        \sum_{i} \Phi_i = \sum_i \Phi_i^+ - \sum_i \Phi_i^-.
    \end{equation*}
\end{defi}

\begin{theorem}
    Let $G$ be a species of structures such that $G(0) = 1$. Then the multiplicative inverse of $G$ is given by
    \begin{equation*}
        G^{-1} = \frac{1}{G} = \sum_{k \geq 0} (-1)^k(G_+)^k
    \end{equation*}
    where $-1 = (0-1)$ is the additive inverse of the species $1$ and $G_+ = G-1$.
\end{theorem}
\begin{proof}
    Since $G = 1 + G_+$, we have $G^{-1} = (1 + G_+)^{-1} = \sum_{k \geq 0} (-1)^k(G_+)^k$. 
    For this to make sense, we must prove that the family of virtual species $\{(-1)^k(G_+)^k\}$ is summable.
    This follows easily from the fact that the family of species $\{(G_+)^k\}$ is summable.
\end{proof}

\chapter{Tree-like Structures}
\label{ch5}
In this chapter, we will show techniques to study structures which, at first glance, look like tree structures with a particular extra property. We begin by presenting enriched trees, which will prove useful later when solving a specific kind of differential equation. We then show an interesting way to relate trees and rooted trees via the Dissymmetry Theorem, and study some particular cases in which it is useful for enumeration.

\section{Enriched Trees}

Those familiar with analysis may recall the Lagrange inversion formula, which states that, if $f(x) \in \C[[x]]$ with $f(0) = 0 \neq f'(0)$, then
\begin{equation*}
    f^{\langle -1 \rangle}(x) = \sum_{n=1}^\infty \frac{d^{n-1}}{dt^{n-1}} \left( \frac{t}{f(t)} \right)^n \Biggr\rvert_{t=0} \frac{x^n}{n!},
\end{equation*}
where $f^{\langle -1 \rangle}(x)$ denotes the compositional inverse of $f(x)$, that is, the unique series satisfying 
$f(f^{\langle -1 \rangle}(x))=x$ and $f^{\langle -1 \rangle}(f(x))=x$.

For notational convenience, we denote by $A(x):=f^{\langle -1 \rangle}(x)$ the compositional inverse of $f(x)$ 
and set $\displaystyle R(x) = \frac{x}{f(x)}$. Then the series $A(x)$ is determined by the functional equation $A(x) = xR(A(x))$, 
and also by
\begin{equation}\label{lag}
    A(x) = \sum_{n = 1}^\infty a_n \frac{x^n}{n!}\text{, with }a_n = \frac{d^{n-1}}{dt^{n-1}}[R(t)]^n \Bigr\rvert_{t=0}.
\end{equation}

We have already seen one particular case of species satisfying an equation similar to $A(x) = xR(A(x))$, namely, the species $\rooted$ of rooted trees. Indeed, by Example \ref{axea}, we have ${\rooted = X \cdot E(\rooted)}$.
As we will see, combinatorial equations of the form $F=X \cdot R(F)$ (with $R$ any given species), 
always have a unique solution $F$.

\begin{defi}
    Let $R$ be a species of structures. An $R$-enriched rooted tree on a finite set $U$ is given by
    \begin{enumerate}
        \item a rooted tree on $U$;
        \item an $R$-structure on the fiber of each vertex $u$ of this rooted tree,
    \end{enumerate}
    where the fiber of a vertex $u$ is defined as being the set of vertices connected to $u$, excluding the vertex closest to the root. The following picture will make this definition clearer.
    \begin{center}
    
\tikzset{every picture/.style={line width=0.75pt}} %set default line width to 0.75pt        


    \end{center}
\end{defi}

The leaves of a rooted tree are exactly those vertices with empty fiber. 
Since the empty fibers of an $R$-enriched rooted tree are also provided with an $R$-structure,
it is reasonable to impose the hypothesis $R(0) \neq 0$ (otherwise there cannot exist any    
$R$-enriched rooted trees).

\begin{theorem}
    Let $R$ be a species of structures with $R(0) \neq 0$. The species $\rooted_R$ of $R$-enriched rooted trees is uniquely determined (up to isomorphism) by the combinatorial equation ${\rooted_R = X \cdot R(\rooted_R)}$.
\end{theorem}
\begin{proof}
    The fact that $\rooted_R$ satisfies the equation is true by construction. The uniqueness is a consequence of the Implicit Species Theorem (see \cite[Th. 2.1]{labelle2}).
\end{proof}

We have a way to enumerate $R$-enriched rooted trees. This is done, naturally, by applying formula \ref{lag}.
However, it is interesting to note that using the algebra of species of structures, it is possible to 
enumerate $R$-enriched rooted trees without using Lagrange inversion, and thus obtain a combinatorial 
proof of formula \ref{lag} (see \cite[Section~3.1]{bergeron2}).  

\begin{exa}
    When $R = E$, we have $\rooted_E = \rooted$, the species of rooted trees. We can then use Lagrange inversion to enumerate 
    rooted trees:
    \begin{equation*}
        |\rooted[n]| = \frac{d^{n-1}}{dt^{n-1}} e^{nt} \Bigr\rvert_{t=0} = n^{n-1} e^{nt} \Bigr\rvert_{t=0} = n^{n-1},
    \end{equation*}
    which coincides with the result found in Example \ref{vapp}.
\end{exa}
\begin{exa}
    When $R = L$ (the species of linear orders), $\rooted_L$ is the species of planar rooted trees. 
    From $\rooted_L=X \cdot L(\rooted_L)$, we obtain 
    $$\rooted_L(x)=\frac{x}{1 - \rooted_L(x)},$$
    and solving the corresponding quadratic equation, we get
    $$\rooted_L(x)=\frac{1 - \sqrt{1-4x}}{2}.$$
    Comparing coefficients (or using Lagrange inversion), we deduce
    $$|\rooted_L[n]|=\frac{(2n-2)!}{(n-1)!}.$$
\end{exa}

Until now, we have been looking at a generalization of rooted trees. One might wonder how to generalize, then, the concept of tree. This is what comes next.

\begin{defi}
    An $R$-enriched tree on a finite set $U$ is given by
    \begin{itemize}
        \item a tree on $U$;
        \item an $R$-structure on the set of adjacent vertices of each vertex of the tree.
    \end{itemize}
    The following picture will make this clearer.
    \begin{center}

\tikzset{every picture/.style={line width=0.75pt}} %set default line width to 0.75pt        


    \end{center}
\end{defi}

\begin{notn}
    Denote by $\tree_R$ the species of $R$-enriched trees.
\end{notn}

The species $\rooted$ of rooted trees is related to the species $\tree$ of trees by the combinatorial equations $\tree' = E(\rooted)$ and $\tree^{\bullet} = \rooted$. Unfortunately, these equations do not hold in the case of $R$-enriched trees, however, there is a similar pair of equations.

\begin{theorem}
    Let $R$ be a species of structures. Then $\tree_R' = R(\rooted_{R'})$ and $\tree_R^\bullet = X \cdot R(\rooted_{R'})$, where $\rooted_{R'} = X \cdot R'(\rooted_{R'})$ is the species of $R'$-enriched rooted trees. 
\end{theorem}
\begin{proof}
    Let $U$ be a finite set, and take an $\tree_R'$-structure on $U$, that is, an $\tree_R$-structure on $U \sqcup \{*\}$. The following picture shows how one can pair it with an $R$-assembly of $\rooted_{R'}$-structures.
    \begin{center}

\tikzset{every picture/.style={line width=0.75pt}} %set default line width to 0.75pt        


    \end{center}
    We then have the equality $\tree_R' = R(\rooted_{R'})$. The equation $\tree_R^\bullet = X \cdot R(\rooted_{R'})$ is an immediate consequence.
\end{proof}

\section{The Dissymmetry Theorem}

In this section, we will study another method to relate trees and rooted trees. For that, we must first define the concept of centre.

\begin{defi}
    The centre of a tree is its subgraph generated by the vertices of minimum eccentricity, where the eccentricity of a vertex is the maximum distance from this vertex to any other vertex on the tree.
\end{defi}

It is easy to see that the centre of any tree is constituted by one vertex or two connected vertices. This creates the concept of canonical pointing.

\begin{defi}
    A tree is said to be canonically pointed if it is pointed at its centre, that is, if the centre is a vertex, the tree is pointed at this vertex, and if the centre is a pair of connected vertices, the tree is pointed at the edge that connects these two vertices.
\end{defi}

With these definitions, we can then prove the dissymmetry theorem for trees.

\begin{theorem}[Dissymmetry Theorem for Trees]\label{distree}
    The species $\tree$ of trees and $\rooted$ of rooted trees are related by ${\rooted + E_2(\rooted) = \tree + \rooted^2}$.
\end{theorem}
\begin{proof}
    The left-hand side of the equation enumerates the trees that are either pointed at a vertex $(\rooted = \tree^\bullet)$ or at an edge, each such tree being paired with two rooted trees whose roots are the vertices originally connected to the distinguished edge, from where we get $E_2(\rooted)$.

    On the right-hand side, the term $\tree$ identifies those trees that are canonically pointed. All that is left to do is prove that there is an isomorphism between trees pointed elsewhere and ordered pairs of rooted trees.
    \begin{enumerate}
        \item If the tree is pointed at a vertex $u$ different from the centre, let $v$ be the vertex adjacent to $u$ closest to the centre (in case the center is an edge and $u$ is one of its vertices, take $v$ to be the other vertex of the center). Cut the edge that connects $u$ to $v$, and pair the two remaining rooted trees putting the one with root $u$ on the left. See the first picture in Figure \ref{fig:pairings}.
        \item If the tree is pointed at an edge distinct from the centre, let $u$ be the vertex of this edge that is closest to the centre. Cut this edge, and pair the two remaining rooted trees putting the one with $u$ on the left. See the second picture in Figure \ref{fig:pairings}.
    \end{enumerate}
    
    Conversely, given an $\rooted^2$-structure, join with an edge the roots of the two trees and find the centre of the newly formed tree. If the centre comes from the right-hand side of the original pair, or if the centre is precisely the newly added edge, then we are in the first case and the root of the left-hand side tree is distinguished. Otherwise, if the centre comes from the left-hand side tree, we are in the second case and the added edge is distinguished.

\begin{figure}
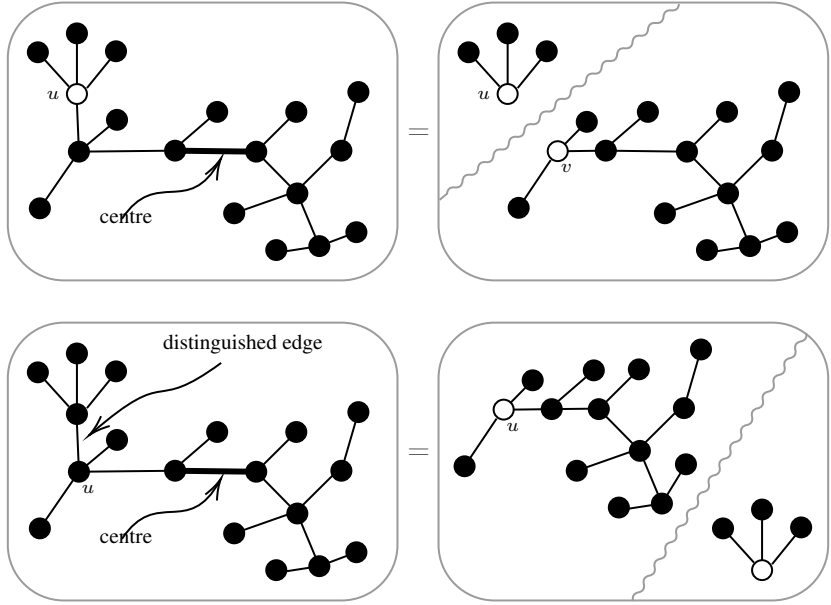

    \centering
    \tikzset{every picture/.style={line width=0.75pt}} %set default line width to 0.75pt        


    \caption{The two pairings between trees and ordered pairs of rooted trees}
    \label{fig:pairings}
\end{figure}
    
\end{proof}

To generalize this theorem to $R$-enriched trees, we must first introduce some notation.

\begin{notn}
    Denote by $\tree_R^{-}$ the species of $R$-enriched trees with a distinguished edge, and by $\tree_R^{\rightarrow}$ the species of $R$-enriched trees with a distinguished oriented edge.
\end{notn}

It is easy to see that $\tree_R^- = E_2(\rooted_{R'})$ and $\tree_R^{\rightarrow} = \rooted_{R'}^2$. With that in mind, we can prove the generalized version of the dissymmetry theorem.

\begin{theorem}[Dissymmetry Theorem for $R$-Enriched Trees]
    Let $R$ be a species such that $R'(0) \neq 0$. Then ${\tree_R = XR(\rooted_{R'})+E_2(\rooted_{R'})-\rooted_{R'}^2}$.
\end{theorem}
\begin{proof}
    Note that this equation can be rewriten as $\tree_R^\bullet + \tree_R^- = \tree_R + \tree_R^\rightarrow$. The proof is similar 
    to that of Theorem \ref{distree}.
\end{proof}

Let us now see some examples of enriched trees and their enumeration.

\begin{exa}
    Denote by $\hit$ the species of homeomorphically irreducible trees, that is, trees without vertices of degree 2. These are 
    $R$-enriched trees, where $R=E-E_2$. Observe that $R' = (E-E_2)' = E-X$, and denote by $\hirt = \rooted_{R'}$ the species of homeomorphically irreducible rooted trees, that is, rooted trees without fibers of cardinality 1. Note that $\hirt \neq \hit^\bullet$. We have that $\hirt = XE(\hirt) - X\hirt$ and $\hit = X(E-E_2)(\hirt) + E_2(\hirt) - \hirt^2$.

    In any rooted tree, the chains of consecutive vertices having fibers of cardinality 1 can be grouped together and be seen as an order, from where $\rooted = \hirt \left( \frac{X}{1-X}\right)$, as $\frac{X}{1-X} = L - L_0$.
    \begin{center}

\tikzset{every picture/.style={line width=0.75pt}} %set default line width to 0.75pt        


    \end{center}
    Inverting for substitution, we have $\hirt = \rooted\left(\frac{X}{1+X} \right)$. From the Dissymmetry Theorem,\newline
    $\tree = (X + E_2 - X^2)\circ\rooted$, and $X(E-E_2)(\hirt) = \hirt + X\hirt - XE_2(\hirt)$, we get
    \begin{equation*}
        \hit = X(E-E_2)(\hirt)+E_2(\hirt)-\hirt^2 = \hirt + X \hirt - XE_2(\hirt) + E_2(\hirt)-\hirt^2
    \end{equation*}
    \begin{equation*}
        = (X + E_2 - X^2)\circ \hirt + X\hirt - XE_2(\hirt) = \tree \left( \frac{X}{1+X} \right) + X\rooted\left( \frac{X}{1+X} \right) - XE_2\left( \rooted\left(\frac{X}{1+X} \right)\right)
    \end{equation*}
    After computation, one can enumerate unlabelled homeomorphically irreducible trees.
    \begin{equation*}
        \widetilde{\hit}(x) = x + x^2 + x^4 + x^5 + 2x^6 + 2x^7 + 4x^8 + 5x^9 + 10x^{10} + 14x^{11} + 26x^{12} + 42x^{13} + ...
    \end{equation*}
\end{exa}

\begin{exa}
    An oriented tree is a tree where each edge has been given an orientation. Denote by $\orit$ the species of oriented trees and by $\orbit = \orit^\bullet$ the species of rooted oriented trees. The labelled enumeration of these structures is very simple, as each of the $n-1$ edges can be oriented in two different ways, and this yields $|\orit[n]|=2(2n)^{n-2}$ and $|\orbit[n]|=(2n)^{n-1}$.

    Now, considering a fixed $\orbit$-structure. It can be decomposed into a singleton (the root), and a pair of sets of $\orbit$-structures (the set of those that connect inwards and the set of those that connect outwards). This yields the equation $\orbit = XE^2(\orbit)$, from where we have ${\widetilde{\orbit}(x) = x \exp \left( 2 \sum_{k \geq 1} \frac{\widetilde{\orbit}(x^k)}{k} \right)}$. Through recursive computation, one then finds
    \begin{equation*}
        \widetilde{\orbit}(x) = x + 2x^2 + 7x^3 + 26x^4 + 107x^5 + 458x^6 + 2058x^7 + 9498x^8 + 44947x^9 + ...
    \end{equation*}
\end{exa}

In the case of oriented trees, the Dissymmetry Theorem takes a beautifully simple form.

\begin{theorem}[Dissymmetry Theorem for Oriented Trees]
    The species $\orit$ of oriented trees and $\orbit$ of rooted oriented trees are related by the equation ${\orbit = \orit + \orbit^2}$.
\end{theorem}
\begin{proof}
    Firstly, note that $\orbit^2$ is identified as the species of oriented trees where one oriented edge is distinguished. Define now the centre of an oriented tree as being the vertex which is either the centre of the equivalent non-oriented tree or, if the centre of that tree is an edge, the origin of that oriented edge. Consider now a rooted oriented tree. If the root is the centre of the oriented tree, then pair it canonically to an oriented tree. Otherwise, distinguish the oriented edge pointing towards the centre whose origin is the root, and associate it with an $\orbit^2$-structure.
\end{proof}

\begin{corollary}
    The enumeration of unlabelled oriented trees comes directly from that of unlabelled rooted oriented trees via the equation ${\widetilde{\orit}(x) = \widetilde{\orbit}(x)-\widetilde{\orbit}^2(x)}$, and therefore
    \begin{equation*}
        \widetilde{\orit}(x) = x + x^2 + 3x^3 + 8x^4 + 27x^5 + 91x^6 + 350x^7 + 1376x^8 + 5743x^9 + ...
    \end{equation*}
\end{corollary}

\chapter{Differential Equations}\label{ch6}

In this chapter, we will present a manner of studying and solving differential equations in the context of species of structures. Firstly, we consider the case of virtual species. We then study linearly ordered sets – and consequently linear species – as a manner to overcome an issue that shows up regarding the solutions to differential equations in the virtual sense. Finally, we provide a method to describe the solutions to a particular type of differential equation on linear species.

\section{Differential Equations on Virtual Species}

Let us begin by investigating the case of virtual species. We start by observing that one of the properties that we take for granted in usual calculus is false. 

\begin{theorem}
    The differential equation $\Phi' = 0$ has infinitely many non-constant solutions in the context of virtual species.
\end{theorem}
\begin{proof}
    For $n \geq 2$, we take $\Phi_n = n\cyc_n - X^n$. Since $\cyc_n' = L_{n-1}$, we have $$\Phi_n' = n\cyc_n' - (X^n)' = nL_{n-1} - nL_{n-1} = 0.$$
\end{proof}

%We can still find solutions though, as can be seen ahead.
Next we show that every virtual species admits many antiderivatives.

\begin{theorem}
    Let $\Psi$ be a virtual species. Then the equation $\Phi' = \Psi$ admits the solution 
    $\Phi = \Omega + \int \Psi$, where 
    $$ \int \Psi = E_1 \Psi - E_2 \Psi' + E_3 \Psi'' - \ldots + (-1)^{n+1}E_n\Psi^{(n-1)} + \ldots$$
    and $\Omega$ is any solution to $\Omega' = 0$.
\end{theorem}
\begin{proof}
    It suffices to prove that $(\int \Psi)'=\Psi$. Recall that $E_n' = E_{n-1}$ and that the product rule for derivatives holds 
    in the context of virtual species. Hence,
    \begin{align*}
        \left( \int \Psi \right)' &= E_1'\Psi + E_1\Psi' - E_2'\Psi' - E_2 \Psi'' + E_3'\Psi'' + E_3\Psi''' + \ldots \\
         &= \Psi + E_1\Psi' - E_1\Psi' - E_2\Psi'' + E_2\Psi'' + E_3\Psi''' - \ldots = \Psi
    \end{align*}
\end{proof}

This shows that, in the context of virtual species, differential equations typically admit solutions, however, the solutions are almost never uniquely determined by initial conditions. Fortunately, this nuisance can be removed by equipping the underlying set 
of our combinatorial structures with a linear order. This takes us to the realm of linear species.

\section{Linear Species}

In this section, we will present linear species and some operations on them. 
%For completeness, we begin by recalling the definition of a linearly ordered sets.
\begin{defi}
    A linearly ordered set is a pair $\ell = (U,\leq)$, where $U$ is a finite set and $\leq$ is an $L$-structure on $U$. We call $U$ the underlying set of $\ell$ and $\leq$ the order of $\ell$.
\end{defi}

\begin{notn}
    We write $u \in \ell$ instead of $u \in U$ when $U$ is the underlying set of $\ell$.
\end{notn}

\begin{notn}
    Let $u,v \in \ell$. We write $u \leq v$ if $u$ is smaller than or equal to $v$ with respect to the order of $\ell$, and $u < v$ when $u \leq v$ and $u \neq v$.
\end{notn}

\begin{notn}
    We denote by $\min \ell$ the smallest element with respect to the order of $\ell$.
\end{notn}

Consider a linearly ordered set $\ell = (U, \leq)$. For each $V \subseteq U$, we can define a linearly ordered set 
$\ell_V = (V,\leq_V)$, where $\leq_V$ is the restriction of $\leq$ to $V$. 
This gives us a way to decompose linearly ordered sets.

\begin{defi}
    Let $\ell = (U,\leq)$ be a linearly ordered set and $U_1 \sqcup \ldots \sqcup U_k = U$ a partition of $U$. 
    Define $\ell_i = (U_i,\leq_i)$ to be the the linearly ordered set whose order $\leq_i$ is obtained by restricting $\leq$ to $U_i$. 
    We then say that $\{\ell_1, \ldots, \ell_k\}$ is a partition of $\ell$, and write  
    $\ell_1 \sqcup \ldots \sqcup \ell_k = \ell$. We denote by Par$[\ell]$ the set of all partitions of $\ell$.
\end{defi}

Every partition $\pi = \{p_1, \ldots ,p_k\}$ of a linearly ordered set $\ell = (U,\leq)$ can be turned into a 
linearly ordered set $\ell_\pi = (\pi,\leq_\pi)$, where $p_i \leq_\pi p_j \iff \min p_i \leq \min p_j$.

\begin{notn}
    Denote by $\emptyset$ the empty linearly ordered set, that is, the pair $(\emptyset,\leq_\emptyset)$, where $\leq_\emptyset$ is the unique element in $L[0]$. In the same manner, denote by 1 the linearly ordered set $1 = ([1],\leq_1)$, where $\leq_1$ is the unique element of $L[1]$. Also, for $n \geq 1$, denote by $[n]$ the linearly ordered set $([n],\leq)$, where $\leq$ is the order induced from the natural order on $\Z$.
\end{notn}

\begin{defi}
    Let $\ell_1 = (U_1, \leq_1), \ell_2 = (U_2, \leq_2)$ be two linearly ordered sets. Then the ordinal sum $\ell_1 \oplus \ell_2$ is given by $\ell = (U,\leq)$, where $U = U_1 \sqcup U_2$, and the new linear order is defined by
    \begin{equation*}
        u \leq v \iff \begin{cases}
            u \leq_1 v\text{, if }u,v \in \ell_1;\\
            u \in \ell_1, v \in \ell_2;\\
            u \leq_2 v\text{, if }u,v \in \ell_2.
        \end{cases}
    \end{equation*}
    In particular, $1 \oplus \ell$ is the linearly ordered set obtained by adding a new minimal element to $\ell$.
\end{defi}

\begin{defi}
    Let $\ell_1 = (U_1,\leq_1), \ell_2 = (U_2,\leq_2)$ be two linearly ordered sets. A function $f : U_1 \to U_2$ is said to be increasing if, for $u,v \in \ell_1$, $u \leq_1 v$ implies $f(u) \leq_2 f(v)$.
\end{defi}

Note that for any linearly ordered set 
$\ell$ of cardinality $n$, there is a unique increasing bijection from $[n]$ to $\ell$. 

We now have the necessary tools to define linear species and the operations on them.

\begin{defi}
    A linear species is a functor $F$ from the category $\mathbb{L}$ of linearly ordered sets and increasing bijections to the category $\mathbb{B}$ of finite sets and bijections. We refer to linear species as $\mathbb{L}$-species, and to regular species 
    as $\mathbb{B}$-species.
\end{defi}

Note that any $\mathbb{B}$-species $F$ gives rise to an $\mathbb{L}$-species, also denoted by $F$, by setting 
$F[\ell]=F[U]$ for any linearly ordered set $\ell=(U, \leq)$; the transport of structures being obtained by restricting 
to increasing bijections. 

The notions of isomorphism of structures and isomorphism type for $\mathbb{L}$-species are defined in the same manner as for $\mathbb{B}$-species. However, since the only increasing bijection from a linearly ordered set to itself is the identity bijection, any structure of an $\mathbb{L}$-species has a unique automorphism, and no two distinct structures are isomorphic. 
Consequently, the concepts of isomorphism type generating series and \emph{Zyklenzeiger} are irrelevant in the context of 
$\mathbb{L}$-species: the only relevant series is the (exponential) generating series.

Two $\mathbb{L}$-species are said to be combinatorially equal if they are naturally isomorphic. 
In contrast to the theory of $\mathbb{B}$-species, two $\mathbb{L}$-species are cominatorially equal if and only if 
they are equipotent. In other words, if $F$ and $G$ are two linear species, $F = G$ if and only if $F(x) = G(x)$. In particular, the species $L$ and $\perm$ are (combinatorially) equal in the context of linear species.

Let us now define operations on linear species.

\begin{defi}
    Let $F$ and $G$ be linear species, and $\ell = (U,\leq)$ a linearly ordered set. 
    Then we define the following operations (in each case, the transport of structures being defined in 
    the obvious manner):
    \begin{itemize}
        \item Sum: $$(F + G)[\ell] = F[\ell] \sqcup G[\ell]$$
        \item Cartesian product: $$(F \times G)[\ell] = F[\ell] \times G[\ell]$$
        \item Product: $$ (F \cdot G)[\ell] = \sum_{\ell_1 \sqcup \ell_2 = \ell} F[\ell_1] \times G[\ell_2] $$
        \item Composition $(G(0) = 0)$: $$ (F \circ G)[\ell] = \sum_{\pi \in \text{Par}[\ell]} \left[ F[\ell_\pi] \times \prod_{p \in \pi} G[\ell|_p] \right] $$
        \item Derivative: $$ \frac{d}{dX}F(X)[\ell] = F'[\ell] = F[1 \oplus \ell] $$
        \item Integral: $$ \left( \int_{0}^X F(T) \, \mathrm{d}T \right)[\ell] = \left(\int F\right)[\ell] = \begin{cases}
            \emptyset\text{, if }\ell=\emptyset;\\
            F[\ell \setminus \{\min\ell\}]\text{, otherwise}
        \end{cases} $$
        \item Ordinal product: $$ (F \otimes G)[\ell] = \sum_{\ell_1 \oplus \ell_2 = \ell} F[\ell_1] \times G[\ell_2] $$
        \item Convolution: $$ (F * G)[\ell] = (F \otimes X \otimes G)[\ell] $$
    \end{itemize}
\end{defi}

%These definitions show that, in the sense of linear species, the integral is properly defined, and it extends naturally to the power series, as the generating series of the integral is the integral of the generating series.
It is not difficult to see that passage to generating series preserves the operations on $\mathbb{L}$-species.
Moreover, it can be verified that 
$$\int F'=F_+=F - F_0 \quad \text{and} \quad \left(\int F \right)'=F.$$

Let us now see some examples of linear species.
\begin{exa}
    Recall that, for a species $F$, the (virtual) species $F^c$ of connected $F$-structures is defined by the equation $F = E(F^c)$. In the context of linear species, we can give a simple definition of a connected list. An $L^c$-structure, that is, a connected list, is a non-empty list that begins with the minimum element of its underlying linearly ordered set. Naturally, $L = E(L^c)$, as any list can be cut into blocks that start with left-to-right minima. For example, the list $(3\ 4\ 1\ 2\ 5)$ can be paired with the set of lists $\{(3\ 4),(1\ 2\ 5)\}$. This gives a list equivalent to the cycle decomposition of permutations. 
\end{exa}

\begin{exa}
    Let $E_e$ and $E_o$ be the species of sets with even and odd cardinality, respectively. 
    Then $E_e(x) = \cosh(x)$ and $E_o(x) = \sinh(x)$. We have $E_e^2 = 1 + E_o^2$, 
    an equation that is easily seen not to hold in the context of $\mathbb{B}$-species. 
    Indeed, if $\ell = (U,\leq)$ is a linearly ordered set with $U \neq \emptyset$, then given a partition of $U$ into two sets of even cardinality, remove the element $\min\ell$ from the part where it occurs and add it to the complementary part to obtain a partition of $U$ into two sets of odd cardinality; this clearly defines a bijection from $E_e^2$ to $E_o^2$. 
    Passing to generating functions, we recover the well-known identity 
    $$\cosh^2(x) - \sinh^2(x) = 1.$$
\end{exa}

\begin{exa}
    The linear species $\rooted^\uparrow$ of increasing rooted trees is defined as follows: an $\rooted^\uparrow$-structure on a linearly ordered set $\ell = (U,\leq)$ is a rooted tree on $U$ such that the vertices on any path from the root towards a leaf are in increasing order. More generally, denote by $\rooted_R^\uparrow$ the species of $R$-enriched increasing rooted trees.
\end{exa}

Before moving on to the study of differential equations, let us extend the concept of linear species to the multisort setting.

\begin{defi}
    A $k$-(multi)sort linear species is a functor from the category $\mathbb{L}^k$ of linearly ordered $k$-(multi)sets and increasing bijective multifunctions to the category $\B$ of finite sets and bijections. We follow the convention that every element of sort $i$ is smaller than all elements of sort $i+1$; equivalently, one can say that a linearly ordered $k$-set $(\ell_1,\ldots,\ell_k)$ gives rise to a linearly ordered set $\ell_1 \oplus \ldots \oplus \ell_k$.
\end{defi}

This definition allows us to extend naturally the operations from $\mathbb{L}$-species to the multisort context. For example, the partial derivatives of a 2-sort species $F$ are given by
\begin{equation*}
    \frac{\del}{\del X}F[\ell,m] = F[1 \oplus \ell,m] \quad\quad\quad \frac{\del}{\del Y}F[\ell,m] = F[\ell,1 \oplus m].
\end{equation*}
The (exponential) generating series is also derived in the natural manner, with
\begin{equation*}
    F(x,y) = \sum_{k,n \in \N} F[k,n] \frac{x^k}{k!} \frac{y^n}{n!},
\end{equation*}
where $[k,n]$ denotes the pair of ordered sets $([k],[n])$.

\section{Differential Equations on Linear Species}

In this section, we will present a method for solving differential equations of the type ${Y' = R(Y)}$, where $R$ is a linear species. Let us start with an example.

\begin{exa}
    Consider the equation $Y' = E(Y)$, with initial condition ${Y(0) = 0}$. It is easy to see that $\rooted^\uparrow$ is a solution to this differential equation. Indeed, removing the minimal element (\emph{i.e.} the root) of an $\rooted^\uparrow$-structure results in an assembly of $\rooted^\uparrow$-structure.  This is analogous to Example \ref{dtf}. Now, the analytic solution to $y' = e^y$ is $y(x) = -\log(1-x)$, which gives us $\rooted^\uparrow(x) = -\log(1-x)$, but also $L = E(\rooted^\uparrow)$.
\end{exa}

With this example in mind, we can move on to solve more general differential equations. Consider the differential equation 
$Y' = R(Y)$, with initial condition $Y(0) = Z$. We can see $Z$ as a variable representing a second sort of elements.
Hence, the solution to our differential equation is a 2-sort $\mathbb{L}$-species $Y=Y(X, Z).$
Integrating, we get an equivalent (integral) equation 
\begin{equation*}
    Y(X,Z) = Z + \int_0^X R(Y(T,Z)) dT,
\end{equation*}
which can be visualised as follows:
\begin{center}

\tikzset{every picture/.style={line width=0.75pt}} %set default line width to 0.75pt        


\end{center}
Iterating this decomposition we obtain, for any pair $(\ell,m)$ of linearly ordered sets, $R$-enriched increasing rooted trees 
like the one pictured below. The vertices of sort $X$ (black points) are called fertile vertices and they all carry  
$R$-structure (on the set of their descendants); the elements of sort $Z$ (white points) always occur as leaves of the tree and they do not carry $R$-structure. 
We denote the $2$-sort $\mathbb{L}$-species of $R$-enriched increasing rooted trees by $\rooted_R^\uparrow(X,Z)$. 
\begin{center}

\tikzset{every picture/.style={line width=0.75pt}} %set default line width to 0.75pt        


\end{center}
It is clear that $\rooted_R^\uparrow(0,Z) = Z$ and that the decomposition of any structure $s \in \rooted_R^\uparrow[\ell,m]$, with $\ell \neq \emptyset$, into a minimal element and an $R$-assembly of $\rooted_R^\uparrow$-structures gives an isomorphism
\begin{equation*}
    \alpha : \frac{\del}{\del X}\rooted_R^\uparrow(X,Z) \to R(\rooted_R^\uparrow(X,Z)).
\end{equation*}
Therefore, the pair $(\rooted_R^\uparrow,\alpha)$ constitutes a solution to the equation $Y' = R(Y)$ with $Y(0)=Z$.

Now, suppose that a different pair $(\mathcal{B},\beta)$ is also a solution to the equation $Y' = R(Y)$ with $Y(0)=Z$. We will prove that there is a unique isomorphism $\Phi: \rooted_R^\uparrow \to \mathcal{B}$ such that the following diagram commutes
(where $\frac{\del}{\del X}\Phi$ and $R(\Phi)$ are defined in an obvious manner):
%, where $\frac{\del}{\del \Phi}$ and $R(\Phi)$ are built in the obvious manner.

\begin{equation}\label{diag}
\begin{tikzcd}
\frac{\del}{\del X} \rooted_R^\uparrow \arrow{r}{\alpha} \arrow{d}[left]{\frac{\del}{\del X}\Phi} & R(\rooted_R^\uparrow) \arrow{d}{R(\Phi)} \\
\frac{\del}{\del X} \mathcal{B} \arrow{r}[below]{\beta}                                              & R(\mathcal{B})                  
\end{tikzcd}
\end{equation}
Since $\rooted_R^\uparrow(0,Z) = \mathcal{B}(0,Z) = Z$, we must define $\Phi_{0,[n]}$ as the identity $Z[n] \to Z[n]$ for all $n \in \N$. Suppose, for $n \geq 0$, that the natural bijection $\Phi_{h,s}$ has already been defined in a unique manner for any pair of linearly ordered sets $(h,s)$, with $|h| \leq n$. Let $(\ell,m)$ be a pair of linearly ordered sets with $|\ell| = n+1$. Without loss of generality, assume $\ell = 1 \oplus \ell^-$, where $\ell^- = \ell \setminus \{\min\ell\}$. We then have bijections
\begin{equation*}
    \rooted_R^\uparrow[\ell,m] = \rooted_R^\uparrow[1 \oplus \ell^-,m] = \frac{\del}{\del X} \rooted_R^\uparrow[\ell^-,m] \overset{\alpha_{\ell^-,m}}{\longrightarrow} R(\rooted_R^\uparrow)[\ell^-,m]
\end{equation*}
and
\begin{equation*}
    \mathcal{B}[\ell,m] = \mathcal{B}[1 \oplus \ell^-,m] = \frac{\del}{\del X} \mathcal{B}[\ell^-,m] \overset{\beta_{\ell^-,m}}{\longrightarrow} R(\mathcal{B})[\ell^-,m].
\end{equation*}
An $R(\rooted_R^\uparrow)$-structure on $(\ell^-, m)$ is an $R$-assembly of $\rooted_R^\uparrow$-structures on pairs of linearly ordered sets $(h,s)$ such that $|h| \leq n$. Now, the induction hypothesis guarantees the existence of bijections 
$\Phi_{h,s} : \rooted_R^\uparrow[h,s] \to \mathcal{B}[h,s]$ on each member of the $R$-assembly, and thus we obtain an induced  
bijection $R(\Phi)_{\ell^-,m} : R(\rooted_R^\uparrow)[\ell^-,m] \to R(\mathcal{B})[\ell^-,m]$. We then set
\begin{equation*}
    \Phi_{\ell,m} = \frac{\del}{\del X} \Phi_{\ell^-,m} = \beta_{\ell^-,m}^{\langle -1 \rangle} \circ R(\Phi)_{\ell^-,m} \circ \alpha_{\ell^-,m}
\end{equation*}
and we obtain a bijection from $\rooted_R^\uparrow[\ell,m]$ to $\mathcal{B}[\ell,m]$, and this is the only choice that makes Diagram \ref{diag} commute.

Finally, we conclude with the following result.

\begin{theorem}
    Let $R$ be a linear species, and consider the differential equation $Y' = R(Y)$ with $Y(0)=Z$. Then $\rooted_R^\uparrow(X,Z)$ is the unique solution to that equation, in the sense that, if $\mathcal{B}(X,Z)$ is another solution, then there exists a unique isomorphism $\Phi : \rooted_R^\uparrow \to \mathcal{B}$ that makes Diagram \ref{diag} commute.
\end{theorem}

So far, we have introduced the species $\rooted_R^\uparrow(X,Z)$, as well as seen that they are solutions to a specific kind of differential equations. In the following theorem, we present a different way of describing $\rooted_R^\uparrow(X,Z)$. For the reader to fully understand the notation used in the proof, we recommend reading \cite[Section 2.6]{bergeron2}.

\begin{theorem}
    Let $R$ be a linear species. Then the species $\rooted_R^\uparrow(X,Z)$ can be expressed as $\rooted_R^\uparrow(X,Z) = \exp(XR(Z)\frac{\del}{\del Z})Z$, or more explicitly,
    \begin{equation*}
        \rooted_R^\uparrow(X,Z) = \sum_{n \geq 0} \mathcal{D}^n(Z)E_n(X)
    \end{equation*}
    where $\mathcal{D} = R(Z)\frac{\del}{\del Z}$.
\end{theorem}
\begin{proof}
    Let the elements of sort $X$ be coloured black and the elements of sort $Z$ be coloured white, and let $\Psi(X,Z)$ be a 2-sort linear species. An $X\mathcal{D}\Psi$-structure can be seen as a $\Psi$-structure that has undergone an eclosion, that is, its minimal white point has been replaced by an $XR(Z)$-structure, that is, by a black point (called the eclosion point) followed by an $R$-assembly of white points. The following picture will make the concept of eclosion clearer.
    \begin{center}

\tikzset{every picture/.style={line width=0.75pt}} %set default line width to 0.75pt        


    \end{center}
    For $n \geq 0$, an $(X\mathcal{D})^n\Psi$-structure is built from a $\Psi$-structure by performing the eclosion $n$ times. The order of eclosions is relevant, and should always be noted. Naturally, there are $n!$ ways to label the $n$ eclosion points of an $(X\mathcal{D})^n\Psi$-structure. We call a labelling orderly when it coincides with the order of eclosions. With that, we can interpret an $\frac{(X\mathcal{D})^n}{n!}\Psi$-structure as an $(X\mathcal{D})^n\Psi$-structure with orderly labelling. We can also define $\exp(X\mathcal{D})$ as simply $$\exp(X\mathcal{D}) := \sum_{n \geq 0} \frac{(X\mathcal{D})^n}{n!}$$
    from where an $\exp(X\mathcal{D})\Psi$-structure is an $\frac{(X\mathcal{D})^n}{n!}\Psi$-structure for some $n \in \N$.

    Now, take $\Psi(X,Z) = Z$, that is, a $\Psi$-structure being simply a white point, and a $\Psi$-structure after one eclosion being an $XR(Z)$-structure. After multiple eclosions, we have an $R$-enriched rooted tree with black fertile vertices and white leaves. Then all of the black points are eclosion points, and in the case of an orderly labelling, \emph{id est}, an $\exp(X\mathcal{D})\Psi$-structure, the labels are in increasing order from the root to the leaves, since naturally that is the direction that the eclosions take, and we conclude that ${\rooted_R^\uparrow(X,Z) = \exp(XR(Z)\frac{\del}{\del Z})Z}$.

    Since the operator $\mathcal{D}$ commutes with multiplication by $X$, and $E_n(X) = \frac{X^n}{n!}$, we conclude that $\rooted_R^\uparrow(X,Z) = \sum_{n \geq 0} \mathcal{D}^n(Z)E_n(X)$. This can also be deduced from the fact that the right-hand side of the equation tells us to perform eclosions considering white points only, that is, perform $\mathcal{D}^n(Z)$, and then add the black points in the order that the eclosions happened.
\end{proof}

\chapter{General Differential Operators}
\label{ch7}

In this chapter, we return to the study of ($k$-sort) $\B$-species. More precisely, we will present a combinatorial interpretation of $\Omega(X,D)$, where $\Omega$ is a 2-sort species and $D = \frac{d}{dX}$ is the differential operator. We then extend this definition to the case where $\Omega$ is a 1-sort species, and use it to connect the first chapter to the concept of species by providing a way to interpret the finite difference operator in the broader sense of species of structures.

\section{Definition and Operations}

\begin{defi}
    Let $\Omega(X, T)$ be a two sort species. To every finite set $A$ (of sort $X$), we associate a $1$-sort species
$\Omega_A$ defined as follows:
\begin{itemize}
\item $\Omega_A[U]=\Omega[(A, U)]$ for every finite set $U$;
\item for a given bijection $\tau:U \to V$ between finite sets, we set 
$\Omega_A[\tau]=\Omega[(id_A, \tau)]$. 
\end{itemize}
\end{defi}

Note that each bijection $\sigma: A \to B$ between finite sets (of sort $X$) gives rise to a natural isomorphism 
$\overline{\sigma}=\{\overline{\sigma}_U:=\Omega[(\sigma, id_U)]\}_{U \in \mathbf{Ob}(\mathbb{B})}$ from $\Omega_A$ to $\Omega_B$,
\emph{i.e.}, all diagrams of the following form commute: 

\begin{equation*}
    \begin{tikzcd}
\Omega_A[U] \arrow{r}{\overline{\sigma}_U} \arrow{d}[left]{\Omega_A[\tau]} & \Omega_B[U] \arrow{d}{\Omega_B[\tau]} \\
\Omega_A[V] \arrow{r}[below]{\overline{\sigma}_V}                                              & \Omega_B[V]                  
\end{tikzcd}
\end{equation*}

Consequently, $s_1 \in \Omega_A[U]$ and $s_2 \in \Omega_A[V]$ are isomorphic structures of species $\Omega_A$ if and only if 
$\overline{\sigma}_U[s_1]$ and $\overline{\sigma}_V[s_2]$ are isomorphic structures of species $\Omega_B$. In other words,
the natural isomorphism $\overline{\sigma}$ induces a bijection from the (possibly infinite) set of isomorphism types of 
structures of species $\Omega_A$ to the set of isomorphism types of structures of species $\Omega_B$. 

\begin{defi}
    We say that a $2$-sort species $\Omega(X, T)$ is finitary in $T$ if for every finite set $U$ of elements of sort $X$ there are no $\Omega$-structures on the pair of sets $(U,V)$ for every sufficiently large set $V$.
If $\Omega(X, T)$ is finitary in $T$, we may define a new $1$-sort species $\Omega(X, 1)$ (also denoted by 
$[\Omega(X, T)]|_{T:=1}$) by the following rules:    
\begin{itemize}
    \item the $\Omega(X, 1)$-structures on a finite set $A$ are the isomorphism types of structures of species $\Omega_A$;
    \item the transport along a bijection $\sigma: A \to B$ ($A, B$ finite sets) is induced by the natural isomorphism 
    $\overline{\sigma}$.
\end{itemize}
\end{defi}
 
Behind this formal definition of the species $\Omega(X, 1)$, hides a very simple idea, namely, an $\Omega(X, 1)$-structure on a finite set $A$ is nothing else but an $\Omega$-structure on the set $A$ of (labelled) elements of sort $X$ and an arbitrary set of unlabelled elements of sort $T$. A justification for the choice of notation is provided by the following formula expressing the \emph{Zyklenzeiger} of $\Omega(X, 1)$ in terms of the \emph{Zyklenzeiger} of $\Omega$:
\[Z_{\Omega(X, 1)}(x_1, x_2, \ldots)=[Z_{\Omega}(x_1, x_2, \ldots; t_1, t_2, \ldots)]|_{t_j=1}=Z_{\Omega}(x_1, x_2, \ldots; 1, 1, \ldots).\]

Before defining general differential operators, we first provide some conventions for the graphical representations. We will follow that used by Labelle and Lamathe in \cite{labelle}.

\begin{notn}
$\ $

    \begin{enumerate}
        \item For 1-sort species $F(X)$, we will maintain the same graphical convention used in the previous chapters, with elements of sort $X$ represented by black circles.
        \begin{center}

\tikzset{every picture/.style={line width=0.75pt}} %set default line width to 0.75pt        


        \end{center}
        \item For 2-sort species $\Omega(X,T)$, we will represent elements of sort $X$ as black circles and elements of sort $T$ as black squares.
        \begin{center}

\tikzset{every picture/.style={line width=0.75pt}} %set default line width to 0.75pt        

%
        \end{center}
        \item For the species $\Omega(X,1)$, obtained by setting $T := 1$, where elements of sort $T$ are unlabelled, we will represent these unlabelled elements as white squares. We note again that this substitution is only possible when $\Omega$ is finitary in $T$.
        \begin{center}

\tikzset{every picture/.style={line width=0.75pt}} %set default line width to 0.75pt        

%
        \end{center}
    \end{enumerate}
\end{notn}

We also need to define yet another operation on 2-sort species: the partial Cartesian product.

\begin{defi}
    Let $\Phi(X,T), \Omega(X,T)$ be two 2-sort species. Then the partial Cartesian product of $\Phi$ and $\Omega$ with respect to 
    the sort $T$, denoted by $\Phi(X,T) \times_T \Omega(X,T)$, is defined as follows: a $\Phi \times_T \Omega$-structure $s$ on a pair of sets $(U,V)$ is a pair $s = (f,w)$, where $f$ is a $\Phi$-structure on $(U_1,V)$, $w$ is an $\Omega$-structure on $(U_2,V)$, $U_1 \cup U_2 = U$ and $U_1 \cap U_2 = \emptyset$. It is easy to see that the \emph{Zyklenzeiger} of $\Phi \times_T \Omega$ can be described as follows:   
    \begin{align*}
        Z_{\Phi \times_T \Omega}(x_1,x_2,\ldots;t_1,t_2, \ldots) = \sum_{n_1,n_2, \ldots} \omega_{\Phi,n_1,n_2, \ldots}(x_1,x_2, \ldots) \omega_{\Omega,n_1,n_2,\ldots} (x_1,x_2,\ldots) \frac{t_1^{n_1}t_2^{n_2}\ldots}{1^{n_1}n_1!2^{n_2}n_2!\ldots},
    \end{align*}
    where the coefficients $\omega_{\Phi,n_1,n_2,\ldots}$ and $\omega_{\Omega,n_1,n_2,\ldots}$ are given by
    \begin{equation*}
        Z_{\Phi} (x_1,x_2, \ldots;t_1,t_2,\ldots) = \sum_{n_1,n_2,\ldots} \omega_{\Phi,n_1,n_2,\ldots}(x_1,x_2, \ldots) \frac{t_1^{n_1}t_2^{n_2} \ldots}{1^{n_1}n_1!2^{n_2}n_2! \ldots}
    \end{equation*}
    \begin{equation*}
        Z_{\Omega} (x_1,x_2,\ldots;t_1,t_2,\ldots) = \sum_{n_1,n_2,\ldots} \omega_{\Omega,n_1,n_2,\ldots}(x_1,x_2,\ldots) \frac{t_1^{n_1}t_2^{n_2} \ldots}{1^{n_1}n_1!2^{n_2}n_2! \ldots}
    \end{equation*}
    The right-hand side of this equation expressing $Z_{\Phi \times_T \Omega}$ is denoted by $Z_\Phi \times_\mathbf{t} Z_\Omega$ and is called the Hadamard product of $Z_\Phi$ and $Z_\Omega$ with respect to $\mathbf{t} = (t_1,t_2, \ldots)$.

    The following picture offers a graphical representation of a partial Cartesian product.
    \begin{center}

\tikzset{every picture/.style={line width=0.75pt}} %set default line width to 0.75pt        


    \end{center}
\end{defi}

We now have the required tools to define general differential operators.

\begin{defi}
    Let $F(X)$ be a 1-sort species and $\Omega(X,T)$ be a 2-sort species. If $F$ is of finite degree in $X$ or $\Omega$ is finitary in $T$, then $\Omega(X,D)F(X)$ is defined as
    \begin{equation*}
        \Omega(X,D)F(X) := [\Omega(X,T) \times_T F(X+T)]|_{T := 1}
    \end{equation*}
    The following picture shows an $\Omega(X,D)F(X)$-structure.
    \begin{center}

\tikzset{every picture/.style={line width=0.75pt}} %set default line width to 0.75pt        

%
    \end{center}
    Note that the limitations on the hypothesis are needed. For example, the picture below shows an $\rooted(X,D)\cyc_{16}(X)$-structure, where $\rooted(X,T)$ is the species of rooted trees with internal vertices of sort $X$ and leaves of sort $T$, and $\cyc_{16}$ is the species of cycles of length $16$. Note that $\rooted(X,D)\cyc(X)$ would not be a well-defined species, as there could be infinitely many such structures on any finite set of elements of sort $X$. From now on, we will tacitly assume that the hypothesis of the definition is satisfied.
    \begin{center}

\tikzset{every picture/.style={line width=0.75pt}} %set default line width to 0.75pt        


    \end{center}
\end{defi}

Let us now see how this general differentiation influences the power series.

\begin{theorem}\label{714}
    Let $F(X)$ be a 1-sort species, and $\Omega(X,T)$ be a 2-sort species. Denote $G(X) = \Omega(X,D)F(X)$. Then
    \begin{itemize}
        \item[ ] $G(x) = \sum_{n_1,n_2, \ldots} \omega_{n_1,n_2, \ldots}(x,0,0, \ldots) \left[ \frac{\left(\frac{\del}{\del x_1}\right)^{n_1}\left(\frac{\del}{\del x_2}\right)^{n_2} \ldots}{n_1!n_2! \ldots}Z_F \right](x,0,0, \ldots)$
        \item[ ] $\widetilde{G}(x) = \sum_{n_1,n_2, \ldots} \omega_{n_1,n_2, \ldots}(x,x^2,x^3,\ldots) \left[ \frac{\left(\frac{\del}{\del x_1}\right)^{n_1}\left(\frac{\del}{\del x_2}\right)^{n_2} \ldots}{n_1!n_2! \ldots}Z_F \right](x,x^2,x^3, \ldots)$
        \item[ ] $Z_G(x_1,x_2,x_3, \ldots) = Z_\Omega(x_1,x_2,x_3, \ldots;\frac{\del}{\del x_1},2\frac{\del}{\del x_2},3\frac{\del}{\del x_3}, \ldots)Z_F(x_1,x_2,x_3, \ldots)$
    \end{itemize}
    where $\omega_{n_1,n_2, \ldots}$ comes from
    \begin{equation*}
        Z_\Omega(x_1,x_2,x_3, \ldots;t_1,t_2,t_3,\ldots) = \sum_{n_1,n_2,\ldots} \omega_{n_1,n_2,\ldots}(x_1,x_2,x_3,\ldots) \frac{t_1^{n_1}t_2^{n_2}t_3^{n_3}\ldots}{1^{n_1}n_1!2^{n_2}n_2!3^{n_3}n_3!\ldots}
    \end{equation*}
\end{theorem}
\begin{proof}
    It is sufficient to prove the equation for $Z_G$. We have
    \begin{align*}
        &Z_G(x_1,x_2, \ldots) = [Z_\Omega(x_1,x_2, \ldots;t_1,t_2, \ldots) \times_\mathbf{t} Z_F(x_1+t_1,x_2+t_2,\ldots)]|_{t_j = 1}\\
        &= \left[ Z_\Omega(x_1,x_2,\ldots;t_1,t_2,\ldots) \times_\mathbf{t} \sum_{n_1,n_2,\ldots} f_{n_1,n_2,\ldots} \frac{(x_1+t_1)^{n_1}(x_2+t_2)^{n_2}\ldots}{1^{n_1}n_1!2^{n_2}n_2!\ldots} \right]_{t_j = 1}\\
        &= \left[ Z_\Omega \times_\mathbf{t} \sum_{n_1,n_2,\ldots} f_{n_1,n_2,\ldots} \frac{\left(\sum_{i_1=0}^{n_1}\binom{n_1}{i_1}x_1^{n_1-i_1}t_1^{i_1} \right)\left(\sum_{i_2=0}^{n_2}\binom{n_2}{i_2}x_2^{n_2-i_2}t_2^{i_2} \right)\ldots}{1^{n_1}n_1!2^{n_2}n_2!\ldots} \right]_{t_j = 1}\\
        &= \left[ Z_\Omega \times_\mathbf{t} \sum_{n_1,n_2,\ldots} f_{n_1,n_2,\ldots} \left(\sum_{i_1=0}^{n_1}\frac{x_1^{n_1-i_1}}{1^{n_1}(n_1-i_1)!}\frac{t_1^{i_1}}{i_1!} \right)\left(\sum_{i_2=0}^{n_2}\frac{x_2^{n_2-i_2}}{2^{n_2}(n_2-i_2)!}\frac{t_2^{i_2}}{i_2!} \right)\ldots \right]_{t_j = 1}\\
        &= \left[ Z_\Omega \times_\mathbf{t} \sum_{i_1,i_2,\ldots} 
        \frac{t_1^{i_1}t_2^{i_2}\ldots}{i_1!i_2!\ldots} \left( \frac{\del}{\del x_1} \right)^{i_1}\left( \frac{\del}{\del x_2} \right)^{i_2}\ldots \left(\sum_{n_1,n_2,\ldots} f_{n_1,n_2,\ldots} \frac{x_1^{n_1}x_2^{n_2}\ldots}{1^{n_1}n_1!2^{n_2}n_2!\ldots}\right) \right]_{t_j = 1}\\
        &= \left[ Z_\Omega \times_\mathbf{t} \sum_{i_1,i_2,\ldots} \left( 1\frac{\del}{\del x_1} \right)^{i_1}\left( 2\frac{\del}{\del x_2} \right)^{i_2}\ldots Z_F(x_1,x_2,\ldots) 
        \frac{t_1^{i_1}t_2^{i_2}\ldots}{1^{i_1}i_1!2^{i_2}i_2!\ldots}  \right]_{t_j = 1}\\
        &= \left[ \sum_{i_1,i_2,\ldots} \omega_{i_1,i_2,\ldots}(x_1,x_2,\ldots)\left( 1\frac{\del}{\del x_1} \right)^{i_1}\left( 2\frac{\del}{\del x_2} \right)^{i_2} \ldots Z_F(x_1,x_2,\ldots) 
        \frac{t_1^{i_1}t_2^{i_2}\ldots}{1^{i_1}i_1!2^{i_2}i_2!\ldots} \right]_{t_j=1}\\
        &= \sum_{i_1,i_2, \ldots} \omega_{i_1,i_2,\ldots}(x_1,x_2,\ldots)\frac{\left( 1\frac{\del}{\del x_1} \right)^{i_1}\left( 2\frac{\del}{\del x_2} \right)^{i_2}\ldots}{1^{i_1}i_1!2^{i_2}i_2!\ldots}Z_F(x_1,x_2,\ldots)\\
        &= Z_\Omega(x_1,x_2,\ldots;1\frac{\del}{\del x_1},2\frac{\del}{\del x_2},\ldots)Z_F(x_1,x_2,\ldots).
    \end{align*}
\end{proof}

We can now proceed to the second operation: composition of differential operators.

\begin{defi}
    Let $\Omega_1(X,T), \Omega_2(X,T)$ be two 2-sort species. Then the composition of $\Omega_1$ and $\Omega_2$ is defined as
    \begin{equation*}
        \Omega_3(X,T) = \Omega_2(X,T) \odot \Omega_1(X,T) := [\Omega_2(X,T+Y) \times_Y \Omega_1(X+Y,T)]|_{Y:=1}
    \end{equation*}
    The differential operator $\Omega_3(X,D)$, denoted ${\Omega_2(X,D) \odot \Omega_1(X,D)}$, is called the composition of $\Omega_1(X,D)$ and $\Omega_2(X,D)$. The following picture presents an ${[\Omega_2(X,T) \odot \Omega_1(X,T)]}$-structure, where black circles represent $X$-sort elements, black squares represent $T$-sort elements, and white triangles represent unlabelled $Y$-sort elements.
    \begin{center}

\tikzset{every picture/.style={line width=0.75pt}} %set default line width to 0.75pt        


    \end{center}
\end{defi}

The next theorem explains why this operation is called composition.

\begin{theorem}
    Let $F(X)$ be a 1-sort species. Then
    \begin{equation}\label{716}
        [\Omega_2(X,D) \odot \Omega_1(X,D)]F(X) = \Omega_2(X,D)[\Omega_1(X,D)F(X)]
    \end{equation}
\end{theorem}
\begin{proof}
    Firstly, write $\Omega_2(X,D)[\Omega_1(X,D)F(X)]$ as
    \begin{equation*}
        \Omega_2(X,D)[\Omega_1(X,T_1) \times_{T_1} F(X+T_1)]_{T_1 := 1} 
    \end{equation*}
    \begin{equation*}
        = \Big\{ \Omega_2(X,T_2) \times_{T_2} \big[ [\Omega_1(X,T_1) \times_{T_1} F(X+T_1)]_{T_1 := 1} \big]_{X := X+T_2} \Big\}_{T_2 :=1}
    \end{equation*}
    Denote $T_1$-sort elements by squares, and $T_2$-sort elements by triangles. This yields the following diagram for a generic ${\Omega_2(X,D)[\Omega_1(X,D)F(X)]}$-structure.
    \begin{center}

\tikzset{every picture/.style={line width=0.75pt}} %set default line width to 0.75pt        


    \end{center}
    Now compare this to the diagram below, which represents an $[\Omega_2(X,D) \odot \Omega_1(X,D)]F(X)$-structure. It should then become clear that equation \ref{716} holds.
    \begin{center}

\tikzset{every picture/.style={line width=0.75pt}} %set default line width to 0.75pt        

%
    \end{center}
\end{proof}

Composition of 2-sort species has the following effect on the \emph{Zyklenzeiger}.

\begin{theorem}
    Let $\Omega_1, \Omega_2$ be two 2-sort species. Then
    \begin{equation}\label{717}
        Z_{\Omega_2 \odot \Omega_1} = \sum_{n_1,n_2,\ldots} \frac{\left[ \left(\frac{\del}{\del x_1}\right)^{n_1}\left(2\frac{\del}{\del x_2}\right)^{n_2}\ldots Z_{\Omega_1} \right]\left[ \left(\frac{\del}{\del t_1}\right)^{n_1}\left(2\frac{\del}{\del t_2}\right)^{n_2}\ldots Z_{\Omega_2} \right]}{1^{n_1}n_1!2^{n_2}n_2!\ldots }
    \end{equation}
\end{theorem}
\begin{proof}
    In a way analogous to that used in the proof of Theorem \ref{714} for $Z_F(x_1+t_1,x_2+t_2,...)$, we have
    \begin{equation*}
        Z_{\Omega_1(X+Y,T)} = \sum_{n_1,n_2, \ldots} \left(\frac{\del}{\del x_1}\right)^{n_1}\left(2\frac{\del}{\del x_2}\right)^{n_2}\ldots Z_{\Omega_1(X,T)} \frac{y_1^{n_1}y_2^{n_2} \ldots }{1^{n_1}n_1!2^{n_2}n_2! \ldots }
    \end{equation*}
    and
    \begin{equation*}
        Z_{\Omega_2(X,T+Y)} = \sum_{n_1,n_2, \ldots} \left(\frac{\del}{\del t_1}\right)^{n_1}\left(2\frac{\del}{\del t_2}\right)^{n_2} \ldots Z_{\Omega_2(X,T)} \frac{y_1^{n_1}y_2^{n_2} \ldots }{1^{n_1}n_1!2^{n_2}n_2! \ldots }
    \end{equation*}
    Equation \ref{717} is then simply a case of performing Hadamard product with respect to $Y$ and evaluating each $y_j := 1$.
\end{proof}

The next operation we will provide is Joyal's scalar product, which is defined in the following manner.

\begin{defi}
    Let $F(X), G(X)$ be two 1-sort species. Then Joyal's scalar product is the bilinear form $\langle\  F(X),G(X) \ \rangle$ defined as
    \begin{equation*}
        \langle\  F(X),G(X) \ \rangle = [ F(X) \times G(X) ] |_{X := 1} = \text{number of unlabelled $(F \times G)$-structures, if finite.} 
    \end{equation*}
\end{defi}

%This bilinear form allows us to study adjoint operators in the sense of combinatorial differential operators.

\begin{theorem}
    Let $\Omega(X,T)$ be a 2-sort species. Then
    \begin{equation*}
        \langle\  \Omega(X,D)F(X),G(X) \ \rangle = \langle\  F(X),\Omega(D,X)G(X) \ \rangle
    \end{equation*}
    In other words, the adjoint operator of $\Omega(X,D)$ is $\Omega^*(X,D) = \Omega(D,X)$.
\end{theorem}
\begin{proof}
    See the following diagram.
    \begin{center}

\tikzset{every picture/.style={line width=0.75pt}} %set default line width to 0.75pt        


    \end{center}
\end{proof}

\begin{corollary}
    Let $\Phi(X,D), \Psi(X,D)$ be two general differential operators. Then
    \begin{equation*}
        [\Phi(X,D) \odot \Psi(X,D)]^* = \Psi^*(X,D) \odot \Phi^*(X,D) = \Psi(D,X) \odot \Phi(D,X)
    \end{equation*}
\end{corollary}
\begin{proof}\begin{equation*}
    \langle\  [\Phi(X,D) \odot \Psi(X,D)]F(X),G(X) \ \rangle = \langle\  \Phi(X,D)[\Psi(X,D)F(X)],G(X) \ \rangle 
\end{equation*}
\begin{equation*}
    = \langle\  \Psi(X,D)F(X),\Phi(D,X)G(X) \ \rangle = \langle\  F(X),\Psi(D,X)[\Phi(D,X)G(X)] \ \rangle 
\end{equation*}
\begin{equation*}
    = \langle\  F(X),[\Psi(D,X)\odot\Phi(D,X)]G(X) \ \rangle
\end{equation*}
\end{proof}

\section{Hammond Differential Operators}

In this section, we will take a brief look at the definition and some examples of Hammond differential operators.

\begin{defi}
    A Hammond differential operator is a differential operator of the form $\Phi(D)$, that is, a general differential operator of the form $\Omega(X,T) = \Phi(T)$. Observe that
    \begin{equation*}
        \Phi(D)F(X) = [\Phi(T) \times_T F(X+T)]|_{T := 1} = \big\{[E(X)\Phi(T)] \times F(X+T)\big\}\big|_{T := 1}
    \end{equation*}\newpage
    The following picture shows a $\Phi(D)F(X)$-structure.
    \begin{center}

\tikzset{every picture/.style={line width=0.75pt}} %set default line width to 0.75pt        


    \end{center}
\end{defi}

The composition of Hammond operators coincides with the product operation, that is, ${\Phi(D) \odot \Psi(D) = (\Phi \cdot \Psi)(D)}$. 
Therefore, restricted to Hammond operators, composition is commutative, which is not the case for the composition of general differential operators. 
One also has that the adjoint of a Hammond differential operation $\Phi(D)$ is $\Phi(X)$, that is, 
$\langle\  \Phi(D)F(X),G(X) \ \rangle = \langle\  F(X),\Phi(X)G(X) \ \rangle$. 

\begin{exa}
    Consider the Hammond differential operator $E_2(D)$ associated with the species $E_2$ of sets of cardinality two. Then, as it can be seen from the following pictures, we have
    \begin{equation*}
        E_2(D)(F \cdot G) = (E_2(D)F) \cdot G + F' \cdot G' + F \cdot (E_2(D)G)
    \end{equation*}
    \begin{equation*}
        E_2(D)(E \circ F) = (E \circ F) \cdot (E_2(D)F + E_2(F'))
    \end{equation*}
    \begin{center}

\tikzset{every picture/.style={line width=0.75pt}} %set default line width to 0.75pt        

        
    \end{center}
\end{exa}

\begin{exa}
    Consider the Hammond differential operator $T^n(D)$. Naturally, one has the relation $T^n(D) = D^n$, from where
    \begin{equation*}
        T^n(D)F(X) = D^n F(X) = \frac{d^n F(X)}{dX^n}
    \end{equation*}
\end{exa}

\begin{exa}
    A general combinatorial differential operator $\Omega(X,D)$ is called self-adjoint when $\Omega(X,T) = \Omega(T,X)$. For example, given any 1-sort species $\Phi(X)$, the operator $\Phi(X+D)$ is self-adjoint. One important kind of self-adjoint operators is a generalization of the concept of pointing. As the reader may recall, pointing is the operation given by the operator $XD$, or also $T(XD)$. One can then generalize this idea to provide the concept of $\Phi$-pointing, that is, an application of the operator $\Phi(XD)$, where $\Phi$ is a 1-sort species. The composition of pointing operations is given by the kiss product $\Phi(XD) \odot \Psi(XD) = (\Phi \dot{\times} \Psi)(XD)$, which is defined below.
\end{exa}

\begin{defi}
    Let $F(X), G(X)$ be two 1-sort species. Then the kiss product of $F$ and $G$ is given by
    \begin{equation*}
        (F \dot{\times} G)(X) = F(X+XD)G(X) = G(X+XD)F(X).
    \end{equation*}
    (See the pictures below for a justification of the last equality.)
    The two regular products $F \cdot G$ and $F \times G$ can be seen as the "empty" and "full" kiss products, and therefore
    \begin{equation*}
        F \cdot G \subset F \dot{\times} G \quad \text{and} \quad F \times G \subset F \dot{\times} G.
    \end{equation*}
    \begin{center}

\tikzset{every picture/.style={line width=0.75pt}} %set default line width to 0.75pt        


    \end{center}
\end{defi}

Another example of pointing is the operator $T^2(XD) = (XD)^2 = X^2D^2$, which corresponds to pointing two distinct elements of the $F$-structure. Be careful, as this is not equal to ${(XD) \odot (XD)}$, as the latter one is equal to $XD + X^2D^2$, for it includes the case of pointing the same element twice.

\begin{exa}
    Consider the Hammond operator given by $E(D)$. Then, as per the definition, ${E(D)F(X) = [E(T) \times_T F(X+T)]|_{T := 1} = F(X+1)}$, which gives us an equivalent to the shift operator in the sense of species of structures (recall Definition \ref{shift}). More so, since ${F(X) \subset F(X+1)}$, this allows us to define the finite difference operator in the sense of species of structures without needing to enter the realm of virtual species.
\end{exa}

\begin{defi}
    The finite difference operator $\dif$ admits an extension to species of structures through the definition 
    $\dif F(X) := E_+(D)F(X)$, where $E_+ = E - 1$ is the species of non-empty finite sets.
\end{defi}

\begin{exa}
    Let $L_n = X^n$ be the species of linear orders restricted to sets of cardinality $n$, and 
    set $G(X)=\dif L_n$.
    Since $Z_{L_n}(x_1, x_2, \ldots)=x_1^n$ and 
    \[ Z_{E_+}(t_1,t_2,t_3, \ldots) = \left[\sum_{n_1,n_2,\ldots}  \frac{t_1^{n_1}t_2^{n_2}t_3^{n_3}\ldots}{1^{n_1}n_1!2^{n_2}n_2!3^{n_3}n_3!\ldots} \right] - 1,\]
    it follows from Theorem~\ref{714} that 
    \[Z_G(x_1,x_2,x_3, \ldots) =\left[\sum_{n_1,n_2,\ldots}  \frac{\left(\frac{\del}{\del x_1}\right)^{n_1}\left(\frac{\del}{\del x_2}\right)^{n_2} \ldots}{n_1!n_2! \ldots}  -1 \right] x_1^n =\sum_{k=0}^{n-1} \binom{n}{k}x^{k}=(x_1+1)^n - x_1^n, \]
    and thus 
    \[G(x)=(x+1)^n - x^n=\sum_{k=0}^{n-1} k!\binom{n}{k}\frac{x^{k}}{k!}.\]
 \end{exa}

The extension of the finite difference operator to the realm of species opens the possibility of a theory of general combinatorial difference calculus, that is, the study of operators of the form $\Omega(X,\dif)$, where $\Omega(X,T)$ is a 2-sort species as usual.

\nocite{streicher}
\nocite{joyal2}
\nocite{garcia}
\printbibliography

\end{document}